\documentclass[11pt]{article}
\usepackage{mathtools}
\usepackage[utf8]{inputenc}
\usepackage{amssymb,amsmath,amsfonts,eucal,mathrsfs,amsthm}
\usepackage[colorinlistoftodos]{todonotes}
\usepackage{xcolor}
\usepackage{enumitem}
\usepackage{accents}
\newtheorem{theorem}{Theorem}
\newtheorem{proposition}[theorem]{Proposition}
\newtheorem{lemma}[theorem]{Lemma}

\newtheorem{corollary}[theorem]{Corollary}

\newtheorem{example}[theorem]{Example}

\theoremstyle{definition}
\newcommand{\R}{\mathbb{R}}
\newcommand{\Z}{\mathbb{Z}}
\newcommand{\Q}{\mathbb{Q}}
\newcommand{\Sf}{\mathbb{S}}
\newcommand{\C}{\mathbb{C}}

\newcommand{\spa}{\mbox{span}}

\newcommand{\Ric}{\mbox{Ric}}

\newcommand{\trace}{\mbox{tr\,}}

\def\<{{\langle}}
\def\>{{\rangle}}
\def\B{\mathcal{B}}
\def\RP{\mathord{\mathbb R}\mathord{P}}
\def\CP{\mathord{\mathbb C}\mathord{P}}
\def\HP{\mathord{\mathbb H}\mathord{P}}
\def\OP{\mathord{\mathbb O}\mathord{P}}

\def\n{\nabla}

\def\a{\alpha}

\def\be{\begin{equation}}
\def\ee{\end{equation}}

\begin{document}

\title{Geometric and topological rigidity of 
pinched submanifolds in Riemannian manifolds}
\author{Theodoros Vlachos}
\date{}
\maketitle

\renewcommand{\thefootnote}{\fnsymbol{footnote}} 
\renewcommand{\thefootnote}{\arabic{footnote}} 

\renewcommand{\thefootnote}{\fnsymbol{footnote}} 
\renewcommand{\thefootnote}{\arabic{footnote}}

\begin{abstract}
We study the rigidity of compact submanifolds of arbitrary codimension in Riemannian manifolds satisfying a pinching condition involving the norm of the second fundamental form and the mean curvature. In contrast to the corresponding earlier results, the pinching condition is not required to hold strictly. Allowing equality gives rise to richer and subtler geometric and topological phenomena. By analyzing the nonstrict boundary case in arbitrary ambient manifolds, we obtain a classification of the possible intrinsic geometries and topologies of the submanifold, together with restrictions on its extrinsic geometry. We also provide several families of examples satisfying the pinching condition, showing that the conclusions cannot, in general, be strengthened to extrinsic rigidity statements. The resulting theorems extend several known results in the literature, particularly sphere theorems, to the nonstrict setting while removing several additional geometric or topological assumptions appearing in earlier results.

\end{abstract}

\tableofcontents

\section{Introduction}

A fundamental problem in differential geometry is to understand how 
curvature restrictions influence the global geometry and topology of a 
Riemannian manifold. In submanifold theory, this problem naturally leads 
to pinching conditions involving the second fundamental form and the mean 
curvature, and to the question of how such conditions constrain both the 
intrinsic geometry of the submanifold and the geometry of its immersion.

The study of this type of rigidity has a long history. For minimal 
submanifolds of spheres with sufficiently pinched second fundamental form, 
Simons \cite{Sim} established a fundamental result, followed by the rigidity 
theorem of Chern, do Carmo and Kobayashi \cite{CdCK}. These works initiated 
a large body of research on geometric and topological rigidity under 
curvature pinching conditions, particularly for submanifolds of space 
forms; see, for instance, %\cite{adc,AC,HW,Le1,Sim,XG,XZ,XG} 
\cite{adc, AB, font, HW, a, OV, p, SX, v3, v4, Xu, XGu, XG} and the references therein.

In the present paper, we study compact submanifolds
$f\colon M^n\longrightarrow \accentset{\sim}{M}^{n+m}$ 
of arbitrary codimension in Riemannian manifolds satisfying the 
pointwise pinching condition
\begin{equation}\label{1}
S\leq
\frac{16}{3}
\left(
\accentset{\sim}{K}^{\,f}_{\min}
-\frac14\accentset{\sim}{K}^{\,f}_{\max}
\right)
+\frac{n^2H^2}{n-2}.
\tag{$*$}
\end{equation}
Here $S$ denotes the \emph{squared norm} of the second fundamental form 
$\alpha_f\colon TM\times TM\to N_fM$, which takes values in the normal 
bundle $N_fM$, and $H=\|\mathcal H_f\|$ is the \emph{mean curvature 
function}, where the \emph{mean curvature vector field} is given by
$$
\mathcal H_f=\frac{1}{n}\,\mathrm{tr}\,\alpha_f.
$$
Moreover,
$$
\accentset{\sim}{K}^{\,f}_{\max}(x)
=
\max\left\{
\accentset{\sim}{K}(f_{*x}\sigma):
\sigma\subset T_xM\ \text{is a two-plane}
\right\},
$$
and
$$
\accentset{\sim}{K}^{\,f}_{\min}(x)
=
\min\left\{
\accentset{\sim}{K}(f_{*x}\sigma):
\sigma\subset T_xM\ \text{is a two-plane}
\right\},
$$
where $\accentset{\sim}{K}$ denotes the sectional curvature of the ambient 
manifold. Notice that \eqref{1} is invariant under scaling of the ambient 
metric.

The pinching condition \eqref{1} was introduced by Xu and Zhao 
\cite{XZ} in a slightly stronger form, mainly under the assumption that 
the inequality is strict at every point. Later, Xu and Gu \cite{XG1} 
studied Einstein submanifolds under the same condition, assuming that 
strict inequality holds at some point. More recently, Li, Xu and Zhao 
\cite{LXZ} considered the strict form of \eqref{1} under additional 
hypotheses. Related results were obtained in \cite{cuisun,XG}. 
For submanifolds of the space form $\Q_c^{n+m}$ of curvature 
$c\geq0$, Howard and Wei \cite{HW} studied the strict pinching 
condition
$$
S<2kc+\frac{n^2H^2}{n-k},
$$ 
where $k$ is an integer satisfying $1\leq k\leq n/2$. 
For $k=2$, this is precisely the strict version of \eqref{1}. 
Pinching conditions of a similar nature have also been 
investigated in complex and quaternionic projective 
spaces and in K\"ahler ambient manifolds; see 
\cite{ps,ss1,ss2}.

The principal aim of this paper is to understand what happens at the
\emph{boundary} of the pinching condition. Thus, in contrast to the
corresponding previous results, we do not require \eqref{1} to hold
strictly. This seemingly small change leads to substantially richer
geometric phenomena. In the strict case, the pinching condition forces
positive isotropic curvature and hence sphere-type conclusions. Once
equality is allowed, additional geometries appear, and the problem 
becomes one of classifying the possible geometries arising in the 
borderline case.

The main results give such a classification. In dimensions $n\geq5$, 
condition \eqref{1} leads to a trichotomy: either the submanifold has 
spherical topology, its universal cover splits off a Euclidean factor, 
or equality holds identically and the geometry is constrained to a small number of rigid model cases. In dimension four, the boundary behavior is 
more delicate and gives rise to K\"ahler, flat and product-type geometries. 
In both situations, the intrinsic classification is accompanied by 
restrictions on the extrinsic geometry of the immersion.

A key ingredient in the proof is the analysis of the second fundamental 
form at points where the isotropic curvature vanishes. As observed in 
\cite{XG,XZ}, condition \eqref{1} implies nonnegative isotropic curvature. 
Our main local contribution is to extract precise information about the structure of the second fundamental form at points where the isotropic curvature vanishes, relating such points to principal normal directions and to the equality case of the pinching condition. This information is then combined, in higher dimensions, with 
global structure results for manifolds with nonnegative isotropic 
curvature and with the Bochner--Weitzenb\"ock technique. In dimension 
four, additional arguments specific to four-dimensional geometry are 
required.

We now state the two main results. We recall that a vector 
$\delta\in N_fM(x)$ is called a principal normal of $f$ at $x$ if
$$
E_\delta(x)
=
\left\{
X\in T_xM:
\alpha_f(X,Y)=\langle X,Y\rangle\delta
\text{ for all }Y\in T_xM
\right\}
$$
is nontrivial. Its dimension is called the multiplicity of $\delta$, and 
$\delta$ is called a Dupin principal normal if its multiplicity is at 
least two. The relative nullity subspace of $f$ at $x$ is
$$
D_f(x)
=
\left\{
X\in T_xM:
\alpha_f(X,Y)=0
\text{ for all }Y\in T_xM
\right\}.
$$
Throughout the paper, all submanifolds are assumed to be connected.

Our first main theorem concerns dimensions $n\geq5$.

\begin{theorem}\label{n5}
Let $f\colon M^n\to\accentset{\sim}{M}^{n+m}$, $n\geq5$, 
be an isometric immersion of a compact Riemannian manifold. 
Assume that inequality \eqref{1} is satisfied and that
$\accentset{\sim}{K}^{\,f}_{\min}\geq0$ at every point. 
Then one of the following holds:
\vspace{1ex}

\noindent $(i)$ The manifold $M^n$ is diffeomorphic to a spherical 
space form.
\vspace{1ex}

\noindent $(ii)$ The universal cover of $M^n$ is isometric to a 
Riemannian product $\R\times N$, where $N$ is diffeomorphic to 
$\Sf^{n-1}$ and has nonnegative isotropic curvature.
\vspace{1ex}

\noindent $(iii)$ Equality holds in \eqref{1} at every point, and one 
of the following occurs:
\begin{enumerate}[topsep=1pt,itemsep=1pt,partopsep=1ex,
parsep=0.5ex,leftmargin=*]
\item[(a)] The immersion $f$ is totally geodesic, and either $M^n$ 
is isometric to the complex projective space $\CP^{n/2}$, the 
quaternionic projective space $\HP^{n/4}$, the Cayley plane $\OP^2$ 
(all endowed with their canonical Riemannian metrics), or $M^n$ is 
isometric to the twisted complex projective space 
$\CP^{2k-1}/\Z_2$, where the $\Z_2$-action is induced by a fixed-point-free anti-holomorphic involutive isometry.

\item[(b)] The manifold $M^n$ is flat, $f$ is totally geodesic, and
$\accentset{\sim}{K}^{\,f}_{\min}=\accentset{\sim}{K}^{\,f}_{\max}=0$ 
everywhere.

\item[(c)] The manifold $M^n$ is a quotient
$(\R^2\times\Sf^{n-2}(r))/\Gamma$, 
where $\Gamma$ is a discrete, fixed-point-free, cocompact subgroup of 
the isometry group of the standard Riemannian product 
$\R^2\times\Sf^{n-2}(r)$, and
$\accentset{\sim}{K}^{\,f}_{\min}=
\accentset{\sim}{K}^{\,f}_{\max}=0$ 
everywhere. Moreover, $f$ has index of relative nullity $2$, and
$\delta=nH_f/(n-2)$ 
is a Dupin principal normal vector field of $f$ with multiplicity 
$n-2$ satisfying $E_\delta=D_f^\perp$. 
\end{enumerate}

In particular, if $M^n$ is locally irreducible, $n$ is even and 
$\beta_2(M^n)>0$, then $f$ is totally geodesic and $M^n$ is isometric 
to the complex projective space $\CP^{n/2}$, endowed with the 
Fubini--Study metric up to scaling of the ambient metric.
\end{theorem}

The situation in dimension four is genuinely different. Besides 
sphere-type and splitting alternatives, equality in \eqref{1} produces 
K\"ahler and product geometries that have no higher-dimensional analogue 
in the preceding theorem. We recall that an isometric immersion of a 
K\"ahler manifold is called $(2,0)$-geodesic if the $(2,0)$-component 
of its second fundamental form vanishes; this notion was introduced by 
Eschenburg and Tribuzy \cite{ET} (see also \cite{ET1,eft}).

\begin{theorem}\label{n=4}
Let $f\colon M^4\to\accentset{\sim}{M}^{4+m}$ be an 
isometric immersion of a compact, oriented 
Riemannian four-manifold. Suppose that 
inequality \eqref{1} is satisfied at every point. 
Then one of the following holds:
\vspace{1ex}

\noindent $(i)$ $M^4$ is diffeomorphic to 
a spherical space form.
\vspace{1ex}

\noindent $(ii)$ The universal cover of $M^4$ is 
isometric to a Riemannian product 
$\R\times N$, where $N$ is diffeomorphic to 
$\Sf^3$ and has nonnegative Ricci curvature.
\vspace{1ex}

\noindent $(iii)$ Equality holds in \eqref{1} at every point, and 
one of the following occurs:
\vspace{0.5ex}

\begin{enumerate}[topsep=1pt,itemsep=1pt,partopsep=1ex,
parsep=0.5ex,leftmargin=*]

\item[(a)] $M^4$ is a K\"ahler manifold biholomorphic 
to the complex projective plane $\CP^2$, equipped with a 
complex structure for which $f$ is $(2,0)$-geodesic.

\item[(b)] $M^4$ is flat with first Betti number 
$1\leq\beta_1(M^4)\leq4$, and the immersion $f$ is totally 
umbilical, satisfying
$\accentset{\sim}{K}^f_{\min}
=\accentset{\sim}{K}^f_{\max}
=-H^2$ 
at every point.

\item[(c)] $M^4$ is isometric to a Riemannian product 
$(N_1,g_1)\times(N_2,g_2)$, where each $N_i$, $i=1,2$, 
is diffeomorphic to $\Sf^2$ and has nonnegative 
Gaussian curvature. Moreover,
$\accentset{\sim}{K}^f_{\min}
=\accentset{\sim}{K}^f_{\max}$ 
everywhere, and at any point where $f$ is not totally umbilical, 
there exist two distinct Dupin principal normals $\delta_1$ and 
$\delta_2$, both of multiplicity $2$, such that
$$
\langle\delta_1,\delta_2\rangle
=
-\accentset{\sim}{K}^f_{\min},
\qquad
E_{\delta_i}=TN_i,\quad i=1,2.
$$

\item[(d)] The universal cover of $M^4$ is isometric to a 
Riemannian product $(N_1,g_1)\times(N_2,g_2)$, where $N_2$ 
is diffeomorphic to $\Sf^2$, and either 
$(N_1,g_1)=(\R^2,g_0)$, where $g_0$ is the Euclidean metric 
and $N_2$ has nonnegative Gaussian curvature, or the Gaussian 
curvatures $K_{N_i}$ of $(N_i,g_i)$, $i=1,2$, satisfy
$\min K_{N_2}\geq-\min K_{N_1}>0$. 
Moreover,
$\accentset{\sim}{K}^f_{\min}
=\accentset{\sim}{K}^f_{\max}$ 
everywhere, and at any point where $f$ is not totally umbilical 
there exist two distinct Dupin principal normals satisfying the 
conclusions of $(c)$.
\end{enumerate}
\end{theorem}

Theorems \ref{n5} and \ref{n=4} show that substantial rigidity persists even at the boundary of the pinching condition. At the same time, the pinching condition does not, in general, lead to extrinsic rigidity: geometrically distinct immersions occur among the alternatives above. In Section 6 we construct families of examples realizing these phenomena. In particular, we obtain isometric immersions satisfying \eqref{1} whose underlying manifolds are either diffeomorphic to $\Sf^n$ or isometric to a product $\Sf^1\times N^{n-1}$, where $N^{n-1}$ is diffeomorphic to $\Sf^{n-1}$ and has positive sectional curvature. These examples show that the conclusions of the main theorems cannot, in general, be strengthened to extrinsic rigidity statements.

The space-form case provides a particularly transparent consequence of 
the general theory. When the ambient manifold has constant sectional 
curvature $c\geq0$, condition \eqref{1} becomes
\begin{equation}\label{sfpinch}
S\leq 4c+\frac{n^2H^2}{n-2}.
\tag{$**$}
\end{equation}
The following result extends the corresponding strict-pinching sphere 
theorems of Howard and Wei \cite{HW} and Xu and Zhao \cite{XZ} to the 
boundary case.

\begin{theorem}\label{sfhigh}
Let $f\colon M^n\to\Q_c^{n+m}$, with $n\geq5$ and $c\geq0$, 
be an isometric immersion of a compact Riemannian manifold. Assume that
$$
S\leq4c+\frac{n^2H^2}{n-2}
$$
at every point. Then $M^n$ is diffeomorphic to $\Sf^n$, or its universal 
cover is isometric to a Riemannian product $\R\times N$, where $N$ is 
diffeomorphic to $\Sf^{n-1}$ and has nonnegative isotropic curvature.
\end{theorem}

The four-dimensional space-form case, 
addressed in Theorem 2 of \cite{v4}, also follows from 
Theorem~\ref{n=4}, together with Theorem~1.1 in \cite{eft} and 
Proposition~12 in \cite{v3}.

\begin{theorem}\label{sf4}
Let $f\colon M^4\to\Q_c^{4+m}$, $c\geq0$, be an isometric 
immersion of a compact, oriented Riemannian four-manifold. Suppose that 
\eqref{sfpinch} is satisfied at every point. Then one of the following 
holds:
\vspace{1ex}

\noindent $(i)$ $M^4$ is diffeomorphic to $\Sf^4$.
\vspace{1ex}

\noindent $(ii)$ The universal cover of $M^4$ is isometric to a 
Riemannian product $\R\times N$, where $N$ is diffeomorphic to $\Sf^3$ 
and has nonnegative Ricci curvature.
\vspace{1ex}

\noindent $(iii)$ Equality holds in \eqref{sfpinch} at every point, 
and one of the following occurs:
\begin{enumerate}[topsep=1pt,itemsep=1pt,partopsep=1ex,
parsep=0.5ex,leftmargin=*]

\item[(a)] $M^4$ is isometric to a torus
$\Sf^2(r)\times\Sf^2\big(\sqrt{R^2-r^2}\big)$, 
and $f$ is a composition $f=j\circ g$, where
$g\colon M^4\rightarrow \Sf^5(R)$ 
is the standard embedding of
$\Sf^2(r)\times\Sf^2(\sqrt{R^2-r^2})$ into $\Sf^5(R)$ and
$j\colon\Sf^5(R)\rightarrow \Q_c^{4+m}$ 
is an umbilical inclusion.

\item[(b)] $M^4$ is isometric to the complex projective plane 
$\CP_r^2$ of constant holomorphic curvature $4/(3r^2)$, 
where $r=1/\sqrt{c+H^2}$, and $f=j\circ g$. Here 
$g\colon\CP_r^2\rightarrow\Sf^7(r)$ is the standard 
minimal embedding, and
$j\colon\Sf^7(r)\rightarrow \Q_c^{4+m}$
is an umbilical inclusion.
\end{enumerate}
\end{theorem}

Another consequence is the following sphere theorem, originally obtained 
by Howard and Wei \cite{HW} under the assumption of strict inequality. 
It also strengthens related results of Gu and Xu \cite{XG} (see also \cite{AB}).

\begin{corollary}\label{cor}
Let $f\colon M^n\to\Q_c^{n+m}$, with $n\geq4$ and $c\geq0$, 
be an isometric immersion of a compact, oriented Riemannian manifold. 
If
$$
S\leq2c+\frac{n^2H^2}{n-1}
$$
at every point, then $M^n$ is diffeomorphic to $\Sf^n$.
\end{corollary}

The general results also yield a rigidity theorem for submanifolds of 
complete, simply connected manifolds with bounded nonnegative sectional 
curvature.

\begin{corollary}\label{cor1l}
Let $f\colon M^n\to\accentset{\sim}{M}^{n+m}$, $n\geq5$, 
be an isometric immersion of a compact, oriented Riemannian manifold 
into a complete, simply connected Riemannian manifold 
$\accentset{\sim}{M}^{n+m}$ with bounded nonnegative sectional curvature 
$\accentset{\sim}{K}$. Assume that
$$
S\leq
\frac{16}{3}
\left(
\inf\accentset{\sim}{K}
-\frac14\sup\accentset{\sim}{K}
\right)
+
\frac{n^2H^2}{n-2}
$$
at every point. Then one of the following occurs:
\begin{enumerate}[topsep=1pt,itemsep=1pt,partopsep=1ex,
parsep=0.5ex,leftmargin=*]
\item[(i)] $M^n$ is diffeomorphic to a spherical space form;

\item[(ii)] the universal cover of $M^n$ is isometric to a Riemannian 
product $\R\times N$, where $N$ is diffeomorphic to $\Sf^{n-1}$ and 
has nonnegative isotropic curvature;

\item[(iii)] equality holds everywhere, and the submanifold is as in 
part $(iii)(a)$ of Theorem~\ref{n5}, where the ambient manifold is 
either diffeomorphic to, but not isometric to, a sphere, or isometric 
to a projective space over $\mathbb C$, $\mathbb H$, or the Cayley 
numbers.
\end{enumerate}
\end{corollary}

We conclude the introduction with a brief description of the organization 
of the paper. In Section~2 we review isotropic curvature and the global 
structure results that will be used throughout the proofs. We then study 
the relation between condition \eqref{1} and isotropic curvature, with 
particular emphasis on the equality case and the resulting structure of 
the second fundamental form. The section also contains auxiliary results 
for immersions of Riemannian products, a discussion of the 
Bochner--Weitzenb\"ock operator, and applications to even-dimensional 
submanifolds.

Section~3 is devoted to the proof of Theorem~\ref{n5}. The argument is 
divided according to whether the fundamental group of the submanifold is 
finite or infinite. Section~4 treats the four-dimensional case and proves 
Theorem~\ref{n=4}, again separating the finite and infinite fundamental 
group cases. In Section~5 we derive the consequences for submanifolds of 
space forms and for ambient manifolds with bounded nonnegative sectional 
curvature. Section~6 contains examples illustrating the sharpness of the 
main results and showing that stronger extrinsic rigidity does not hold in 
general. Finally, in the Appendix we establish an auxiliary fact concerning 
the complex sectional curvature of compact irreducible symmetric spaces 
that we were unable to locate in the literature.

\section{Preliminaries and auxiliary results}
\subsection{Isotropic curvature}
In this section, we recall the notions of positive and nonnegative isotropic curvature introduced by Micallef and Moore in \cite{MM}.

Let $(M^n,g)$ be a Riemannian manifold of dimension 
$n\geq4$ with Levi-Civit\'a connection 
$\n$ and curvature tensor $R$ given by
$$
R(X,Y)=[\n_X,\n_Y]-\n_{[X,Y]},\quad X,Y\in\mathcal X(M).
$$

Recall that the \emph{curvature operator} at a point 
$x\in M^n$ is the self adjoint endomorphism 
$$
\mathcal R\colon\Lambda^2T_xM\to\Lambda^2T_xM
$$
on the second exterior power $\Lambda^2T_xM$ of 
the tangent space, defined by
$$
\<\<\mathcal R(X\wedge Y),Z\wedge W\>\>
=\<(R(X,Y)W,Z\>,\quad X,Y,Z,W\in T_xM, 
$$
where the inner products $\<\<\cdot,\cdot\>\>$ and 
$\<\cdot,\cdot\>$ are induced by the metric $g$. 

The inner product $\<\cdot,\cdot\>$ on $T_xM$ can 
be extended to the complexification $T_xM\otimes\C$ 
either as a complex bilinear form $(\cdot,\cdot)$ or to a 
Hermitian inner product $((\cdot,\cdot))$ defined by 
$$
((v,w))=(v,\overline w),\quad v,w\in T_xM\otimes\C.
$$
The curvature operator is also extended to a complex 
linear map
$$
\mathcal R\colon\Lambda^2T_xM\otimes\C
\to\Lambda^2T_xM\otimes\C.
$$
To each complex two-plane $\sigma\subset T_xM\otimes\C$ 
we associate the \emph{complex sectional curvature} 
$K^c(\sigma)$ defined by 
$$
K^c(\sigma)=
\frac{((\mathcal R(v\wedge w),v\wedge w))}{\|v\wedge w\|^2}, 
$$
where $\{v,w\}$ is a basis of $\sigma$. A two-plane 
$\sigma\subset T_xM\otimes\C$ is called 
\emph{isotropic} if $(v,v)=0$ for all $v\in\sigma$. 

We say that $(M^n,g)$ has \emph{positive isotropic curvature} 
(respectively, \emph{nonnegative isotropic curvature})
at a point $x\in M$ if $K^c(\sigma)>0$ (respectively, 
$K^c(\sigma)\geq0$) for every isotropic two-plane at
$x$. 

To each orthonormal four-frame 
$F=\{e_1,e_2,e_3,e_4\}\subset T_xM$ at a point $x\in M^n$, 
we associate the isotropic two-plane 
$$
\sigma_F=\spa_\C\left\{e_1+\mathsf ie_2,e_3+\mathsf ie_4\right\},
$$
where $\mathsf i=\sqrt{-1}$. It follows (see \cite{MM}) that 
$$
K^c(\sigma_F)=R_{1331} + R_{1441}+ R_{2332} + R_{2442}
-2 R_{1234}.
$$
Here $R_{ijk\ell}$ denote the 
components of the curvature tensor $R$, defined by 
$$
R_{ijk\ell}=g\big(R(e_i,e_j)e_k, e_\ell\big), \quad 1\leq i,j,k,\ell\leq4. 
$$ 
Hence, $(M^n,g)$ has positive 
isotropic curvature (respectively, nonnegative isotropic 
curvature) at a point $x\in M^n$ if
$$
R_{1331} + R_{1441}+ R_{2332} + R_{2442}
-2 R_{1234}>0, 
$$
(respectively, 
$$
R_{1331} + R_{1441}+ R_{2332} + R_{2442}
-2 R_{1234} \geq0), 
$$
for all orthonormal four-frames 
$\{e_i\}_{i=1}^4\subset T_xM$. 
The manifold $(M^n,g)$ is said to have nonnegative (or positive) 
isotropic curvature if it satisfies the corresponding condition 
at every point and for all orthonormal four-frames. We follow 
this notation throughout the paper.

We recall for later use some important results on the geometry of manifolds with nonnegative isotropic curvature. A Riemannian manifold is said to be \emph{locally irreducible} if it is not locally a product of lower-dimensional manifolds.

\begin{theorem}\cite{MW}\label{duke}
Let $(M^n,g)$, $n\geq4$, be a compact locally reducible 
Riemannian manifold with nonnegative isotropic curvature. 
Then its Riemannian universal cover $(\hat M^n,\hat g)$ 
is isometric to one of the following following Riemannian 
products:
\vspace{1ex}

\noindent $(a)$ 
$(\R^{n_0},g_0)\times(N_1^{n_1},g_1)\times\cdots\times(N_r^{n_r},g_r)$, 
where $n_0\geq0, g_0$ is the flat Euclidean metric, and 
for each $1\leq i\leq r$, either $n_i=2$ and $N_i=\Sf^2$ 
has nonnegative Gaussian curvature, or else $n_i\geq3$ 
and $(N_i,g_i)$ is compact, irreducible with 
nonnegative flag curvature. 
\vspace{1ex}

\noindent $(b)$ $(\varSigma^2,g_\varSigma)\times(N^{n-2},g_N)$, 
where $\varSigma^2$ is a surface with Gaussian curvature 
$K_{\varSigma}$, and $(N,g_N)$ is a compact irreducible 
Riemannian manifold with sectional curvature $K_N$ satisfying 
$\min K_N\geq-\min K_{\varSigma}>0$. 
\end{theorem}

\begin{theorem}\cite{BSacta}\label{acta}
Let $(M^n,g)$, $n\geq4$, be a compact locally irreducible 
Riemannian manifold with nonnegative isotropic curvature. 
Then one of the following holds:
\vspace{0.1ex}

\noindent $(i)$ The manifold $M^n$ is diffeomorphic to 
a spherical space form.
\vspace{1ex}

\noindent $(ii)$ The universal cover of 
$M^n$ is a K\"ahler manifold biholomorphic to $\CP^{n/2}$.
\vspace{1ex}

\noindent $(iii)$ The universal cover of $M^n$ is isometric 
to a compact symmetric space.
\end{theorem}

\subsection{The pinching condition and the isotropic curvature}
The aim of this section is to investigate in detail the relationship between the pinching condition \eqref{1} and isotropic curvature. The auxiliary results obtained are pointwise in nature.

The following algebraic lemma was essentially proved in \cite{XZ}; we include a proof here for completeness, to make the paper self-contained.

\begin{lemma}\label{alg1}
Let $X=(x_{ij})$ be a real symmetric $4\times4$ matrix 
and define the function
\begin{align*}
f_{\varepsilon}(X)=&\sum_{i=1}^{2}\sum_{j=3}^{4}\big(x_{ii}x_{jj}-x^2_{ij}\big)+2\varepsilon(x_{14}x_{23}-x_{13}x_{24}),
\end{align*}
where $\varepsilon=\pm1$. Then 
$$
f_{\varepsilon}(X)\geq\frac{1}{2}(\trace X)^2-\|X\|^2,
$$
where $\|X\|$ denotes the Frobenius norm of $X$. 
Moreover, equality holds only if and only if 
$$
x_{11}=x_{22},\quad x_{33}=x_{44},\quad x_{12}=x_{34}=0,\quad 
x_{23}=-\varepsilon x_{14},\quad x_{24}=\varepsilon x_{13}.
$$
\end{lemma}
\proof
We apply the method of Lagrange multipliers to find the 
minimum of the function $f_{\varepsilon}$ subject to 
the constraints
\begin{align*}
\sum_{i=1}^4x_{ii}=\sigma_1, \quad 
\sum_{i,j=1}^4x_{ij}^2=\sigma_2. 
\end{align*}
We consider the function 
\begin{equation*}
g=f_{\varepsilon}+\lambda\big(\sum_{i=1}^4x_{ii}-\sigma_1\big )
+\mu\big(\sum_{i,j=1}^4x_{ij}^{2}-\sigma_2\big),
\end{equation*}
where $\lambda$ and $\mu$ are Lagrange multipliers. 
Differentiating with respect to the variables $x_{ii},1\leq i\leq4$, 
and $x_{ij},1\leq i<j\leq4$, we obtain the following system 
characterizing the critical points:
\begin{align*}
\sum_{j=3}^4&x_{jj}+\lambda+2\mu x_{ii}=0,\;i=1,2,\quad
\sum_{i=1}^2x_{ii}+\lambda+2\mu x_{jj}=0,\; j=3,4,\\
&-x_{13}-\varepsilon x_{24}+2\mu x_{13}=0,\quad 
-x_{14}+\varepsilon x_{23}+2\mu x_{14}=0,\\
&-x_{23}+\varepsilon x_{14}+2\mu x_{23}=0,\quad 
-x_{24}-\varepsilon x_{13}+2\mu x_{24}=0,\\
&\mu x_{12}=\mu x_{34}=0,\quad\sum_{i=1}^4x_{ii}
-\sigma_1=0,\quad\sum_{i,j=1}^4x_{ij}^{2}-\sigma_2=0.
\end{align*}

In the case $\lambda=0$, we obtain the critical
value $\sigma_1^2/4$. When $\lambda\neq0$, it follows 
that 
$$
x_{11}=x_{22},\quad x_{33}=x_{44},\quad x_{12}=x_{34}=0,\quad 
x_{23}=-\varepsilon x_{14},\quad x_{24}=\varepsilon x_{13}.
$$
Substituting these relations into the constraints yields the 
critical value $(\sigma_1^2-2\sigma_2)/2$, which is the 
minimum of the function $f_\varepsilon$. 
%This completes the proof.
\qed
\vspace{1ex}

\begin{lemma}\label{alg2}
Let $f\colon M^n\to\accentset{\sim}{M}^{n+m}$, $n\geq4$, be an 
isometric immersion of an $n$-dimensional Riemannian 
manifold $M^n$. For any point $x\in M^n$, any 
orthonormal basis $\{e_i\}_{i=1}^n$ 
of $T_xM$, and any orthonormal basis 
$\{\xi_\alpha\}_{\a=1}^m$ 
of the normal space $N_fM(x)$, the following 
assertions hold at $x$:
\vspace{1ex}

\noindent $(i)$ For $\varepsilon=\pm1$, we have 
\begin{align}\label{ineq2alg}
\sum_{\alpha=1}^m\Big(&\sum_{i=1}^{2}\sum_{j=3}^{4}\big(h^{\alpha}_{ii}h^{\alpha}_{jj}-(h^\alpha_{ij})^2\big)+2\varepsilon(h^{\alpha}_{14}h^{\alpha}_{23}-h^{\alpha}_{13}h^{\alpha}_{24})\Big)
\geq \frac{n^2H^2}{n-2}-S, 
\end{align}
where $h^{\alpha}_{ij}=\<A_{\xi_\alpha}e_i,e_j\>$, 
and $A_{\xi_\alpha}$ 
denotes the shape operator associated with the 
normal vector $\xi_\alpha$.
\vspace{1ex}

\noindent $(ii)$ Suppose now that equality holds in \eqref{ineq2alg}. 
Then, for every $1\leq \alpha\leq m$, 
\be\label{ineq3alg}
h^{\alpha}_{11}=h^{\alpha}_{22},\quad h^{\alpha}_{33}
=h^{\alpha}_{44},\quad h^{\alpha}_{12}
=h^{\alpha}_{34}=0,\quad 
h^{\alpha}_{23}=-\varepsilon h^{\alpha}_{14},\quad h^{\alpha}_{24}
=\varepsilon h^{\alpha}_{13}.
\ee
Moreover, if $n\geq5$, 
then there exist numbers $\rho_\alpha,1\leq \alpha\leq m$, 
such that
\be\label{ineq4alg}
h^{\alpha}_{ij}=0\;\;\text{for all}
\;\;1\leq i\leq 4,5\leq j\leq n, 
\quad h^{\alpha}_{ij}=\delta_{ij}\rho_\alpha\;\;\text{for each}
\;\;5\leq i,j\leq n 
\ee
and
\be\label{ineq5alg}
h^{\alpha}_{11}+h^{\alpha}_{44}=\rho_\alpha\quad\text{for all }
\; 1\leq \alpha\leq m.
\ee
\end{lemma}
\proof
If $n=4$, then the result follows directly from Lemma \ref{alg1}. 

Suppose now that $n\geq5$. By Lemma \ref{alg1}, we have 
\be\label{min1}
\sum_{i=1}^2\sum_{j=3}^4\big(h^{\alpha}_{ii}h^{\alpha}_{jj}-(h^\alpha_{ij})^2\big)+2\varepsilon(h^{\alpha}_{14}h^{\alpha}_{23}-h^{\alpha}_{13}h^{\alpha}_{24})
\geq\frac{1}{2}\big(\sum_{i=1}^4h^{\alpha}_{ii}\big)^2-\sum_{i,j=1}^4(h^\alpha_{ij})^2, 
\ee
\vspace{0.5ex}
\vspace*{-\baselineskip}
for any $1\leq \alpha\leq m$. Summing over $\a$ yields
\begin{align}\label{min2}
\sum_{\alpha=1}^m\Big(\sum_{i=1}^{2}&\sum_{j=3}^{4}\big(h^{\alpha}_{ii}h^{\alpha}_{jj}-(h^\alpha_{ij})^2\big)+2\varepsilon(h^{\alpha}_{14}h^{\alpha}_{33}-h^{\alpha}_{13}h^{\alpha}_{24})\Big)\nonumber \\
\geq &\sum_{\alpha=1}^m\Big(\frac{1}{2}\big(\sum_{i=1}^4h^{\alpha}_{ii}\big)^2
-\|A_{\xi_\a}\|^2+2\sum_{i=1}^4\sum_{j\geq5}^n(h^\alpha_{ij})^2
+\sum_{i,j\geq5}^n(h^\alpha_{ij})^2\Big). 
\end{align}
Using the Cauchy-Schwarz inequality, we obtain 
\be\label{min3}
\frac{1}{n-4}\big(\sum_{i\geq5}^nh^{\alpha}_{ii}\big)^2
\leq\sum_{i\geq5}^n(h^\alpha_{ii})^2
\leq\sum_{i,j\geq5}^n(h^\alpha_{ij})^2.
\ee 
Substituting this into \eqref{min2} gives 
\begin{align}\label{min4}
\sum_{\alpha=1}^m\Big(\sum_{i=1}^{2}&\sum_{j=3}^{4}\big(h^{\alpha}_{ii}h^{\alpha}_{jj}-(h^\alpha_{ij})^2\big)+2\varepsilon(h^{\alpha}_{14}h^{\alpha}_{23}-h^{\alpha}_{13}h^{\alpha}_{24})\Big)\nonumber \\
\geq &\sum_{\alpha=1}^m\Big(\frac{1}{2}\big(\sum_{i=1}^4h^{\alpha}_{ii}\big)^2
+\frac{1}{n-4}\big(\sum_{i\geq5}^nh^{\alpha}_{ii}\big)^2+2\sum_{i=1}^4\sum_{j\geq5}^n(h^\alpha_{ij})^2\Big)-S.
\end{align}
Next we observe that 
\be\label{min40}
\frac{1}{2}\big(\sum_{i=1}^4h^{\alpha}_{ii}\big)^2
+\frac{1}{n-4}\big(\sum_{i\geq5}^nh^{\alpha}_{ii}\big)^2
\geq \frac{1}{n-2}(\trace A_{\xi_\a})^2. 
\ee

Applying \eqref{min40} in \eqref{min4}, we finally obtain 
\begin{align}\label{min5}
\sum_{\alpha=1}^m\Big(\sum_{i=1}^{2}&\sum_{j=3}^{4}\big(h^{\alpha}_{ii}h^{\alpha}_{jj}-(h^\alpha_{ij})^2\big)+2\varepsilon(h^{\alpha}_{14}h^{\alpha}_{23}-h^{\alpha}_{13}h^{\alpha}_{24})\Big)\nonumber \\
\geq &\frac{n^2H^2}{n-2}-S+2\sum_{i=1}^4\sum_{j\geq5}^n(h^\alpha_{ij})^2
\geq \frac{n^2H^2}{n-2}-S,
\end{align}
which completes the proof of \eqref{ineq2alg}. 

Suppose now that equality holds in \eqref{ineq2alg}. Then 
all intermediate inequalities \eqref{min1}--\eqref{min5} 
must hold as equalities. In particular, equality in 
\eqref{min1} implies that all conditions in 
\eqref{ineq3alg} hold. Equality in \eqref{min3} yields 
$$
h^\alpha_{55}=\dots=h^\alpha_{nn}\quad \text{and}\quad 
h^\alpha_{ij}=0\quad \text{for }\; i,j\geq5, \;\, i\neq j.
$$
Equality in \eqref{min40} implies 
\be\label{41}
\sum_{i=1}^4h^{\alpha}_{ii}=
\frac{2}{n-4}\sum_{i\geq5}^nh^{\alpha}_{ii}.
\ee
Moreover, equality in \eqref{min5} forces 
$$
h^{\alpha}_{ij}=0\quad\text{for all }\; 1\leq i\leq4,\; j\geq5.
$$
Hence \eqref{ineq4alg} holds with 
$\rho_a=h^\alpha_{55}=\dots=h^\alpha_{nn}$. 
Finally, \eqref{ineq5alg} follows directly from \eqref{ineq4alg}, 
\eqref{41} and the first two conditions in \eqref{ineq3alg}. 
\qed
\vspace{1.0ex}

Let $f\colon M^n\to\accentset{\sim}{M}^{n+m}$, $n\geq4$, 
be an isometric immersion of a Riemannian manifold 
$(M^n, \<\cdot,\cdot\>)$ into a Riemannian manifold 
$(\accentset{\sim}{M}^{n+m}, \<\cdot,\cdot\>)$. Denote 
by $R$ and $\accentset{\sim}{R}$ the curvature tensors 
of $M$ and $\accentset{\sim}{M}$, respectively. 
We define the {\em{restricted curvature tensor}} of 
$f$ as the tensor field 
$$
\accentset{\sim}{R}^f\colon TM\times TM\times TM \to TM
$$ 
given by
$$
\<\accentset{\sim}{R}^f(X,Y)Z,W\>=\<\accentset{\sim}{R}(f_*X,f_*Y)f_*Z,f_*W\>
\quad\text{for all }X,Y,Z,W\in TM.
$$ 
Clearly, if at a point $x\in M^n$ one has 
$\accentset{\sim}{K}^f_{\max}(x)=\accentset{\sim}{K}^f_{\min}(x)$, 
as defined in the introduction, then
$$
\accentset{\sim}{R}^f(X,Y)Z=
\accentset{\sim}{K}^f_{\min}(x)\big(\<Y,Z\>X-\<X,Z\>Y\big)\quad\text{for all }X,Y,Z\in T_xM.
$$

The Gauss equation for the immersion $f$ can be written as
\begin{align*}
\<R(X,Y)Z,W\>&=\<\accentset{\sim}{R}^f(X,Y)Z,W\>+\<\a_f(X,W),\a_f(Y,Z)\>\\
&-\<\a_f(X,Z),\a_f(Y,W)\>, 
\end{align*}
for all $X,Y,Z,W\in TM$, where $\a_f$ is the second 
fundamental form of $f$. 

\iffalse
We recall that the functions $\accentset{\sim}{K}^f_{\max},\accentset{\sim}{K}^f_{\min}\colon M^n\to\R$ are defined by 
\begin{align*}
\accentset{\sim}{K}^f_{\max}(x)&=\max\big\{\accentset{\sim}{K} (f_{*x}\pi):\pi\subset T_xM\;\text{is a two-plane}
\big\},\\
\accentset{\sim}{K}^f_{\min}(x)&=\min\big\{\accentset{\sim}{K} (f_{*x}\pi):\pi\subset T_xM\;\text{is a two-plane}
\big\},
\end{align*}
where $\accentset{\sim}{K} (f_{*x}\pi)$ denotes the sectional curvature of the 
ambient manifold $\accentset{\sim}{M}^{n+m}$ along the two-plane 
$f_{*x}\pi\subset T_{f(x)}\accentset{\sim}{M}$ at the point $f(x)$. 
\fi

Given an orthonormal basis $\{e_i\}_{i=1}^n$
of $T_xM$ at a point $x\in M^n$, we 
write 
$$
\accentset{\sim}{R}^f_{ijk\ell}=\<\accentset{\sim}{R}^f(e_i,e_j)e_k,e_\ell\>, 
\quad \a_{ij}=\a_f(e_i,e_j), 
$$
for any $1\leq i,j,k,\ell \leq n$. This notation will be used 
throughout the paper. 

\begin{lemma}\label{R}
Let $f\colon M^n\to\accentset{\sim}{M}^{n+m}$, 
$n\geq4$, be an isometric immersion, and let $\{e_i\}_{i=1}^n$ 
be an orthonormal basis of $T_pM$ at a point 
$p\in M^n$. Suppose that 
$$
\accentset{\sim}{R}^f_{ijji}=\accentset{\sim}{K}^f_{\min}(p)\quad 
\text{for }\; i=1,2 \; \text{ and }\; j=3,4.
$$ 
Then the following assertions hold: 
\begin{align}
\accentset{\sim}{R}^f_{kiij}=\accentset{\sim}{R}^f_{kjji}=0\quad\text{for }\; & i=1,2,
\;\, j=3,4,\;\, 
1\leq k\leq n,\;\, k\neq i, k\neq j \label{R1},\\
&\accentset{\sim}{R}^f_{1342}=\accentset{\sim}{R}^f_{1423}. \label{R2}
\end{align}
\end{lemma}
\proof 
We consider the function $F\colon U\to\R$ defined by 
\begin{align*}
F(x,y)&=\<\accentset{\sim}{R}^f\big(\tau(x),\tau(y)\big)\tau(y),\tau(x)\>-\accentset{\sim}{K}^f_{\min}(p)\|\tau(x)\wedge\tau(y)\|^2\\
&=\sum_{i,j,k,\ell=1}^nx_ix_jy_ky_\ell \accentset{\sim}{R}^f_{ik\ell j}-\accentset{\sim}{K}^f_{\min}(p)\big(\|x\|^2\|y\|^2-\<x,y\>^2\big), 
\end{align*}
where $x=(x_1,\dots,x_n)$, $y=(y_1,\dots,y_n)$,
$
U=\left\{(x,y)\in\R^n\times \R^n: x\wedge y\neq0\right\}, 
$ 
and $\tau\colon\R^n\to T_pM$ 
is the linear isometry given by $\tau(x)=\sum_{i}x_ie_i$.

By our assumption, the nonnegative function $F$ attains its minimum 
at $\varepsilon_{ij}=(\tau^{-1}e_i,\tau^{-1}e_j)\in U$, for $i=1,2$ and $j=3,4$. 
Hence, 
\be\label{fermat}
\frac{\partial F}{\partial x_k}(\varepsilon_{ij})
=\frac{\partial F}{\partial y_k}(\varepsilon_{ij})=0\quad\text{for all }\;i=1,2,\;\, 
j=3,4\; \text { and }\; 1\leq k\leq n. 
\ee
A direct computation using the algebraic properties of the curvature 
tensor yields 
$$
\frac{\partial F}{\partial x_k}(\varepsilon_{ij})
=2\accentset{\sim}{R}^f_{kjji}\quad \text{and}\quad \frac{\partial F}{\partial y_k}(\varepsilon_{ij})
=2\accentset{\sim}{R}^f_{kiij}\quad\text{for}\;\;\;i=1,2,\;j=3,4\;\text{ and }\; k\neq i, k\neq j.
$$
Using \eqref{fermat} together with the curvature tensor symmetries, we obtain \eqref{R1}.

Next, consider the tangent vectors 
$$
v=\cos\theta e_1+\sin\theta e_2,\quad w=\cos\varphi e_3+\sin\varphi e_4,\quad\theta,\varphi\in\R. 
$$
Using \eqref{R1}, we compute 
$$
\<\accentset{\sim}{R}^f(v,w)w,v\>=\accentset{\sim}{K}^f_{\min}(p)+
\frac{1}{2}\big(\accentset{\sim}{R}^f_{1342}+\accentset{\sim}{R}^f_{1432}\big)\sin2\theta\sin2\varphi.
$$
Since 
$$
\<\accentset{\sim}{R}^f(v,w)w,v\>\geq\accentset{\sim}{K}^f_{\min}(p)
$$ 
for all 
$\theta,\varphi\in\R$, it follows that
$\accentset{\sim}{R}^f_{1342}+\accentset{\sim}{R}^f_{1432}=0$, 
which proves \eqref{R2}. 
\qed
\vspace{1ex}

\begin{lemma}\label{algc}
Let $f\colon M^n\to\accentset{\sim}{M}^{n+m}$, $n\geq4$, be an 
isometric immersion. For any point $x\in M^n$ and any 
orthonormal four-frame $\{e_i\}_{i=1}^4\subset T_xM$, the following 
assertions hold at $x$:
\vspace{1ex}

\noindent $(i)$ For $\varepsilon=\pm1$, we have 
\begin{align}\label{ineq1}
R_{1331}&+R_{1441}+R_{2332}+R_{2442}+2\varepsilon R_{1234}\nonumber\\
&\geq\frac{16}{3}\big(\accentset{\sim}{K}^f_{\min}(x)-\frac{1}{4}\accentset{\sim}{K}^f_{\max}(x)\big)
+\frac{n^2H^2(x)}{n-2}-S(x). 
\end{align}

\noindent $(ii)$ Assume equality in \eqref{ineq1} holds 
for an orthonormal four-frame $F=\{e_i\}_{i=1}^4\subset T_xM$. 
Then the following conditions are satisfied:
\begin{align}
\alpha_{11}=\alpha_{22},\quad\alpha_{33}
=\alpha_{44},\quad\alpha_{12}=&\alpha_{34}=0,\quad 
\alpha_{23}=-\varepsilon \alpha_{14},\quad\alpha_{24}
=\varepsilon\alpha_{13}, \label{c1}\\ 
\alpha_{11}+\alpha_{44}&=\frac{n}{n-2}\mathcal H_f(x), \label{c2}\\
\accentset{\sim}{R}^f_{1331}=\accentset{\sim}{R}^f_{1441}&=\accentset{\sim}{R}^f_{2332}=
\accentset{\sim}{R}^f_{2442}=\accentset{\sim}{K}^f_{\min}(x),\label{c3}\\
\accentset{\sim}{R}^f_{1234}=-2\accentset{\sim}{R}^f_{1342},\quad 
\accentset{\sim}{R}^f_{1342}=\frac{\varepsilon}{3}&\big(\accentset{\sim}{K}^f_{\max}(x)-
\accentset{\sim}{K}^f_{\min}(x)\big)
=\accentset{\sim}{R}^f_{1423},\label{c4}\\
\accentset{\sim}{R}^f_{kiij}=\accentset{\sim}{R}^f_{kjji}=0\quad \text{for }\; i=1,2,&\quad j=3,4,\quad 
1\leq k\leq n,\quad k\neq i,\quad k\neq j. \label{c5}
\end{align}
If $n\geq5$, then $\{e_i\}_{i=5}^n$ is chosen as any 
orthonormal basis of the orthogonal complement of 
$V_F=\spa\left\{e_i\right\}_{i=1}^4$ in $T_xM$. Moreover, the normal 
vector 
$\delta=n\mathcal H_f/(n-2)$ is a principal normal with 
$V_F^\perp\subset E_\delta(x)$.
\end{lemma}
\proof
From Berger's inequality \cite{Bergineq}, 
we obtain 
\be\label{Berg}
\varepsilon\accentset{\sim}{R}^f_{1234}\geq
-\frac{2}{3}\big(\accentset{\sim}{K}^f_{\max}(x)-\accentset{\sim}{K}^f_{\min}(x)\big).
\ee
Using the Gauss equation, \eqref{Berg}, and part $(i)$ 
of Lemma \ref{alg2}, we obtain
\begin{align}\label{algc1}
R_{1331}&+R_{1441}+R_{2332}+R_{2442}+2\varepsilon R_{1234}\nonumber\\
&=\accentset{\sim}{R}^f_{1331}+\accentset{\sim}{R}^f_{1441}+\accentset{\sim}{R}^f_{2332}+\accentset{\sim}{R}^f_{2442}+2\varepsilon\accentset{\sim}{R}^f_{1234}\nonumber\\
&+\sum_{\alpha=1}^m\Big(\sum_{i=1}^{2}\sum_{j=3}^{4}\big(h^{\alpha}_{ii}h^{\alpha}_{jj}-(h^\alpha_{ij})^2\big)+2\varepsilon(h^{\alpha}_{14}h^{\alpha}_{23}-h^{\alpha}_{13}h^{\alpha}_{24})\Big)\nonumber\\
&\geq\frac{16}{3}\big(\accentset{\sim}{K}^f_{\min}(x)-\frac{1}{4}\accentset{\sim}{K}^f_{\max}(x)\big)
+\frac{n^2H^2(x)}{n-2}-S(x), 
\end{align}
which proves \eqref{ineq1}. 

Assume now that equality is attained in \eqref{ineq1} 
for an orthonormal four-frame $F=\{e_i\}_{i=1}^4$. 
If $n\geq5$, extend it to an orthonormal basis 
$\{e_i\}_{i=1}^n$ of $T_xM$. 
Choose an orthonormal basis 
$\{\xi_\alpha\}_{\a=1}^m$ of $N_fM(x)$. 
Since equality holds in \eqref{ineq1}, 
both \eqref{Berg} and \eqref{algc1} must also 
holds as equalities. Hence, part $(ii)$ of 
Lemma \ref{alg2} applies. Then \eqref{c1} and 
\eqref{c2} follow from \eqref{ineq3alg} and 
\eqref{ineq5alg}. 

From \eqref{ineq4alg}, we obtain that the normal vector 
$
\delta=\sum_{\a=1}^m\rho_\a\xi_\a
$
is a principal normal such that 
$V_F^\perp\subset E_\delta(x)$ 
when $n\geq5$. Moreover, combining \eqref{ineq5alg} with the 
first two equalities in \eqref{c1} gives 
$\delta=n\mathcal H_f(x)/(n-2)$. 

Since \eqref{Berg} holds with equality, we obtain
\be\label{Berg1}
\accentset{\sim}{R}^f_{1234}=
-\frac{2\varepsilon}{3}\big(\accentset{\sim}{K}^f_{\max}(x)-
\accentset{\sim}{K}^f_{\min}(x)\big).
\ee
Furthermore, the fact that \eqref{algc1} is an equality immediately implies \eqref{c3}. Lemma \ref{R} then yields \eqref{c5}. Finally, \eqref{c4} follows from \eqref{R2}, \eqref{Berg1}, and the Bianchi identity.
\qed
\vspace{1ex}

The following proposition, which establishes the relationship 
between the inequality \eqref{1} and the isotropic curvature, 
is one of the key auxiliary results used in our proofs.

\begin{proposition}\label{propu}
Let $f\colon M^n\to\accentset{\sim}{M}^{n+m}$, $n\geq4$, 
be an isometric immersion such that the inequality \eqref{1} 
holds at a point $x\in M^n$. Then the 
following assertions hold at $x$:
\vspace{1ex}

\noindent $(i)$ The manifold $M^n$ has nonnegative 
isotropic curvature at $x$. 
\vspace{1ex}

\noindent $(ii)$ Suppose that $M^n$ does not have 
positive isotropic curvature at $x$. Then equality 
holds in \eqref{1} at $x$. Moreover, for any orthonormal 
four-frame $F=\{e_i\}_{i=1}^4\subset T_xM$ satisfying 
\be\label{eq1}
R_{1331}+R_{1441}+R_{2332}+R_{2442}
+2\varepsilon R_{1234}=0,
\ee
with $\varepsilon=\pm1$, the conditions 
\eqref{c1}--\eqref{c5} in part $(ii)$ of 
Lemma \ref{algc} are satisfied at $x$ 
and
\be\label{c22}
\<\alpha_{11},\alpha_{44}\>=
\|\a_{13}\|^2+\|\a_{14}\|^2
-\frac{4}{3}\big(\accentset{\sim}{K}^f_{\min}(x)
-\frac{1}{4}\accentset{\sim}{K}^f_{\max}(x)\big).
\ee
If $n\geq5$, the normal vector 
$\delta=n\mathcal H_f/(n-2)$ 
is a principal normal with $V_F^\perp\subset E_\delta(x)$, where 
$V_F=\spa\left\{e_i\right\}_{i=1}^4$. 
Furthermore, for any rotations $R_\theta$ 
and $R_\varphi$, by angles $\theta$ and $\varphi$ 
respectively, on the subspaces 
$V_1=\spa\left\{e_1,e_2\right\}$ and $V_2=\spa\left\{e_3,e_4\right\}$, 
all these conditions also hold 
for the rotated four-frame 
$\{\accentset{\sim}{e}_i\}_{i=1}^4$, defined by 
$$
\accentset{\sim}{e}_i=R_\theta e_i,
\quad \accentset{\sim}{e}_j=R_\varphi e_j,\quad i=1,2,\quad j=3,4.
$$
After appropriate rotations on both $V_1$ 
and $V_2$, we may further assume that 
\be\label{perp}
\<\a_{13},\a_{14}\>=0. 
\ee
\end{proposition}
\proof 
\noindent $(i)$ 
That $M^n$ has nonnegative 
isotropic curvature at $x$ follows from part $(i)$ of 
Lemma \ref{algc} together with \eqref{1}. 
\vspace{1ex}

\noindent $(ii)$ Suppose now that $M^n$ does not 
have positive isotropic curvature at $x$, and let 
$\left\{e_i\right\}_{i=1}^4\subset T_xM$ be an orthonormal 
four-frame satisfying \eqref{eq1}. 
That equality holds in \eqref{1} at $x$ follows 
directly from \eqref{ineq1}. Moreover, part $(ii)$ 
of Lemma \ref{algc} implies that conditions 
\eqref{c1}--\eqref{c5} hold for this four-frame and 
for any orthonormal basis $\{e_i\}_{i=5}^n$ 
of the orthogonal complement of 
$V_F$ in $T_xM$. 
Furthermore, the normal vector 
$\delta=n\mathcal H_f/(n-2)$ 
is a principal normal with 
$V_F^\perp\subset E_\delta(x)$ 
when $n\geq5$.

By \eqref{c1} and \eqref{c2}, we have
\begin{align*}
S=2\|\a_{13}\|^2+2\|\a_{14}&\|^2+(n-2)(\|\a_{11}\|^2+\|\a_{44}\|^2)
+2(n-4)\<\a_{11},\a_{44}\>,\\
n^2H^2&=(n-2)^2(\|\a_{11}\|^2+\|\a_{44}\|^2+2\<\a_{11},\a_{44}\>).
\end{align*}
Then \eqref{c22} follows form \eqref{1}, which now holds 
as equality, together with the two equations above.

Consider now the rotated orthonormal 
four-frame $\{\accentset{\sim}{e}_i\}_{i=1}^4$, 
as in the statement of the proposition. 
For convenience, write 
$\accentset{\sim}{\a}_{ij}=
\a_f(\accentset{\sim}{e}_i,\accentset{\sim}{e}_j)$. 
Then, using \eqref{c1}, we obtain 
\begin{align*}
&\accentset{\sim}{\a}_{ii}=\a_{ii},\;1\leq i\leq4,\;\,{\text{and}}
\;\;\accentset{\sim}{\a}_{12}=\tilde\a_{34}=0,\\
&\accentset{\sim}{\a}_{13}=\cos(\theta-\varepsilon\varphi)\a_{13}
+\varepsilon\sin(\theta-\varepsilon\varphi)\a_{14},\\
&\accentset{\sim}{\a}_{14}=-\varepsilon\sin(\theta-\varepsilon\varphi)\a_{13}
+\cos(\theta-\varepsilon\varphi)\a_{14},\\
&\accentset{\sim}{\a}_{23}=\sin(\theta-\varepsilon\varphi)\a_{13}
-\varepsilon\cos(\theta-\varepsilon\varphi)\a_{14},\\
&\accentset{\sim}{\a}_{24}=\varepsilon\cos(\theta-\varepsilon\varphi)\a_{13}
+\sin(\theta-\varepsilon\varphi)\a_{14}.
\end{align*}
Hence, $\accentset{\sim}{\a}_{23}=-\varepsilon\accentset{\sim}{\a}_{14}$ 
and $\accentset{\sim}{\a}_{24}=\varepsilon\accentset{\sim}{\a}_{13}$. 
Thus conditions \eqref{c1}, \eqref{c2} and \eqref{c22} also 
hold for the rotated four-frame. Moreover, using the algebraic 
properties of the curvature tensor, direct computations show 
that \eqref{c3}--\eqref{c5} remain valid. 

Suppose now that the frame $\{e_i\}_{i=1}^4$ does 
not satisfy \eqref{perp}. We claim that the angles $\theta$ 
and $\varphi$ can be chosen so that \eqref{perp} is satisfied 
for the frame $\{\accentset{\sim}{e}_i\}_{i=1}^4$. 
Indeed, straightforward computations yield 
$$
\<\accentset{\sim}{\a}_{13},\accentset{\sim}{\a}_{14}\>
=\<\a_{13},\a_{14}\>\cos2(\theta-\varepsilon\varphi)
-\frac{\varepsilon}{2}(\|\a_{13}\|^2-\|\a_{14}\|^2)\sin2(\theta-\varepsilon\varphi).
$$
Thus, we may chose angles $\theta$ and $\varphi$ such that 
\eqref{perp} holds for the rotated four-frame. 
%This completes the proof of the proposition.
\qed

\subsection{Isometric immersions of Riemannian products}

Let $f\colon M^n\to\accentset{\sim}{M}^{n+m}$, $n\geq4$, be an 
isometric immersion of a Riemannian manifold $M^n$ 
that is a Riemannian product 
$$
(N^{n_1},g_1)\times\cdots\times (N^{n_r},g_r),\quad r\geq2. 
$$
We fix a point $\bar x=(\bar x_1,\dots,\bar x_r)\in M^n$. 
For $1\leq i\leq r$, let $\sigma_i\colon N_i\to M^n$ 
denote the inclusion of $N_i$ into $M^n$, defined by 
$$
\sigma_i(x_i)=(\bar x_1,\dots,x_i,\dots,\bar x_r), \quad x_i\in N_i.
$$
Then $\sigma_i$ is totally geodesic. 

We consider the isometric immersions 
$$
f_i=f\circ\sigma_i\colon N_i\to\accentset{\sim}{M}^{n+m}, \quad 1\leq i\leq r.
$$ 
Since each $\sigma_i$ is totally geodesic, it follows that 
\be\label{ai}
\a_{f_i}(X_i,Y_i)=\a_f(\sigma_i{_*}X_i,\sigma_i{_*}Y_i) 
\quad\text{for all }\; X_i,Y_i\in TN_i.
\ee
Let $H_i$ and $\Phi_i=\a_{f_i}-g_i\mathcal H_{f_i}$ denote the 
mean curvature and the traceless part of the second fundamental 
form of the immersion $f_i$, respectively. 
We need the following auxiliary result. 

\begin{lemma}\label{lemmaprod}
Let $f\colon M^n\to\accentset{\sim}{M}^{n+m}$, $n\geq4$, be an 
isometric immersion of a Riemannian product as above. 
Suppose that $f$ satisfies the inequality \eqref{1} 
at a point $x=(x_1,\dots,x_r)\in M^n$, and consider 
an orthonormal basis $\{e^i_k\}_{k=1}^{n_i}$ of 
$T_{x_i}N_i$, $1\leq i\leq r$. Then the following 
inequality holds at $x$:
\begin{align}\label{pro}
\sum_{i=1}^r\big(\|\Phi_i\|^2+n_iH_i^2(1&-\frac{n_i}{n-2})\big)+\frac{2(n-3)}{n-2}\sum_{1\leq i<j\leq r}\sum_{k=1}^{n_i}
\sum_{\ell=1}^{n_j}\|\a_f(\sigma_i{_*}e^i_k,\sigma_j{_*}e^j_\ell) \|^2\nonumber\\
&\leq-\frac{4}{3}\accentset{\sim}{K}^f_{\max}+
\frac{1}{n-2}\Big(\sum_{i=1}^rn_i^2-n^2
+\frac{16}{3}(n-2)\Big)\accentset{\sim}{K}^f_{\min}.
\end{align} 
\end{lemma}
\proof
We consider the orthonormal basis 
$
\left\{\sigma_i{_*}e^i_k:1\leq i\leq r, 1\leq k\leq n_i\right\}
$ 
of $T_xM$. Using \eqref{ai}, we obtain 
\begin{align}\label{S}
S=&\sum_{1\leq i,j\leq r}\sum_{k=1}^{n_i}
\sum_{\ell=1}^{n_j}\|\a_f(\sigma_i{_*}e^i_k,\sigma_j{_*}e^j_\ell) \|^2
 \nonumber\\
=&\sum_{i=1}^rS_i+2\sum_{1\leq i<j\leq r}\sum_{k=1}^{n_i}
\sum_{\ell=1}^{n_j}\|\a_f(\sigma_i{_*}e^i_k,\sigma_j{_*}e^j_\ell) \|^2,
\end{align} 
where $S_i$ denotes the squared length of the second fundamental 
form of the immersion $f_i$. Furthermore, we have
$
n\mathcal H_f=\sum_{i=1}^rn_i\mathcal H_{f_i},
$
and hence 
\be\label{H}
n^2H^2=\sum_{i=1}^rn^2_iH^2_i+
2\sum_{1\leq i<j\leq r}n_in_j\<\mathcal H_{f_i}, \mathcal H_{f_j}\>.
\ee 
Since 
$
\|\Phi_i\|^2=S_i-n_iH^2_i,
$
using \eqref{S} and \eqref{H}, inequality \eqref{1} 
can be equivalently written as 
\begin{align}\label{*}
\sum_{i=1}^r\big(\|\Phi_i\|^2&+n_iH_i^2(1-\frac{n_i}{n-2})\big)
+2\sum_{1\leq i<j\leq r}\sum_{k=1}^{n_i}
\sum_{\ell=1}^{n_j}\|\a_f(\sigma_i{_*}e^i_k,\sigma_j{_*}e^j_\ell) \|^2\nonumber\\
&\leq\frac{16}{3}\big(\accentset{\sim}{K}^f_{\min}-\frac{1}{4}\accentset{\sim}{K}^f_{\max}\big)+
\frac{2}{n-2}\sum_{1\leq i<j\leq r}n_in_j\<\mathcal H_{f_i}, \mathcal H_{f_j}\>.
\end{align} 

Using \eqref{ai}, the Gauss equation for the immersion $f$ implies that the 
sectional curvature for the two-plane spanned by $e^i_k$ and $e^j_\ell$ is 
\begin{align*}
K(e^i_k\wedge e^j_\ell)=\big\langle\accentset{\sim}{R}^f(&\sigma_i{_*}e^i_k,\sigma_j{_*}e^j_\ell)\sigma_j{_*}e^j_\ell,\sigma_i{_*}e^i_k\big\rangle
+\<\a_{f_i}(e^i_k,e^i_k),\a_{f_j}(e^j_\ell,e^j_\ell)\>\\
&-\|\a_f(\sigma_i{_*}e^i_k,\sigma_j{_*}e^j_\ell) \|^2,
\end{align*}
for any $1\leq i\neq j\leq r$, $1\leq k\leq n_i$ and $1\leq \ell\leq n_j$.
Since $K(e^i_k\wedge e^j_\ell)=0$, it follows that 
\begin{align*}
n_in_j\<\mathcal H_{f_i}, \mathcal H_{f_j}\>=
-&\sum_{k=1}^{n_i}\sum_{\ell=1}^{n_j}\big\langle\accentset{\sim}{R}^f(\sigma_i{_*}e^i_k,\sigma_j{_*}e^j_\ell)\sigma_j{_*}e^j_\ell,\sigma_i{_*}e^i_k\big\rangle\\
&+\sum_{k=1}^{n_i}
\sum_{\ell=1}^{n_j}\|\a_f(\sigma_i{_*}e^i_k,\sigma_j{_*}e^j_\ell) \|^2, 
\end{align*}
for any $1\leq i\neq j\leq r$. 

Substituting this expression into \eqref{*} and estimating the curvature term by 
$\widetilde{K}^f_{\min}$ yields the desired inequality \eqref{pro}.
\qed
\vspace{1ex}

The next result characterizes the isometric immersions of Riemannian 
products satisfying the pinching condition \eqref{1}, where each factor 
is assumed to have dimension at least two.

\begin{proposition}\label{propprod}
Let $f\colon M^n\to\accentset{\sim}{M}^{n+m}$, $n\geq4$, be an 
isometric immersion of an $n$-dimensional Riemannian 
manifold $M^n$ that is isometric to a Riemannian product 
$$
(N^{n_1},g_1)\times\cdots\times (N^{n_r},g_r)\quad \text{with}\quad r\geq2 
\quad \text{and}\quad n_i\geq2. 
$$
If inequality \eqref{1} holds at every point, then one of the 
following alternatives occurs:
\vspace{1ex}

\noindent $(i)$ If $n_i<n-2$ for all $1\leq i\leq r$, 
then $\accentset{\sim}{K}^f_{\min}\leq0$ at every point. 
Moreover, if $\accentset{\sim}{K}^f_{\min}\geq0$ 
everywhere, then $f$ is totally geodesic, the 
manifold $M^n$ is flat, 
$
\accentset{\sim}{K}^f_{\min}=\accentset{\sim}{K}^f_{\max}=0
$ 
and the inequality \eqref{1} holds as an equality.
\vspace{1ex}

\noindent $(ii)$ If $n\geq5$, $n_1=2$ and $n_2=n-2$, 
then $TN_1$ is contained in the relative 
nullity distribution $\mathcal D_f$ of $f$, and the vector field 
$\delta=n\mathcal H_f/(n-2)$ is a Dupin principal 
normal such that $TN_2\subset E_\delta$. 
Moreover, at every point, 
$
\accentset{\sim}{K}^f_{\min}=\accentset{\sim}{K}^f_{\max}=0,
$
the factor $N_1$ is flat and $N_2$ has 
constant nonnegative sectional curvature.
\vspace{1ex}

\noindent $(iii)$ If $n=4$ and $n_1=n_2=2$, 
then at each point there exist two Dupin 
principal normals $\delta_1$ and $\delta_2$
such that $TN_i\subset E_{\delta_i}$, $i=1,2$. 
Moreover, at every point, 
$$
\accentset{\sim}{K}^f_{\min}=\accentset{\sim}{K}^f_{\max} 
\quad \text{and}\quad 
\<\delta_1,\delta_2\>=-\accentset{\sim}{K}^f_{\min}.
$$

Conversely, if the immersion $f$ is as described in either 
case $(ii)$ or $(iii)$, then inequality \eqref{1} holds as an 
equality at every point.
\end{proposition}
\proof
Suppose that $f$ satisfies the inequality \eqref{1} at
every point. Clearly, $n_i\leq n-2$ for all $1\leq i\leq r$. 
Hence, the left-hand side of \eqref{pro} is nonnegative and 
consequently, 
$$
\accentset{\sim}{K}^f_{\max}\leq
\frac{3}{4(n-2)}\Big(\sum_{i=1}^rn_i^2-n^2
+\frac{16}{3}(n-2)\Big)\accentset{\sim}{K}^f_{\min}. 
$$
This implies
\be\label{min}
\accentset{\sim}{K}^f_{\min}\big(\sum_{i=1}^rn^2_i-n^2+4n-8\big)\geq0. 
\ee

\noindent $(i)$ Suppose that $n_i<n-2$ for all $1\leq i\leq r$. 
Set $n_1=n-k$, where $2<k<n-2$. Then 
$$ 
\sum_{i=1}^rn^2_i=(n-k)^2+\sum_{i=2}^rn^2_i
<(n-k)^2+\big(\sum_{i=2}^rn_i\big)^2
=(n-k)^2+k^2. 
$$
Hence, 
$$
\sum_{i=1}^rn^2_i<n^2-4n+8, 
$$
and \eqref{min} yields $\accentset{\sim}{K}^f_{\min}\leq0$. 

Suppose now that $\accentset{\sim}{K}^f_{\min}\geq0$ at 
every point. Hence, $\accentset{\sim}{K}^f_{\min}=0$, 
and \eqref{pro} simplifies to 
\begin{align*}
\sum_{i=1}^r\big(\|\Phi_i\|^2+n_iH_i^2(1-\frac{n_i}{n-2})\big)&+\frac{2(n-3)}{n-2}\sum_{1\leq i<j\leq r}\sum_{k=1}^{n_i}
\sum_{\ell=1}^{n_j}\|\a_f(\sigma_i{_*}e^i_k,\sigma_j{_*}e^j_\ell) \|^2\\
&\leq-\frac{4}{3}\accentset{\sim}{K}^f_{\max}\leq0,\nonumber
\end{align*} 
where $\{e^i_k\}_{k=1}^{n_i}$ is an orthonormal frame of 
$TN_i$, $1\leq i\leq r$. The inequality above implies that 
$
\accentset{\sim}{K}^f_{\max}=0. 
$
Moreover, each of the immersions $f_i$, $1\leq i\leq r$, 
is totally geodesic, and the second fundamental form 
of $f$ is adapted to the product net $\{TN_i\}_{i=1}^r$, 
that is 
$$
\a_f(TN_i,TN_j)=0\quad\text{for all }\;1\leq i\neq j\leq r. 
$$ 
Therefore, by \eqref{ai}, the immersion $f$ is totally geodesic. 
The Gauss equation then implies that $M^n$ is flat. Clearly, equality 
holds in \eqref{1} at every point. 

\vspace{1ex}
\noindent $(ii)$ Suppose now that $n\geq5$, $n_1=2$ and $n_2=n-2$. 
Then \eqref{pro} simplifies to 
\begin{align*}
\sum_{i=1}^2\|\Phi_i\|^2+\frac{2(n-4)}{n-2}H_1^2&
+\frac{2(n-3)}{n-2}\sum_{k=1}^{n_1}
\sum_{\ell=1}^{n_2}\|\a_f(\sigma_1{_*}e^1_k,\sigma_2{_*}e^2_\ell) \|^2\\
&\leq\frac{4}{3}(\accentset{\sim}{K}^f_{\min}-\accentset{\sim}{K}^f_{\max})\leq0.\nonumber
\end{align*} 
The inequality above implies that 
$
\accentset{\sim}{K}^f_{\min}=\accentset{\sim}{K}^f_{\max},
$ 
that the immersion $f_1$ is totally 
geodesic, that $f_2$ is totally umbilical, 
and that the second fundamental from of $f$ is adapted 
to the product net $\{TN_1,TN_2\}$. Thus, by \eqref{ai}, $TN_1$ is 
contained in the relative nullity distribution $\mathcal D_f$ of $f$, 
and the vector field 
$\delta=n\mathcal H_f/(n-2)$
is a Dupin principal normal such that 
$TN_2\subset E_\delta$. 

Now take, at any 
point $x\in M^n$, a two-plane 
$\sigma\subset T_xM$ spanned by orthonormal 
vectors $v_1$ and $v_2$. From the Gauss equation, 
the sectional curvature of $M^n$ along $\sigma$ is 
$$
K(\sigma)=\accentset{\sim}{K}^f_{\min}(x)
+\<\a_f(v_1,v_1),\a_f(v_2,v_2)\>-\|\a_f(v_1,v_2)\|^2.
$$
If $v_1\in TN_1$ and $v_2\in TN_2$ then, 
since $M^n$ is the Riemannian product, this 
sectional curvature must be zero. Therefore,
$
\accentset{\sim}{K}^f_{\min}=\accentset{\sim}{K}^f_{\max}=0. 
$
Furthermore it is clear that \eqref{1} holds as an 
equality at every point. 
If $v_1,v_2\in TN_1$, the Gauss equation then 
implies that $N_1$ is flat.
Finally, if $v_1,v_2\in TN_2$ 
the Gauss equation implies that 
$K(\sigma)=\|\delta(x)\|^2$. 
Then, by Schur's theorem, $N_2$ has constant 
nonnegative sectional curvature. 

\vspace{1ex}
\noindent $(iii)$ Suppose now that $n=4$ and $n_1=n_2=2$. 
Then \eqref{pro} simplifies to 
\begin{align*}
\sum_{i=1}^2\|\Phi_i\|^2
+2\sum_{k=1}^{n_1}
\sum_{\ell=1}^{n_2}\|\a_f(\sigma_1{_*}e^1_k,\sigma_2{_*}e^2_\ell) \|^2
\leq\frac{8}{3}(\accentset{\sim}{K}^f_{\min}-\accentset{\sim}{K}^f_{\max})\leq0.\nonumber
\end{align*} 
Thus, 
$
\accentset{\sim}{K}^f_{\min}=\accentset{\sim}{K}^f_{\max},
$ 
both immersions $f_1$ and $f_2$ are totally umbilical, 
and the second fundamental from of $f$ is adapted 
to the net $\{TN_1,TN_2\}$. Hence, by \eqref{ai}, 
the vector fields 
$\delta_i=\mathcal{H}_{f_i}$, $i=1,2$, 
are Dupin principal normals of $f$ satisfying 
$TN_i\subset E_{\delta_i}$. 

Take, at any point $x\in M^n$, a two-plane 
$\sigma\subset T_xM$ spanned by unit 
vectors $v_1\in TN_1$ and $v_2\in TN_2$.
From the Gauss equation the 
sectional curvature of $M^n$ along $\sigma$ 
is given by
$$
K(\sigma)=\accentset{\sim}{K}^f_{\min}(x)
+\<\delta_1,\delta_2\>.
$$
Since $M^n$ is the Riemannian product of 
$(N_1,g_1)$ and $(N_2,g_2)$, we 
must have $K(\sigma)=0$. Therefore, 
$
\<\delta_1,\delta_2\>=-\accentset{\sim}{K}^f_{\min} 
$
at every point. 

Conversely, suppose that the submanifold 
$f$ is as described in cases $(ii)$ and $(iii)$ of 
the proposition. Then it is straightforward to 
verify that inequality \eqref{1} holds as an 
equality at every point.\qed

\subsection{The Bochner--Weitzenb\"ock operator}

Let $(M^n,\<\cdot,\cdot\>)$ be an oriented Riemannian 
manifold of dimension $n\geq4$ with Levi-Civit\'a connection 
$\n$ and curvature tensor $R$. The Ricci tensor of 
$(M^n,\<\cdot,\cdot\>)$ is defined by
$$
\Ric(X,Y)=\sum_{i=1}^n\<R(X,E_i)E_i,Y\>,\quad X,Y\in\mathcal X(M),
$$
where $\{E_i\}_{i=1}^n$ is a local orthonormal frame.

At any point $x\in M$, we consider the {\emph 
{Bochner--Weitzenb\"ock operator} $\B^{[2]}$ as an 
endomorphism of the space of 2-vectors 
$\Lambda^2T_xM$ at $x$, given by (see \cite{sea})
\begin{align}\label{WB}
\<\<\B^{[2]}(v_1\wedge v_2), w_1\wedge w_2\>\>
=&\;\Ric(v_1,w_1)\<v_2,w_2\>+\Ric(v_2,w_2)\<v_1,w_1\>\nonumber\\
&-\Ric(v_1,w_2)\<v_2,w_1\>-\Ric(v_2,w_1)\<v_1,w_2\>\nonumber\\
&-2\<R(v_1,v_2)w_2,w_1\>,
\end{align}
and then extend it linearly to all of $\Lambda^2T_xM$.
Here $\<\<\cdot,\cdot\>\>$ stands for the inner product of 
$\Lambda^2T_xM$ defined by 
$$
\<\<v_1\wedge v_2,w_1\wedge w_2\>\>=\det(\<v_i,w_j\>).
$$
The Bochner--Weitzenb\"ock operator is self-adjoint. 
If $Z\in\Lambda^2T_xM$, the dual 2-form $\omega$ is 
defined by 
$$
\omega(v,w)=\<\<Z,v\wedge w\>\>, 
$$ 
and we may regard $Z$ as the skew-symmetric 
endomorphism of the tangent space at $x$ via 
$$
\<Z(v),w\>=\<\<Z,v\wedge w\>\>. 
$$

Clearly, $\B^{[2]}$ can also be viewed as an endomorphism 
of the bundle $\Omega^2(M)$ of 2-forms of the manifold via 
the inner product $\<\<\cdot,\cdot\>\>$. If $\omega$ is a 
2-form, then $\B^{[2]}\omega$ is given by 
$$
\B^{[2]}\omega(X_1,X_2)=\omega(\Ric(X_1),X_2)
+\omega(X_1,\Ric(X_2))-\sum_i\omega(R(X_1,X_2)E_i,E_i).
$$
Then the Bochner--Weitzenb\"ock operator acts on 
$\Lambda^2TM$ by 
$$
\omega(\B^{[2]}(X_1\wedge X_2))
=\B^{[2]}\omega\,(X_1,X_2).
$$ 
Taking $\omega$ to 
be the dual form of $w_1\wedge w_2$ yields \eqref{WB}. 
Throughout the paper, we will generally identify 2-forms 
with 2-vectors.

We recall the Bochner--Weitzenb\"ock formula, which 
can be written as 
$$
\<\Delta\omega,\omega\>=\frac{1}{2}\Delta\|\omega\|^2+\|\n\omega\|^2
+\<\B^{[2]}\omega,\omega\>
$$
for any $\omega\in\Omega^2(M)$. It follows that any 
harmonic 2-form on a compact manifold is parallel, provided 
that the Bochner--Weitzenb\"ock operator is nonnegative. 

It is worth mentioning that for any four-dimensional 
Riemannian manifold $M$, the non-negativity of the 
isotropic curvature at a point $x\in M$ is equivalent 
to the non-negativity of the Bochner--Weitzenb\"ock 
operator $\B^{[2]}$ at $x$ (see, for instance, \cite{MWolf}).

\subsection{Even dimensional submanifolds}
Even-dimensional Riemannian manifolds with nonnegative isotropic curvature are known to have a nonnegative Bochner--Weitzenböck operator $\mathcal{B}^{[2]}$ (see \cite{sea}). In this section, we strengthen this result in a form adapted to our purposes.

Let $M^n$, $n=2k$, be an even-dimensional Riemannian 
manifold and let $\omega\in\Omega^2(M^n)$ 
be a 2-form. Choose an orthonormal basis 
$\left\{v_1,w_1,\dots,v_k,w_k\right\}$ of the tangent space
at a point $x\in M^n$ such that
\be\label{Lambda}
\omega(v_a,w_b)=\lambda_a\delta_{ab}\quad \text{and}\quad 
\omega(v_a,v_b)=\omega(w_a,v_w)=0\quad \text{for all }\;1\leq a,b\leq k,
\ee
where $\lambda_a$ are real numbers.

\begin{lemma}\label{lemmaeven2}
Let $M^n$, with $n=2k\geq4$, be an even-dimensional Riemannian 
manifold. Suppose that $M^n$ has nonnegative isotropic 
curvature at a point $x\in M^n$. Then the 
following assertions hold:
\vspace{1ex}

\noindent $(i)$ The Bochner--Weitzenb\"ock operator 
$\B^{[2]}$ is nonnegative at $x$. 
\vspace{1ex}

\noindent $(ii)$ Let $\omega\in\Omega^2(M^n)$ 
be a 2-form such that 
$$
\<\B^{[2]}\omega,\omega\>(x)=0.
$$
Choose an orthonormal basis 
$\{v_1,w_1,\dots,v_k,w_k\}$ of the tangent space
$T_xM$ such that \eqref{Lambda} holds. 
Then, for all $1\leq a\neq b\leq k$, the following condition is satisfied at $x$:
\begin{align}\label{even=}
\lambda_a^2\big(\left<R(v_b,v_a)v_a,v_b\right>&
+\left<R(v_b,w_a)w_a,v_b\right> +\left<R(w_b,v_a)v_a,w_b\right> \nonumber\\
&+\left<R(w_b,w_a)w_a,w_b\right> -2\left|\<R(v_b,w_b)v_a,w_a\>\right|\big)=0.
\end{align} 
\end{lemma}
\proof
Choose an orthonormal basis 
$\{e_i\}_{i=1}^n$ of 
$T_xM$ such that $e_{2a-1}=v_a$ and 
$e_{2a}=w_a$ for $1\leq a\leq k$, where 
$\{v_1,w_1,\dots,v_k,w_k\}$ is chosen so that 
\eqref{Lambda} holds. Let $\{\omega_i\}_{i=1}^n$ 
denote the dual basis to $\{e_i\}_{i=1}^n$. Then
$$
\<\B^{[2]}\omega,\omega\>(x)=
\frac{1}{4}\sum_{i,j=1}^n\omega_{ij}\B^{[2]}\omega(e_i,e_j),
$$
where $\omega_{ij}=\omega(e_i,e_j)$. 
On the other hand, 
\begin{align*}
\B^{[2]}\omega(X_1,X_2)&=\sum_{i=1}^n\big((R(e_i,X_1)\omega)(e_i,X_2)
-(R(e_i,X_2)\omega)(e_i,X_1)\big)\\
&=-\sum_{i=1}^n\omega(R(e_i,X_1)e_i,X_2)-\sum_{i=1}^n\omega(e_i,R(e_i,X_1)X_2)\\
&\;+\sum_{i=1}^n\omega(R(e_i,X_2)e_i,X_1)+\sum_{i=1}^n\omega(e_i,R(e_i,X_2)X_1)
\end{align*}
for every $X_1,X_2\in T_xM$. Hence,
$$
2\<\B^{[2]}\omega,\omega\>(x)=\sum_{i,j,k,\ell=1}^n\big(R_{i\ell ij}\omega_{jk}
+R_{i\ell kj}\omega_{ij}\big)\omega_{k\ell}. 
$$

Using \eqref{Lambda}, we have
\begin{align*}
\sum_{i,j,k,\ell=1}^nR_{i\ell ij}\omega_{jk}\omega_{k\ell}=&\sum_{a,b=1}^n\lambda_a^2\big(\left<R(v_b,w_a)w_a,v_b\right>
+ \left<R(w_b,w_a)w_a,w_b\right> \nonumber\\
&\;+\left<R(v_b,v_a)v_a,v_b\right>
+ \left<R(w_b,v_a)v_a,w_b\right>\big).
\end{align*}
Similarly, 
$$
\sum_{i,j,k,\ell=1}^nR_{i\ell kj}\omega_{ij}\omega_{k\ell}
=2\sum_{a,b=1}^n\lambda_a\lambda_b\big(\left<R(v_b,w_a)w_a,v_b\right> 
-\left<R(v_b,v_a)w_a,w_b\right> \big).
$$
Form the Bianchi identity, 
$$
\left<R(v_b,w_a)v_a,w_b\right>-\left<R(v_b,v_a)w_a,w_b\right>
=\left<R(v_b,w_b)w_a,v_a\right>,
$$
and thus
$$
\sum_{i,j,k,\ell=1}^nR_{i\ell kj}\omega_{ij}\omega_{k\ell}
=2\sum_{a,b=1}^n\lambda_a\lambda_b\left<R(v_b,w_a)v_a,w_b\right>.
$$

Combing the above, we obtain 
\begin{align}\label{even}
\<\B^{[2]}\omega,\omega\>(x)&=
\sum_{1\leq a\neq b\leq k}\lambda_a^2\big(\left<R(v_b,v_a)v_a,v_b\right>
+ \left<R(v_b,w_a)w_a,v_b\right> \nonumber\\
&+\left<R(w_b,v_a)v_a,w_b\right>
+ \left<R(w_b,w_a)w_a,w_b\right>\big)\\
&+2\sum_{1\leq a\neq b\leq k}\lambda_a\lambda_b\left<R(v_b,w_b)v_a,w_a\right>.
\nonumber
\end{align} 

\noindent $(i)$ By our assumption on the isotropic curvature, we have 
\begin{align}
\lambda_a^2\big(\left<R(v_b,v_a)v_a,v_b\right>
&+ \left<R(v_b,w_a)w_a,v_b\right> +\left<R(w_b,v_a)v_a,w_b\right>
 \nonumber\\
&+ \left<R(w_b,w_a)w_a,w_b\right>\big)
\geq2\lambda_a^2\left|\<R(v_b,w_b)v_a,w_a\>\right|. \label{even1}
\end{align}
From \eqref{even} and \eqref{even1}, it follows that 
\begin{align*}
\<\B^{[2]}\omega,\omega\>(x)&\geq
2\sum_{1\leq a\neq b\leq k}\big(\lambda_a^2\left|\<R(v_b,w_b)v_a,w_a\>\right|
+2\lambda_a\lambda_b\left<R(v_b,w_b)v_a,w_a\right>\big)\nonumber\\
&\geq\sum_{1\leq a\neq b\leq k}(|\lambda_a|-|\lambda_b|)^2\left|\<R(v_b,w_b)v_a,w_a\>\right|\geq0. 
\end{align*} 
This shows that $\B^{[2]}$ is nonnegative at $x$. 
\vspace{1ex}

\noindent $(ii)$ Let $\omega\in\Omega(M^n)$ 
be a 2-form such that 
$\<\B^{[2]}\omega,\omega\>(x)=0$. 
Then all of the above inequalities hold as equalities. In 
particular, equality in \eqref{even1} yields \eqref{even=}.
\qed

\begin{proposition}\label{propeven}
Let $M^n$, with $n=2k\geq4$, be an even-dimensional 
Riemannian manifold. Suppose that an isometric immersion 
$f\colon M^n \to \accentset{\sim}{M}^{n+m}$ 
satisfies inequality \eqref{1} at a point $x\in M^n$. 
Then the following assertions hold:
\vspace{1ex}

\noindent $(i)$ The Bochner--Weitzenb\"ock operator 
$\B^{[2]}$ is nonnegative at $x$. 
\vspace{1ex}

\noindent $(ii)$ Let $\omega\in\Omega^2(M^n)$ 
be a 2-form such that 
$$
\<\B^{[2]}\omega,\omega\>(x)=0 \quad \text{and}\quad \omega_x\neq0.
$$
If $n\geq6$, then the following conditions are satisfied at $x$:
\begin{itemize}
\item[(a)] The inequality \eqref{1} holds as an equality at $x$, and 
\be\label{mM}
\accentset{\sim}{K}^f_{\max}(x)=4\accentset{\sim}{K}^f_{\min}(x).
\ee 
\item[(b)] If $f$ is not totally geodesic at $x$, then exactly 
one of the numbers $\lambda_a$, $1\leq a\leq k$, in \eqref{Lambda} 
is nonzero, and 
$
\accentset{\sim}{K}^f_{\min}(x)=\accentset{\sim}{K}^f_{\max}(x)=0. 
$
Moreover, the vector 
$\delta=n\mathcal H_f/(n-2)$ 
is a Dupin principal normal of multiplicity $n-2$, and the 
orthogonal complement of $E_\delta(x)$ in $T_xM$ is the 
relative nullity distribution of $f$ at $x$.
\end{itemize}
\end{proposition}
\proof
\noindent $(i)$ This follows directly from part $(i)$ of Proposition 
\ref{propu} together with part $(i)$ of Lemma \ref{lemmaeven2}. 
\vspace{1ex}

\noindent $(ii)$ Choose an orthonormal basis 
$\{e_i\}_{i=1}^n$ of 
$T_xM$ such that $e_{2a-1}=v_a$ and 
$e_{2a}=w_a$ for $1\leq a\leq k$, where 
$\{v_1,w_1,\dots,v_k,w_k\}$ is chosen so that \eqref{Lambda} holds. 
Since $\omega_x\neq0$, we may assume, without loss 
of generality, that $\lambda_1\neq0$. It follows from 
\eqref{even=} that the isotropic curvature vanishes for the 
four-frame $\{v_1,w_1,v_a,w_a\}$ for all $a\neq1$. 

Then part $(ii)$ of Proposition \ref{propu} implies that 
\eqref{1} holds as an equality, and the vector 
$\delta=n\mathcal H_f/(n-2)$
is a Dupin principal normal such that $E_\delta(x)$ contains 
the orthogonal complement of 
${\operatorname{span}}\{v_1,w_1,v_a,w_a\}$
in $T_xM$ for all $a\neq1$. Moreover, conditions 
\eqref{c1}--\eqref{c5} are satisfied for all such four-frames. 
Thus 
$
{\operatorname{span}}\left\{e_i\right\}_{i=3}^n\subset E_\delta(x). 
$
In particular, 
$$
\a_{ii}=\delta\quad\text{for all }\;i\geq3. 
$$
Then it follows from the first equality in \eqref{c1} and 
\eqref{c2} that 
$
\a_{11}=\a_{22}=0.
$
Furthermore, from 
$
{\operatorname{span}}\{e_i\}_{i=3}^n\subset E_\delta(x) 
$ 
and the last three equalities in \eqref{c1}, we obtain 
$$
\a_{ij}=0\quad\text{for all }\;1\leq i\neq j\leq n.
$$
Hence, ${\operatorname{span}}\{e_1,e_2\}$ is 
contained in the relative nullity distribution of $f$ at $x$. 
Therefore, 
$$
S(x)=(n-2)\|\delta\|^2=\frac{n^2H^2(x)}{n-2}.
$$ 
Since 
equality holds in \eqref{1}, we conclude that 
\eqref{mM} holds.

Suppose now that $f$ is not totally geodesic at $x$. 
Then $\delta\neq0$ is a Dupin 
principal normal of multiplicity $n-2$ with 
$E_\delta(x)={\operatorname{span}}\{e_i\}_{i=3}^n$, 
and the relative nullity distribution of $f$ at $x$ is 
$\mathcal D_f(x)={\operatorname{span}}\{e_1,e_2\}$.

We now prove that exactly one of the 
numbers $\lambda_a$, $1\leq a\leq k$, is 
nonzero. Suppose to the 
contrary that $\lambda_2\neq0$ as well. Then 
\eqref{even=} implies that the isotropic 
curvature vanishes for the four-frame 
$\{v_2,w_2,v_a,w_a\}$ for all $a\neq2$. 
Hence, by part $(ii)$ of Proposition \ref{propu}, 
we conclude that the orthogonal complement of 
${\operatorname{span}}\{v_2,w_2,v_a,w_a\}$
in $T_xM$ is contained in $E_\delta(x)$ for 
all $a\neq2$. Therefore, 
$
{\operatorname{span}}\left\{e_1,e_2\right\}\subset E_\delta(x), 
$
which is clearly a contradiction. 
Hence, 
$$
\lambda_a=0\quad\text{for all }\; 2\leq a\leq k.
$$

It remains to prove that 
$
\accentset{\sim}{K}^f_{\min}(x)=\accentset{\sim}{K}^f_{\max}(x)=0. 
$
Let $\{\omega_i\}_{i=1}^n$ denote the dual 
basis to $\{e_i\}_{i=1}^n$. Then 
$
\omega_x=\lambda_1\omega_1\wedge\omega_2
$
and thus, by assumption, 
$$
\<\<\B^{[2]}e_1\wedge e_2,e_1\wedge e_2\>\>=0.
$$
Using \eqref{WB}, this can be written as
$$
\Ric(e_1)+\Ric(e_2)-2R_{1221}=0.
$$
Since $\mathcal D_f(x)={\operatorname{span}}\{e_1,e_2\}$,
via the Gauss equation, the above becomes
$$
\sum_{i=3}^n\big(\accentset{\sim}{R}^f_{1ii1}
+\accentset{\sim}{R}^f_{2ii2}\big)=0. 
$$
On the other hand, condition \eqref{c3} implies that 
$$
\accentset{\sim}{R}^f_{1ii1}
=\accentset{\sim}{R}^f_{2ii2}
=\accentset{\sim}{K}^f_{\min}(x)\quad\text{for all }\;3\leq i\leq n, 
$$
by applying the condition to the corresponding frames.
Hence, $\accentset{\sim}{K}^f_{\min}(x)=0$, and then \eqref{mM} 
yields $\accentset{\sim}{K}^f_{\max}(x)=0$.
\qed 

\section{Proof of the result for $n\geq5$}
The proof of Theorem \ref{n5} is divided into two cases, depending on whether the fundamental group of the submanifold is finite or infinite.

Before proceeding, we first establish the following result.

\begin{theorem}\label{ceven}
Let $f\colon M^n\to\accentset{\sim}{M}^{n+m}$, 
$n\geq 6$, be an isometric immersion of a compact, 
locally irreducible Riemannian manifold. Assume that 
inequality \eqref{1} holds at every point. If 
$n$ is even and the second Betti number satisfies 
$\beta_2(M^n)>0$, then $M^n$ is 
isometric to the complex projective space 
$\CP^{n/2}$, endowed with the Fubini-Study 
metric up to scaling, and $f$ is totally geodesic. 
\end{theorem}
\proof
Proposition \ref{propu} implies that $M^n$ has 
nonnegative isotropic curvature. Moreover, by part $(i)$ 
of Proposition \ref{propeven}, the Bochner--Weitzenb\"ock 
operator $\B^{[2]}$ is nonnegative. Since $\beta_2(M^n)>0$, 
there exists a nontrivial harmonic 2-form 
$\omega\in\Omega^2(M^n)$. By the 
Bochner--Weitzenb\"ock formula, $\omega$ is 
parallel and 
$
\<\B^{[2]}\omega,\omega\>=0
$
everywhere. Hence part $(ii)$ of Proposition 
\ref{propeven} applies at every point. 

Define the skew-symmetric endomorphism $L$ 
of the tangent bundle of $M^n$ by 
$$
\omega(X,Y)=\<LX,Y\>\quad\text{for all } X,Y\in\mathcal X(M). 
$$
Since $\omega$ is parallel, it follows that $L$ is parallel. 
Hence, $L^2$ is also parallel and self-adjoint. Therefore, 
the eigenvalues of $L^2$ are constant on every connected 
component of an open dense subset of $M^n$. By continuity, 
the eigenvalues are in fact constant on all of $M^n$. We 
claim that these eigenvalues are all equal. Otherwise, the 
eigenbundles, corresponding to distinct eigenvalues, 
would yield an orthogonal decomposition of $TM$ 
into parallel subbundles. This contradicts the fact that 
$M^n$ is locally irreducible. 

Hence $L^2=\mu\,{\rm{Id}}$ for some constant $\mu\in\R$. 
At any point $x\in M^n$, we choose an orthonormal 
basis $\{e_i\}_{i=1}^n$ of $T_xM$ such that 
$e_{2a-1}=v_a$ and $e_{2a}=w_a$ for 
$1\leq a\leq k$, where $\{v_1,w_1,\dots,v_k,w_k\}$ 
is chosen so that \eqref{Lambda} holds. Clearly, 
at least one of the numbers $\lambda_a$ is different 
from zero. Then
$$
Le_{2a-1}=\lambda_ae_{2a},\quad Le_{2a}=
-\lambda_ae_{2a-1}\quad\text{for all }\;1\leq a\leq k.
$$
Since $L^2=\mu\,{\rm{Id}}$, we obtain 
$\lambda^2_1=\dots=\lambda^2_k=-\mu<0$. 
Hence $J=\frac{1}{\sqrt{-\mu}}\, L$
defines an almost complex structure. Because $L$ is 
skew-symmetric, $J$ is orthogonal. Thus, the 
triple $(M^n,\<\cdot,\cdot\>,J)$ is a K\"ahler manifold. 

Next, we claim that $f$ is totally geodesic. Suppose 
to the contrary that $f$ is not totally geodesic at a point 
$x_0\in M^n$. Then part $(ii)$-$(b)$ of Proposition \ref{propeven} 
would imply that exactly one of the numbers 
$\lambda_a$, $1\leq a\leq k$, 
is nonzero, which is a contradiction. 
Hence, $f$ is totally geodesic. 

It follows from the Gauss equation and \eqref{mM} that 
the manifold $M^n$ has weakly $1/4$-pinched sectional curvature, 
that is $0\leq K(\sigma_1)\leq 4K(\sigma_2)$ for all points
$x\in M^n$ and all two-planes $\sigma_1,\sigma_2\subset T_xM$. 
Theorem 1 in \cite{BSacta} implies that $M^n$ is locally 
symmetric. It follows from Theorem 2.1 in \cite{MW} that 
$M^n$ is simply connected. Hence, $M^n$ is a compact 
Hermitian symmetric space. 

We claim that $M^n$ has positive sectional curvature. 
Suppose, to the contrary, that $K(\sigma_0)=0$ 
for some point $x_0\in M^n$ and a two-plane 
$\sigma_0\subset T_{x_0}M$. Then $M^n$ is flat at $x_0$, 
because it has weakly $1/4$-pinched sectional curvature. 
Since $M^n$ is homogeneous, it follows that $M^n$ is flat. 
This contradicts the fact that $M^n$ is compact and 
simply connected. Therefore, $M^n$ has positive 
sectional curvature. By a classical result due to Berger 
\cite{Berg} and Klingenberg \cite{Kling}, $M^n$ is a symmetric 
space of rank one. Thus $M^n$ is isometric to the complex 
projective space $\CP^{n/2}$ with the Fubini-Study metric 
up to scaling.
\qed

\subsection{Submanifolds with finite fundamental group}

The aim of this section is to establish the following result for 
submanifolds with finite fundamental group satisfying the 
pinching condition \eqref{1}.

\begin{theorem}\label{p1fn5}
Let $f\colon M^n\to\accentset{\sim}{M}^{n+m}$, 
$n\geq 5$, be an isometric immersion of a 
compact Riemannian manifold with finite 
fundamental group. Assume that inequality 
\eqref{1} is satisfied and that 
$\accentset{\sim}{K}^f_{\min}\geq0$ at every 
point. Then $M^n$ is locally irreducible, and 
one of the following alternatives holds: 
\vspace{1ex}

\noindent $(i)$ The manifold $M^n$ is diffeomorphic 
to a spherical space form.
\vspace{1ex}

\noindent $(ii)$ Equality holds in \eqref{1} at every 
point, the immersion $f$ is totally geodesic, and either 
$M^n$ is isometric to the complex projective space 
$\CP^{n/2}$, the quaternionic projective space 
$\HP^{n/4}$, the Cayley plane $\OP^2$ (all endowed 
with their canonical Riemannian metrics), or $M^n$ 
is isometric to the twisted complex projective space 
$\CP^{2k-1}/\Z_2$, where the $\Z_2$-action is 
induced by a fixed-point-free anti-holomorphic 
involutive isometry.

\indent 
In particular, if $n$ is even and $\beta_2(M^n)>0$, then 
$f$ is totally geodesic and $M^n$ is isometric to the 
complex projective space $\CP^{n/2}$, endowed 
with the Fubini-Study metric up to scaling of the 
ambient metric. 
\end{theorem}
\proof
Proposition \ref{propu} implies that $M^n$ has 
nonnegative isotropic curvature. We begin by 
proving the theorem for simply connected submanifolds. 
The proof is divided into two cases.
\vspace{0.5ex}

\indent\emph{Case I}. Suppose that there exists 
a point at which $M^n$ has positive isotropic 
curvature. Since, $M^n$ has nonnegative isotropic 
curvature, it follows from Remark (iv) in \cite{ses} 
that $M^n$ admits a metric of positive isotropic 
curvature. By the result of Moore and Micallef 
\cite{MM}, $M^n$ must be homeomorphic to $\Sf^n$. 

Suppose $M^n$ were locally reducible. Then there would 
exist a decomposition 
$M^n\cong M_1\times M_2$ with $\dim M_1,\dim M_2\geq1$. 
However, $\Sf^n$ cannot be expressed as a topological product 
of lower-dimensional manifolds, so this is impossible. 
Therefore, $M^n$ is locally irreducible.
Hence, $M^n$ falls under the cases described in Theorem \ref{acta}. 
Since $M^n$ is homeomorphic to $\Sf^n$, case $(ii)$ 
in that theorem is excluded. In the remaining cases, 
the manifold is therefore diffeomorphic to $\Sf^n$. 
\vspace{0.5ex}

\indent\emph{Case II}. Now suppose that $M^n$ has no 
points of positive isotropic curvature. It then follows from 
part $(ii)$ of Proposition \ref{propu}, that the inequality 
\eqref{1} holds as an equality at every point. 

We claim that $M^n$ is locally irreducible. Suppose, 
to the contrary, that $M^n$ is reducible. Then $M^n$ 
must be isometric to a Riemannian product as in one 
of the two cases of Theorem \ref{duke}. Case $(b)$ of 
that theorem is ruled out by part $(ii)$ of Proposition 
\ref{propprod}. Suppose that $M^n$ is a Riemannian product 
as in case $(a)$. Since $M^n$ has no Euclidean 
factor, Proposition \ref{propprod} applies. 
For dimension reasons, only cases $(i)$ and $(ii)$ in 
that proposition can occur. If case $(i)$ occurs, then 
$M^n$ would be flat, contradicting the fact 
that $M^n$ is compact and simply connected. In case 
$(ii)$, we have $n_1=2$ and $N_1$ is flat, which 
contradicts the requirement that $N_1$ be 
diffeomorphic to $\Sf^2$. 

Hence, $M^n$ is locally irreducible and therefore 
falls under the cases described in Theorem \ref{acta}. 
In case $(i)$, $M^n$ is clearly diffeomorphic to $\Sf^n$. 
In case $(ii)$, $M^n$ is a K\"ahler manifold 
biholomorphic to $\CP^{n/2}$. Then Theorem 
\ref{ceven} implies that $M^n$ is isometric to 
the complex projective space $\CP^{n/2}$, 
endowed with the Fubini-Study metric up to 
scaling, and that $f$ is totally geodesic. 

Assume now that case $(iii)$ holds, namely that $M^n$ 
is isometric to a compact symmetric space. Since 
$M^n$ has no points of positive isotropic curvature, 
its sectional curvature cannot be constant 
unless $M^n$ is flat. The later possibility is ruled out 
by the fact that $M^n$ is compact and simply connected. 

From part $(ii)$ of Proposition \ref{propu} and Proposition 
\ref{fillup}, we conclude that $E_\delta(x)=T_xM$ at each 
point $x$. Thus, $f$ is totally umbilical. Since 
$\delta=n\mathcal H_f/(n-2)$, it follows that $f$ is 
totally geodesic. The Gauss equation then implies that 
the symmetric space $M^n$ has weakly 
$1/4$-pinched sectional curvatures. Therefore, $M^n$ is 
a compact rank-one symmetric space (see \cite{Berg,Kling}).
\vspace{0.5ex}

Now suppose that the fundamental group 
$\pi_1(M^n)$ of $M^n$ is 
finite, and consider the universal covering 
$\pi\colon\hat M^n\to M^n$. Then 
$\hat M^n$ is compact. Moreover, 
the isometric immersion $\hat f=f\circ\pi$ 
satisfies condition \eqref{1}. Therefore, by 
the preceding argument, $\hat M^n$ is locally irreducible. 
Furthermore, either $\hat M^n$ is diffeomorphic 
to $\Sf^n$, or equality holds in \eqref{1} 
at every point, $\hat f$ is totally geodesic 
and $\hat M^n$ is isometric to a compact 
rank-one symmetric space of non-constant 
sectional curvature. 

In the former case, $M^n$ is diffeomorphic 
to a spherical space form. In the latter 
case, equality holds in \eqref{1} at every 
point, and $M^n=\hat M^n/\Gamma$, where 
$\Gamma\subset{\rm{Isom}}(\hat M^n)$ is a 
finite group acting freely on the rank-one symmetric space 
$\hat M^n$ (see Theorem 2.3.16 in \cite{W}). 
The quaternionic projective space $\HP^{n/4}$ 
and the Cayley plane $\OP^2$ admit no nontrivial 
quotients, even topologically (see \cite{BR-GB} 
or \cite[p.~185]{Besse0}). 
The complex projective space $\CP^{n/2}$ endowed 
with the Fubini-Study metric has a unique $\Z_2$-quotient 
only if $n/2=2k-1$ is odd. In this case, the action arises from 
the anti-holomorphic involutive isometry 
$$
[z_0:z_1:\cdots:z_{2k-1}:z_{2k}]\mapsto[-z_1:\bar z_0
:\cdots:-z_{2k}:\bar z_{2k-1}], 
$$
which has no fixed points (see \cite[p.~135]{Besse0}). 
Hence $M^n$ is isometric to $\CP^{2k-1}/\Z_2$. 

Now suppose that $n$ is even and $\beta_2(M^n)>0$. 
Since $\hat M^n$ is locally irreducible, the same 
holds for $M^n$. Then the result follows from 
Theorem \ref{ceven}. 
\qed

\subsection{Submanifolds with infinite fundamental group}

This section is devoted to the study of submanifolds 
with infinite fundamental group that satisfy the 
pinching condition \eqref{1}. More precisely, 
we prove the following theorem.

\begin{theorem}\label{p1ifn5}
Let $f\colon M^n\to\accentset{\sim}{M}^{n+m}$, 
$n\geq 5$, be an isometric immersion of a 
compact Riemannian manifold with infinite 
fundamental group. Assume that the inequality 
\eqref{1} is satisfied and 
$\accentset{\sim}{K}^f_{\min}\geq0$ at every 
point. Then one of the following holds: 
\vspace{1ex}

\noindent $(i)$ The universal cover of $M^n$ 
is isometric to a Riemannian product 
$\R\times N$, where $N$ is diffeomorphic to 
$\Sf^{n-1}$ and has nonnegative isotropic 
curvature.
\vspace{1ex}

\noindent $(ii)$ Equality holds in \eqref{1}, 
$
\accentset{\sim}{K}^f_{\min}=\accentset{\sim}{K}^f_{\max}=0
$ 
at every point, and one of the following occurs: 
\vspace{0.5ex}
\begin{enumerate}[topsep=1pt,itemsep=1pt,partopsep=1ex,parsep=0.5ex,leftmargin=*, label=(\roman*), align=left, labelsep=-0.5em]
\item [(a)] The manifold $M^n$ is a quotient 
$(\R^2\times\Sf^{n-2}(r))/\Gamma$, where 
$\Gamma$ is a discrete, fixed-point-free, cocompact 
subgroup of the isometry group of the standard 
Riemannian product $\R^2\times\Sf^{n-2}(r)$. 
Moreover, $f$ has index of relative nullity $2$, 
and $\delta=n\mathcal{H}_f/(n-2)$ 
is a Dupin principal normal vector field of $f$ with 
multiplicity $n-2$, satisfying $E_\delta=\mathcal D^\perp_f$.
\item [(b)] 
$M^n$ is flat, and $f$ is totally geodesic. 
\end{enumerate}
\end{theorem}
\proof
By Proposition \ref{propu}, the manifold $M^n$ 
has nonnegative isotropic curvature. We claim 
that $M^n$ is locally reducible. Suppose, 
to the contrary, that $M^n$ is locally irreducible. 
Then one of the cases $(i)$--$(iii)$ in Theorem 
\ref{acta} applies. In each of these cases, 
the universal cover $\hat M^n$ is compact, 
which contradicts the assumption on the 
fundamental group. 

Hence, $M^n$ is locally reducible. It then follows 
from Theorem \ref{duke} that the universal cover 
$\hat M^n$ is isometric to a Riemannian product 
as in cases $(a)$ or $(b)$ of that theorem. 

We claim that case $(b)$ is excluded. Suppose, 
to the contrary, that $\hat M^n$ is isometric to a 
Riemannian product 
$(\varSigma^2,g_\varSigma)\times(N^{n-2},g_N)$,
where $\varSigma^2$ is a surface whose Gaussian 
curvature is negative at some point, and $(N,g_N)$ 
is a compact irreducible Riemannian manifold with 
positive sectional curvature. Clearly, the isometric 
immersion $\hat f=f\circ\pi$ satisfies condition \eqref{1}, 
where $\pi\colon\hat M^n\to M^n$ is the covering 
map. By part $(ii)$ of Proposition \ref{propprod}, 
it then follows that $(\varSigma^2,g_\varSigma)$ 
must be flat, which is a contradiction. 

Hence, $\hat M^n$ is isometric to a Riemannian 
product 
$$
(\R^{n_0},g_0)\times(N_1^{n_1},g_1)\times\cdots\times(N_r^{n_r},g_r),
$$ 
where $n_0\geq1$, $g_0$ is the flat Euclidean metric, and for each $1\leq i\leq r$, 
either $n_i=2$ and $N_i=\Sf^2$ has nonnegative Gaussian curvature, 
or else $n_i\geq3$ and $(N_i,g_i)$ is compact and irreducible. 
We distinguish two cases.
\vspace{0.5ex}

\indent\emph{Case I}. Suppose that there exists 
a point at which $M^n$ has positive isotropic 
curvature. Clearly, $r\geq1$ in this case. 
Since $M^n$ has nonnegative isotropic 
curvature, it follows from Remark (iv) in \cite{ses} 
that $M^n$ admits a metric of positive isotropic 
curvature. By the result of Moore and Micallef 
\cite{MM}, the higher homotopy 
groups of $M^n$ satisfy 
\be\label{homotop}
\pi_i(M^n)=0\quad\text{for all }\;2\leq i\leq n/2. 
\ee

First, we claim that $n_i>2$ for all $1\leq i\leq r$. 
Arguing by contradiction, and without loss of generality,
suppose that $n_1=2$. It is known (see Proposition 4.1 
in \cite{Hat}) that covering maps induce isomorphisms 
on all higher homotopy groups. Hence, from
$
\pi_2(\hat M^n)\cong\pi_2(M^n) 
$ 
and \eqref{homotop}, it follows that 
$
\pi_2(\hat M^n)=0.
$
On other other hand, we have
$$
\pi_2(\hat M^n)\cong\pi_2(N_1)\oplus\cdots\oplus \pi_2(N_r)
\cong\Z\oplus\pi_2(N_2)\oplus\cdots\oplus\pi_2(N_r),
$$
which is a contradiction. Hence, $n_i\geq3$ for all 
$1\leq i\leq r$.

Next, we claim that $r=1$. Suppose to the contrary that 
$r\geq2$. It is easy to see that $n_i\leq n/2$ for some 
$1\leq i\leq r$. Without loss of generality, we may 
suppose that $n_1\leq n/2$. Then, by 
\eqref{homotop}
$$
\pi_k(\hat M^n)\cong\pi_k(M^n)=0\quad\text{for all }\;2\leq k\leq n_1. 
$$
From 
$$
\pi_k(\hat M^n)\cong\pi_k(N_1)\oplus\cdots\oplus \pi_k(N_r),
$$
it follows that 
$$
\pi_k(N_1)=0\quad\text{for all }\; 2\leq k\leq n_1. 
$$
Since $N_1$ is simply connected, the above vanishing result 
together with the Hurewicz isomorphism theorem implies that 
the homology groups of $N_1$ satisfy
$$
H_k(N_1;\Z)=0\quad\text{for all }\; 1\leq k\leq n_1.
$$
This contradicts the fact that 
$H_{n_1}(N_1;\Z)\cong\Z$. 

Hence, $\hat M^n$ is isometric to a Riemannian 
product $(\R^{n_0},g_0)\times(N_1^{n_1},g_1)$, 
where $n_0\geq1$ and $n_1\geq3$. We claim that 
$n_1>n/2$. Suppose, to the contrary, that $n_1\leq n/2$. 
From \eqref{homotop} we have 
$$
\pi_k(\hat M^n)\cong\pi_k(M^n)=0\quad\text{for all }\; 2\leq k\leq n_1, 
$$
and thus 
$$
\pi_k(N_1)=0\quad \text{for }\; 2\leq k\leq n_1. 
$$
As above, this contradicts the fact that $H_{n_1}(N_1;\Z)\cong\Z$.

Therefore, $n_1>n/2$ and consequently $n_0\leq n/2$. We 
now claim that $n_1\geq n-2$. Suppose, to the contrary, that 
$n_1<n-2$. Then $n_0>2$, and part $(i)$ of Proposition 
\ref{propprod} applies to the isometric immersion $\hat f$. 
In particular, $\hat M^n$ is flat, contradicting the assumption 
that there exists a point at which $M^n$ has 
positive isotropic curvature.
Hence, $n_1\geq n-2$. 

We now show that 
$n_1= n-1$. Suppose, to the contrary, that 
$n_1=n-2$. In this case, part $(ii)$ of Proposition 
\ref{propprod} applies to the immersion $\hat f$, and 
consequently $N_1$ has constant nonnegative sectional 
curvature. This contradicts the fact that there exists a point 
at which $M^n$ has positive isotropic curvature. 
Therefore, $n_1=n-1$, and the compact, irreducible, simply 
connected Riemannian manifold $N_1$ has nonnegative 
isotropic curvature. Then $N_1$ fails into one of the cases 
$(i)$--$(iii)$ of Theorem \ref{acta}. On the other hand, by 
assumption, there exists a point where $N_1$ 
has positive isotropic curvature. By the result of 
Moore and Micallef \cite{MM}, $N_1$ is homeomorphic 
to $\Sf^{n-1}$. Therefore, $N_1$ must be diffeomorphic 
to $\Sf^{n-1}$.
\vspace{0.5ex}

\indent\emph{Case II}. Now suppose that $M^n$ has no 
point of positive isotropic curvature. Then, by part $(ii)$ 
of Proposition \ref{propu}, the inequality \eqref{1} holds 
as an equality everywhere. 

We distinguish two subcaces:
\vspace{0.5ex}

\indent\emph{Subcase $II_a$}.
Suppose that $r\geq1$. We claim that 
$n_i\geq n-2$ for some $0\leq i\leq r$. 
Suppose, to the contrary, that $n_i<n-2$ 
for all $0\leq i\leq r$. Then part $(i)$ of Proposition 
\ref{propprod} implies that $\hat M^n$ must be flat. 
This contradicts the fact that each $N_i$ is a compact 
and simply connected. Therefore, $r=1$ and 
$$
(n_0,n_1)=(n-2,2),\quad (n_0,n_1)=(2,n-2),\quad\text{or}\quad (n_0,n_1)=(1,n-1).
$$ 

The case $(n_0,n_1)=(n-2,2)$ cannot occur. Otherwise, 
part $(ii)$ of Proposition \ref{propprod} would imply that 
the surface $N_1$ is flat, contradicting the fact that $N_1$ is 
diffeomorphic to $\Sf^2$. 

Suppose that $n_1=n-2$. In this case, part $(ii)$ of Proposition 
\ref{propprod} applies to the immersion $\hat f$. Hence, 
$\accentset{\sim}{K}^f_{\min}=\accentset{\sim}{K}^f_{\max}=0$ 
at every point, and $N_1$ has constant nonnegative sectional 
curvature. Since $N_1$ is compact and simply connected, it must be 
isometric to a sphere $\Sf^{n-2}(R)$. Then $M^n$ is a quotient 
$(\R^2\times\Sf^{n-2}(R))/\Gamma$, where $\Gamma$ is a discrete, 
fixed-point-free, cocompact subgroup of the isometry group of the 
standard Riemannian product $\R^2\times\Sf^{n-2}(R)$ (see Theorem 
2.3.16 in \cite{W}). Moreover, $f$ has relative nullity distribution 
$\pi_*(T\R^2)$ and $\delta=n\mathcal{H}_f/(n-2)$ 
is a Dupin principal normal vector field of $f$ with 
$E_\delta= \pi_*(T\Sf^{n-2}(R))$. 

Now suppose that $n_1=n-1$; that is, $\hat M^n$ is isometric 
to a Riemannian product 
$
\R\times N_1^{n-1},
$ 
where $N_1$ is a compact, irreducible, simply connected 
manifold. Fix a point $x_0=(t_0,y_0)\in\hat M^n$, and 
consider the curve $c=\hat f\circ\sigma_0$ and the 
isometric immersion 
$$
f_1=\hat f\circ\sigma_1\colon N_1\to\accentset{\sim}{M}^{n+m},
$$ 
where $\sigma_0\colon\R\to\hat M^n$ and $\sigma_1\colon N_1\to\hat M^n$ 
denote the totally geodesic inclusions defined by 
$$
\sigma_0(t)=(t,y_0)\quad \text{and}\quad \sigma_1(y)=(t_0,y),\quad y\in N_1.
$$

Then Lemma \ref{lemmaprod} applies to the immersion $\hat f$, 
and \eqref{pro} simplifies to
\begin{align}\label{pro1}
\frac{(n-3)}{n-2}k_1^2&+
\frac{2(n-3)}{n-2}\sum_{\ell=1}^{n-1}\|\a_{\hat f}
\big(\sigma_0{_*}\frac{\partial}{\partial t},\sigma_1{_*}e_\ell\big)\|^2\nonumber\\
&\leq-\frac{4}{3}\accentset{\sim}{K}^{\hat f}_{\max}+
\frac{2(5n-13)}{3(n-2)}
\accentset{\sim}{K}^{\hat f}_{\min}+\frac{(n-1)^2}{n-2}H_1^2-S_1\nonumber\\
&\leq-\frac{4}{3}\accentset{\sim}{K}^{\hat f}_{\max}+
\frac{2(5n-13)}{3(n-2)}
\accentset{\sim}{K}^{\hat f}_{\min}
+\frac{(n-1)^2}{n-3}H_1^2-S_1,
\end{align} 
where $\{e_\ell\}_{\ell=1}^{n-1}$ is an orthonormal frame 
of $TN_1$. 
Since, $\accentset{\sim}{K}^{\hat f}_{\min}\geq0$ we have
\be\label{pro2}
-\frac{4}{3}\accentset{\sim}{K}^{\hat f}_{\max}+
\frac{2(5n-13)}{3(n-2)}
\accentset{\sim}{K}^{\hat f}_{\min}
\leq\frac{16}{3}\big(\accentset{\sim}{K}^{\hat f}_{\min}
-\frac{1}{4}\accentset{\sim}{K}^{\hat f}_{\max}\big).
\ee 
Hence, \eqref{pro1} and \eqref{pro2} yield 
\begin{align}\label{pro3}
\frac{(n-3)}{n-2}k_1^2+&
\frac{2(n-3)}{n-2}\sum_{\ell=1}^{n-1}\|\a_{\hat f}
\big(\sigma_0{_*}\frac{\partial}{\partial t},\sigma_1{_*}e_\ell\big)\|^2\nonumber\\
&\leq\frac{16}{3}\big(\accentset{\sim}{K}^{\hat f}_{\min}
-\frac{1}{4}\accentset{\sim}{K}^{\hat f}_{\max}\big)
+\frac{(n-1)^2}{n-3}H_1^2-S_1\nonumber\\
&\leq\frac{16}{3}\big(\accentset{\sim}{K}^{f_1}_{\min}
-\frac{1}{4}\accentset{\sim}{K}^{f_1}_{\max}\big)
+\frac{(n-1)^2}{n-3}H_1^2-S_1. 
\end{align} 

It follows from \eqref{pro3} that the immersion 
$f_1\colon N_1\to\accentset{\sim}{M}^{n+m}$
satisfies condition \eqref{1}. Then, by Theorem 
\ref{p1fn5}, either $N_1$ is diffeomorphic to 
$\Sf^{n-1}$, or $f_1$ is totally geodesic satisfying 
condition \eqref{1} with equality at every point and, 
$N_1$ is as described in part $(ii)$ of that theorem. 
We claim that the latter case cannot occur. Suppose 
otherwise. Then \eqref{pro1}--\eqref{pro3} hold as 
equalities at every point. Hence, the curve $c$ is a 
geodesic, and 
$
\accentset{\sim}{K}^{\hat f}_{\min}=
\accentset{\sim}{K}^{\hat f}_{\max}=0
$ 
everywhere. The Gauss 
equation then implies that $N_1$ is flat. This 
contradicts the fact that $N_1$ is compact and 
simply connected, thereby proving the claim that 
$N_1$ is diffeomorphic to $\Sf^{n-1}$. 
\vspace{0.5ex}

\indent\emph{Subcase $II_b$}.
Suppose that $r=0$, namely that $\hat M^n$ is isometric to $\R^n$. 
From part $(ii)$ of Proposition \ref{propu}, it follows that at 
each point $x\in M^n$ there exists an orthonormal four-frame 
$\{e_i\}_{i=1}^4$ such that all conditions \eqref{c1}--\eqref{c5} 
hold. 

The Gauss equation, together with \eqref{c1} 
and \eqref{c22}, implies that the sectional curvatures 
satisfy 
\begin{align*}
K(e_1\wedge e_2)&=\|\a_{11}\|^2+\accentset{\sim}{R}^f_{1221}(x),\quad K(e_3\wedge e_4)=\|\a_{44}\|^2+\accentset{\sim}{R}^f_{3443}(x),\\
K(e_1\wedge e_3)&=\|\a_{14}\|^2+
\frac{1}{3}\big(\accentset{\sim}{K}^f_{\max}(x)-\accentset{\sim}{K}^f_{\min}(x)\big),\\
K(e_1\wedge e_4)&=\|\a_{13}\|^2+
\frac{1}{3}\big(\accentset{\sim}{K}^f_{\max}(x)-\accentset{\sim}{K}^f_{\min}(x)\big). 
\end{align*}
Since $M^n$ is flat, the above equalities yield $\a_{13}=\a_{14}=0$ 
and $\accentset{\sim}{K}^f_{\max}(x)
=\accentset{\sim}{K}^f_{\min}(x)$. By the last two equations in \eqref{c1}, 
we also have $\a_{23}=\a_{24}=0$. 
Hence, 
$$
\|\a_{11}\|^2=\|\a_{44}\|^2=-\accentset{\sim}{K}^f_{\min}(x).
$$
Since by assumption $\accentset{\sim}{K}^f_{\min}\geq0$, 
we obtain $\accentset{\sim}{K}^f_{\min}(x)=0$ and 
$\a_{11}=\a_{44}=0$. It then follows from \eqref{c2} that 
$\delta=0$, and thus $f$ is totally geodesic. 
\qed
\vspace{1.5ex}

\noindent\emph{Proof of Theorem \ref{n5}:} 
The result follows directly from Theorems \ref{ceven}, 
\ref{p1fn5} and \ref{p1ifn5}.\qed

\section{Proof of the result for four-submanifolds}

We now focus on four-dimensional submanifolds. 

\subsection{Geometry of four-dimensional manifolds}
In this section, we collect some basic facts about four-dimensional 
geometry. For a detailed exposition of the subject, we refer 
the reader to \cite{LeB}. 

The bundle of 2-forms of any oriented four-dimensional 
Riemannian manifold $(M,\<\cdot,\cdot\>)$ decomposes as a direct sum 
$$
\Omega^2(M)=\Omega_+^2(M)\oplus\Omega_-^2(M)
$$
of the eigenspaces of the Hodge star operator 
$\ast\colon\Omega^2(M)\to\Omega^2(M)$. 
The sections of $\Omega_+^2(M)$ are called 
{\emph {self-dual 2-forms}}, whereas those of 
$\Omega_-^2(M)$ are called 
{\emph {anti-self-dual 2-forms}}. 
Accordingly, at every point $x\in M^n$, we have the splitting
$$
\Lambda^2T_xM=\Lambda_+^2T_xM\oplus\Lambda_-^2T_xM, 
$$
where $\Lambda_\pm^2T_xM$ 
are the eigenspaces of the Hodge star operator 
$
\ast\colon\Lambda^2T_xM\to\Lambda^2T_xM.
$ 
Both spaces $\Lambda_+^2T_xM$ and $\Lambda_-^2T_xM$ 
are $\B^{[2]}$-invariant (see Proposition 1 in \cite{smichig}). 
Hence, we obtain the corresponding decomposition 
$$
\B^{[2]}=\B_+^{[2]}\oplus\B_-^{[2]}. 
$$

Suppose now that $M$ is a compact, 
oriented Riemannian four-manifold. The Hodge theorem 
guarantees that every de Rham cohomology class on $M$ 
has a unique harmonic representative. In particular, the space 
$\mathscr H^2(M)$ of harmonic 2-forms decomposes as 
$$
\mathscr H^2(M)=\mathscr H^2_+(M)\oplus\mathscr H^2_-(M),
$$ 
where $\mathscr H^2_+(M)$ and $\mathscr H^2_-(M)$ denote 
the spaces of self-dual and anti-self-dual harmonic 2-forms, 
respectively. The dimensions of these subspaces, 
$$
\beta_\pm(M)=\dim\mathscr H^2_\pm(M),
$$ 
are oriented homotopy invariants of $M$. Their difference 
$
\sigma(M)=\beta_+(M)-\beta_-(M)
$ 
is the signature of $M$, 
while their sum equals the second Betti number 
$\beta_2(M)$ of the manifold. 

\subsection{Auxiliary results for four-dimensional submanifolds}

We now state the following auxiliary results.

\begin{lemma}\label{nnic1}
Let $f\colon M^4\to\tilde M^{4+m}$ be an 
isometric immersion of an oriented four-dimensional 
Riemannian manifold $M^4$. Suppose that at a point 
$x\in M^4$, there exists an oriented 
orthonormal four-frame $\{e_i\}_{i=1}^4\subset T_xM$ 
such that \eqref{eq1}, \eqref{c1}--\eqref{c5} and 
\eqref{perp} hold at $x$. 
We consider the orthonormal basis 
$\{\eta_i\}_{1\leq i\leq6}$ of the space 
of 2-vectors $\Lambda^2T_xM$, satisfying 
$\eta_i\in\Lambda_+^2T_xM$ and 
$\eta_{i+3}\in\Lambda_-^2T_xM$ for $1\leq i\leq 3$, defined by 
\begin{align*}
\eta_1=\frac{1}{\sqrt{2}}(e_{12}+e_{34}),\quad
\eta_2&=\frac{1}{\sqrt{2}}(e_{13}-e_{24}),
\quad\eta_3=\frac{1}{\sqrt{2}}(e_{14}+e_{23}),\\
\eta_4=\frac{1}{\sqrt{2}}(e_{12}-e_{34}),\quad
\eta_5&=\frac{1}{\sqrt{2}}(e_{13}+e_{24}),\quad
\eta_6=\frac{1}{\sqrt{2}}(e_{14}-e_{23}),
\end{align*}
where $e_{ij}=e_i\wedge e_j$. 
Then the following assertions hold at $x$:
\vspace{0.5ex}

\noindent $(i)$ The matrix of the Bochner--Weitzenb\"ock 
operator $\B^{[2]}=\B_+^{[2]}\oplus\B_-^{[2]}$ at the point $x$, with respect to the 
basis $\{\eta_i\}_{1\leq i\leq6}$, is given by the direct sum 
\be
\begin{bmatrix}
\mu^+_1&a^+_1&a^+_2\\
a^+_1&\mu^+_2&0\\
a^+_2&0&\mu^+_3&
\!\!\!\!\!\end{bmatrix}
\oplus
\begin{bmatrix}
\mu^-_1&a^-_1&a^-_2\\
a^-_1&\mu^-_2&0\\
a^-_2&0&\mu^-_3&
\!\!\!\!\!\end{bmatrix},\nonumber
\ee
where
\begin{align*}
\mu^\pm_1=\;&2(1\mp\varepsilon_x)\big(\frac{2}{3}(\accentset{\sim}{K}^f_{\max}(x)
-\accentset{\sim}{K}^f_{\min}(x))+\|\a_{13}\|^2+\|\a_{14}\|^2\big),\\
\mu^\pm_2=\;&\accentset{\sim}{R}^f_{1221}+\accentset{\sim}{R}^f_{3443}
+\frac{2}{3}\big((2\pm\varepsilon_x)\accentset{\sim}{K}^f_{\max}(x)
-(5\pm\varepsilon_x)\accentset{\sim}{K}^f_{\min}(x)\big)\\
\;&+\|\a_{11}-\a_{44}\|^2+4\|\a_{13}\|^2+2(1\pm\varepsilon_x)\|\a_{14}\|^2,\\
\mu^\pm_3=\;&\accentset{\sim}{R}^f_{1221}+\accentset{\sim}{R}^f_{3443}
+\frac{2}{3}\big((2\mp\varepsilon_x)\accentset{\sim}{K}^f_{\max}(x)
-(5\mp\varepsilon_x)\accentset{\sim}{K}^f_{\min}(x)\big)\\
\;&+\|\a_{11}-\a_{44}\|^2+4\|\a_{14}\|^2+2(1\pm\varepsilon_x)\|\a_{13}\|^2,\\
a^\pm_1=\;&-(\varepsilon_x\mp1)\<\a_{14},\a_{44}-\a_{11}\>,
\;\,a^\pm_2=(\varepsilon_x\mp1)\<\a_{13},\a_{44}-\a_{11}\>, 
\end{align*}
with $\varepsilon_x=1$ or $\varepsilon_x=-1$.
\vspace{1ex}

\noindent $(ii)$ If $\ker \B_\pm^{[2]}\neq 0$ at $x$, then 
either $\varepsilon_x=\pm1$, or, in the case 
$\varepsilon_x=\mp1$, the following conditions hold:
\begin{align}
\;\accentset{\sim}{K}^f_{\max}(x)
&=\accentset{\sim}{K}^f_{\min}(x)\quad\text{and}\quad\a_{13}=0\quad\text{or}\quad \a_{14}=0,\tag{ii1}\label{ii1}\\
\a_{44}-\a_{11}&=2\rho\,\a_{13}\quad\text{for some}\quad \rho\in\R\quad\text{if}\quad \a_{13}\neq0,\tag{ii2}\label{ii2}\\
\a_{44}-\a_{11}&=2\rho\,\a_{14}\quad\text{for some}\quad \rho\in\R\quad \text{if}\quad \a_{14}\neq0.\tag{ii3}\label{ii3}
\end{align}
\end{lemma} 
\proof 
\noindent $(i)$ This follows from a straightforward 
computation using \eqref{WB}, the Gauss equation, 
\eqref{eq1}, \eqref{c1}--\eqref{c5}, and \eqref{perp}.
\vspace{1ex}

$(ii)$ Using part $(i)$, we obtain 
$$
\det\B_\pm^{[2]}=\mu^\pm_1\mu^\pm_2\mu^\pm_3-\mu^\pm_2(a_2^\pm)^2-\mu^\pm_3(a_1^\pm)^2.
$$
Hence, if $\ker \B_\pm^{[2]}\neq0$ at $x$, then
$$
\mu^\pm_1\mu^\pm_2\mu^\pm_3=\mu^\pm_2(a_2^\pm)^2+\mu^\pm_3(a_1^\pm)^2,
$$
or equivalently 
\begin{align*}
(1\mp\varepsilon_x)&\Big(\mu^\pm_2\mu^\pm_3\big(\frac{2}{3}(\accentset{\sim}{K}^f_{\max}(x)
-\accentset{\sim}{K}^f_{\min}(x))+\|\a_{13}\|^2+\|\a_{14}\|^2\big)\\
&-\mu^\pm_2\<\a_{13},\a_{44}-\a_{11}\>^2-\mu^\pm_3\<\a_{14},\a_{44}-\a_{11}\>^2\Big)
=0.
\end{align*}

Suppose now that $\varepsilon_x=\mp1$. Then 
\begin{align}\label{zero}
\mu^\pm_2\big(\mu^\pm_3\|\a_{13}\|^2-\<\a_{13},\a_{44}-&\a_{11}\>^2\big)
+\mu^\pm_3\big(\mu^\pm_2\|\a_{14}\|^2-\<\a_{14},\a_{44}-\a_{11}\>^2\big)
\nonumber\\
&+\frac{2}{3}\mu^\pm_2\mu^\pm_3\big(\accentset{\sim}{K}^f_{\max}(x)
-\accentset{\sim}{K}^f_{\min}(x)\big)=0.
\end{align}
Observe that 
\begin{align}
\mu^\pm_2\geq&\frac{2}{3}(2\pm\varepsilon_x)(\accentset{\sim}{K}^f_{\max}(x)
-\accentset{\sim}{K}^f_{\min}(x))+\|\a_{44}-\a_{11}\|^2\nonumber\\
&+4\|\a_{13}\|^2+2(1\pm\varepsilon_x)\|\a_{14}\|^2,\label{mu1}\\
\mu^\pm_3\geq&\frac{2}{3}(2\mp\varepsilon_x)(\accentset{\sim}{K}^f_{\max}(x)
-\accentset{\sim}{K}^f_{\min}(x))+\|\a_{44}-\a_{11}\|^2\nonumber\\
&+4\|\a_{14}\|^2+2(1\pm\varepsilon_x)\|\a_{13}\|^2.\label{mu2}
\end{align}
\vspace{2ex}
\vspace*{-\baselineskip}
Using these two inequalities and the Cauchy-Schwarz 
inequality, we obtain 
\begin{align}
&\mu^\pm_2\big(\mu^\pm_3\|\a_{13}\|^2-\<\a_{13},\a_{44}-\a_{11}\>^2\big)
+\mu^\pm_3\big(\mu^\pm_2\|\a_{14}\|^2-\<\a_{14},\a_{44}-\a_{11}\>^2\big)
\nonumber\\
&\geq\mu^\pm_2\|\a_{13}\|^2\big(\mu^\pm_3-\|\a_{44}-\a_{11}\|^2\big)+\mu^\pm_3\|\a_{14}\|^2\big(\mu^\pm_2-\|\a_{44}-\a_{11}\|^2\big)\nonumber\\
&\geq\mu^\pm_2\|\a_{13}\|^2\big(\frac{2}{3}(2\mp\varepsilon_x)(\accentset{\sim}{K}^f_{\max}(x)
-\accentset{\sim}{K}^f_{\min}(x))+4\|\a_{14}\|^2+2(1\pm\varepsilon_x)\|\a_{13}\|^2 \big)\nonumber\\
&+\mu^\pm_3\|\a_{14}\|^2\big(\frac{2}{3}(2\pm\varepsilon_x)(\accentset{\sim}{K}^f_{\max}(x)-\accentset{\sim}{K}^f_{\min}(x))+4\|\a_{13}\|^2+2(1\pm\varepsilon_x)\|\a_{14}\|^2\big)\nonumber\\
&=2\mu^\pm_2\|\a_{13}\|^2\big((\accentset{\sim}{K}^f_{\max}(x)
-\accentset{\sim}{K}^f_{\min}(x))+2\|\a_{14}\|^2\big)\nonumber\\
&+2\mu^\pm_3\|\a_{14}\|^2\big(\frac{1}{3}(\accentset{\sim}{K}^f_{\max}(x)
-\accentset{\sim}{K}^f_{\min}(x))+2\|\a_{13}\|^2\big).\label{CS}
\end{align}
Therefore, \eqref{zero} implies
\begin{align}
\mu^\pm_2&\mu^\pm_3\big(\accentset{\sim}{K}^f_{\max}(x)-\accentset{\sim}{K}^f_{\min}(x)\big)=0,\label{zero1}\\
\mu^\pm_2&\|\a_{13}\|^2\big((\accentset{\sim}{K}^f_{\max}(x)
-\accentset{\sim}{K}^f_{\min}(x))+2\|\a_{14}\|^2\big)=0,\label{zero2}\\
\mu^\pm_3&\|\a_{14}\|^2\big(\frac{1}{3}(\accentset{\sim}{K}^f_{\max}(x)
-\accentset{\sim}{K}^f_{\min}(x))+2\|\a_{13}\|^2\big)=0.\label{zero3}
\end{align}

The first equality in \eqref{ii1} follows from \eqref{mu1}, \eqref{mu2} 
and \eqref{zero1}. The remaining statements in \eqref{ii1}, follow 
again from \eqref{mu1}, \eqref{mu2}, \eqref{zero2}, and \eqref{zero3}. 
Finally, \eqref{ii2} and \eqref{ii3} follow from \eqref{ii1} and \eqref{CS}, 
which now holds as an equality. This concludes 
the proof of part $(ii)$. 
\qed
\vspace{1ex}

The following result will be used in the sequel.

\begin{theorem}\label{n=4beta}
Let $f\colon M^4\to\accentset{\sim}{M}^{4+m}$ be an 
isometric immersion of a compact, oriented four-dimensional 
Riemannian manifold satisfying the inequality \eqref{1} 
at every point. Suppose that $M^4$ has nowhere positive 
isotropic curvature. If $\beta_+(M^4)>0$ or 
$\beta_-(M^4)>0$, then one of the following holds: 
\vspace{1ex}

\noindent $(i)$ $M^4$ is a K\"ahler manifold 
equipped with a complex structure for which $f$ is 
$(2,0)$-geodesic. Moreover, if both $\beta_+(M^4)$ 
and $\beta_-(M^4)$ are positive, then 
$\accentset{\sim}{K}^f_{\min}=
\accentset{\sim}{K}^f_{\max}$ 
at every point.

\vspace{1ex}

\noindent $(ii)$ $M^4$ is flat with first Betti number $\beta_1(M^4)\geq2$, 
and $f$ is totally umbilical with 
$
\accentset{\sim}{K}^f_{\max}=
\accentset{\sim}{K}^f_{\min}=-H^2
$ 
at every point.
\end{theorem}
\proof 
Proposition \ref{propeven} implies that the 
Bochner--Weitzenb\"ock operator $\B^{[2]}$ is 
nonnegative. Moreover, it follows from part $(ii)$ 
of Proposition \ref{propu} that equality 
holds in \eqref{1} at every point. Furthermore, 
at each point $x\in M^4$, there exists an oriented 
orthonormal four-frame 
$\{e^x_i\}_{i=1}^4\subset T_xM$ such that 
the second fundamental form $\a_f$ of $f$ 
satisfies conditions \eqref{c1}--\eqref{c5}. 

We will treat only the case where $\beta_+(M^4)>0$, as 
the other case can be handled in a similar 
manner. There exists a nontrivial self-dual harmonic 
2-form $\omega_+$. By the Bochner--Weitzenb\"ock 
formula, the form $\omega_+$ is parallel, 
and 
$
\<\B^{[2]}(\omega_+),\omega_+\>=0
$ 
at every point. Hence, $\ker \B_+^{[2]}\neq0$ at every 
point, and part $(ii)$ of Lemma \ref{nnic1} applies. Since this form is parallel, 
we may normalize it (after multiplying by a 
constant, if necessary) so that $\|\omega_+\|=\sqrt{2}$. 
Let $Z$ be the dual to the self-dual 
form $\omega_+$. 

Observe that if furthermore $\beta_-(M^4)>0$, then 
$\ker\mathcal B_-^{[2]}\neq0$ at each point. It 
then follows from part $(ii)$ of Lemma \ref{nnic1} that 
$\accentset{\sim}{K}^f_{\min}=
\accentset{\sim}{K}^f_{\max}$ 
at every point.

We consider the subset 
$$
M_1=\left\{x\in M^4: \mu^+_1(x)=0\quad \text{and}
\quad \mu^+_2(x)\mu^+_3(x)>0\right\},
$$
where $\mu^+_1,\mu^+_2$, and $\mu^+_3$ are as 
given in Lemma \ref{nnic1}. It follows from part $(i)$ 
of Lemma \ref{nnic1} that 
\be\label{ker1+}
\ker\B_+^{[2]}(x)=
\spa\left\{e^x_1\wedge e^x_2+e^x_3\wedge e^x_4\right\}
\quad\text{for all }\; x\in M_1.
\ee

Clearly, for any $x\in M\smallsetminus M_1$, either
\be\label{00}
\mu^+_1(x)=0\quad \text{and}\quad \mu^+_2(x)\mu^+_3(x)=0,
\ee
or
\be\label{+0}
\mu^+_1(x)>0. 
\ee
It then follows from Lemma \ref{nnic1} that 
\be\label{min=max}
\accentset{\sim}{K}^f_{\min}(x)=\accentset{\sim}{K}^f_{\max}(x)\quad 
\text{for all }\; x\in M\smallsetminus M_1. 
\ee 
Moreover, at points $x\in M\smallsetminus M_1$ where 
\eqref{00} holds, the immersion $f$ is totally umbilical. 
In particular, \eqref{c22} yields
$\accentset{\sim}{K}^f_{\min}(x)=-H^2(x)\leq0$. 

At points $x\in M\smallsetminus M_1$ where \eqref{+0} 
holds, it follows again from part $(ii)$ of Lemma \ref{nnic1} 
that $\varepsilon_x=-1$ and that exactly one of the vectors 
$\a^x_{13}$ and $\a^x_{14}$ is nonzero. We claim that, 
for all points $x\in M\smallsetminus M_1$, the kernel of 
$\B_+^{[2]}(x)$ is given by 
\be\label{kerr}
\ker\B_+^{[2]}(x)=
\begin{cases}
\spa\left\{\rho_x\eta^x_1+\eta^x_3\right\}
\,&\text{if}\;\,\a^x_{13}\neq0,\\
\spa\left\{-\rho_x\eta^x_1+\eta^x_2\right\}
\,&\text{if}\;\,\a^x_{14}\neq0,
\end{cases}
\ee
where $\rho^x$ is the number appearing in 
\eqref{ii2}--\eqref{ii3}. 
Indeed, it follows from part $(i)$ 
of Lemma \ref{nnic1} that the matrix of the 
Bochner--Weitzenb\"ock operator $\B_+^{[2]}$ at 
$x$, with respect to the basis 
$\{\eta_i\}_{i=1}^3$ of $\Lambda_+^2T_xM$, 
is 
\be\label{ssf}
4\|\a^x_{13}\|^2\begin{bmatrix}
1&0&-\rho_x\\
0&\rho^2+1&0\\
-\rho_x&0&\rho_x^2&
\!\!\!\!\!\end{bmatrix}\quad
\text{or}\quad
4\|\a^x_{14}\|^2\begin{bmatrix}
1&\rho_x&0\\
\rho_x&\rho^2&0\\
0&0&\rho_x^2+1&
\!\!\!\!\!\end{bmatrix},
\ee
depending on whether $\a^x_{13}\neq0$ or $\a^x_{14}\neq0$, 
respectively. Therefore, \eqref{kerr} follows. 

Now we distinguish two cases.
\vspace{0.5ex}

\indent\emph{Case I}. Suppose that the subset $M_1$ 
is not empty, and fix a point $x_0\in M_1$. Because of 
\eqref{ker1+}, we may assume 
that 
$$
Z_{x_0}=e^{x_0}_1\wedge e^{x_0}_2+e^{x_0}_3\wedge e^{x_0}_4. 
$$
Consider the almost complex structure 
$J_{x_0}\colon T_{x_0}M\to T_{x_0}M$ defined by
$$
J_{x_0}e^{x_0}_1=e^{x_0}_2,\quad J_{x_0}e^{x_0}_3=e^{x_0}_4.
$$ 
Then, for all $v,w\in T_{x_0}M$, 
$$
\omega_+(v,w)=\<J_{x_0}v,w\>. 
$$

Define the skew-symmetric 
endomorphism $J$ of the tangent bundle 
of $M^4$ such that 
$$
\omega_+(X,Y)=\<JX,Y\>\quad\text {for all } X,Y\in\mathcal X(M). 
$$
Clearly, $J$ is parallel because $\omega_+$ is parallel. 

We claim that $J$ is orthogonal; that is, $\|J_xv\|=\|v\|$ 
for every point $x\in M^4$ and every $v\in T_xM$. Indeed, let 
$V$ be a parallel vector field along a curve $c\colon[0,1]\to M$ 
such that $c(0)=x,c(1)=x_0$ and $V(0)=v$. Then $W=JV$ is 
also parallel along $c$. Using the fact that $J_{x_0}$ is 
orthogonal, we obtain 
$$
\|J_xv\|=\|W(0)\|=\|W(1)\|=\|J_{x_0}V(1)\|=\|V(1)\|=\|V(0)\|=\|v\|,
$$
which proves the claim. Since $J$ is both 
skew-symmetric and orthogonal, it defines a parallel 
almost complex structure. Hence, the 
triple $(M^4,\<\cdot,\cdot\>,J)$ is a K\"ahler manifold. 

We now show that at every point $x\in M^4$, 
\be\label{Jx}
\a_f(J_xv,J_xw)=\a_f(v,w)\quad\text {for all } v,w\in T_xM.
\ee
To prove this, we consider two subcases.
\vspace{0.5ex}

\indent\emph{Subcase $I_a$}. Let $x$ be an arbitrary point 
in $M_1$. Since by \eqref{ker1+} the kernel of $\B_+^{[2]}$ 
at $x$ is spanned by the vector 
$e^x_1\wedge e^x_2+e^x_3\wedge e^x_4$, 
and since $\|Z_x\|=\sqrt{2}$, it follows that 
$$
Z_x=\pm\left(e^x_1\wedge e^x_2+e^x_3\wedge e^x_4\right).
$$ 
Because $Z$ is the dual of the self-dual 
form $\omega_+$, we have 
$$
\<J_xv,w\>=\<\<Z_x,v\wedge w\>\>\quad \text {for all } v,w\in T_xM. 
$$
Consequently, 
\be\label{JJ}
J_xe^x_1=\pm e^x_2, \quad J_xe^x_3=\pm e^x_4
\ee
at each point $x\in M_1$. 
We now claim that 
\be\label{Jij}
\a_f(J_xe^x_i,J_xe^x_j)=\a_{ij}\quad\text {for all } 1\leq i,j\leq4. 
\ee
If $\varepsilon_x=1$, this follows directly from \eqref{c1}. 
If $\varepsilon_x=-1$, then part $(i)$ of Lemma \ref{nnic1} 
implies that $\a^x_{13}=\a^x_{14}=0$. 
Therefore, \eqref{Jij} again follows from \eqref{c1}, and 
\eqref{Jx} follows from \eqref{Jij} by linearity.
\vspace{0.5ex}

\indent\emph{Subcase $I_b$}. 
We now consider the complementary case 
$x\in M\smallsetminus M_1$. If \eqref{00} holds 
at $x$, then $f$ is totally umbilical and hence 
\eqref{Jx} holds trivially. 

Suppose now that \eqref{+0} holds at $x$. 
In this case, exactly one of the vectors $\a^x_{13}$ 
and $\a^x_{14}$ vanishes, and \eqref{kerr} holds at $x$. 
By \eqref{kerr}, the kernel of $\B_+^{[2]}$ 
at $x$ is spanned by one of the vectors 
$\rho_x\eta^x_1+\eta^x_3,-\rho_x\eta^x_1+\eta^x_2$, 
and the matrix of $\B_+^{[2]}$ 
at $x$, with respect to the basis $\{\eta_i\}_{i=1}^3$ 
of $\Lambda_+^2T_xM$, is given by \eqref{ssf}. 
Since $\|Z_x\|=\sqrt{2}$, it follows that 
$$
Z_x=
\begin{cases}
\lambda_x(\rho_x\eta^x_1+\eta^x_3)
\,&\text{if}\;\,\a^x_{13}\neq0,\\
\lambda_x(-\rho_x\eta^x_1+\eta^x_2)
\,&\text{if}\;\,\a^x_{14}\neq0,
\end{cases}
$$ 
where the scalar $\lambda_x$ satisfies
\be\label{lam}
\lambda_x^2=\frac{2}{\rho_x^2+1}. 
\ee
Because $Z$ is the dual of the self-dual 
form $\omega_+$, we have 
$$
\<J_xv,w\>=\<\<Z_x,v\wedge w\>\>\quad \text {for all } v,w\in T_xM. 
$$

Using this relation together with \eqref{lam}, we obtain 
\begin{align*}
J_xe^x_1&=\frac{\lambda_x}{\sqrt{2}} \left(\rho_x\,e^x_2+e^x_4\right), 
\quad J_xe^x_2=\frac{\lambda_x}{\sqrt{2}} \left(-\rho_x\,e^x_1+e^x_3\right),\\
J_xe^x_3&=\frac{\lambda_x}{\sqrt{2}}\left(-e^x_2+\rho\,e^x_4\right), 
\quad J_xe^x_4=-\frac{\lambda_x}{\sqrt{2}} \left(e^x_1+\rho_x\,e^x_3\right),
\end{align*}
if $\a^x_{13}\neq0$, or 
\begin{align*}
J_xe^x_1&=\frac{\lambda_x}{\sqrt{2}}\left(-\rho_x\,e^x_2+e^x_3\right), 
\quad J_xe^x_2=\frac{\lambda_x}{\sqrt{2}}\left(\rho_x\,e^x_1-e^x_4\right),\\
J_xe^x_3&=-\frac{\lambda_x}{\sqrt{2}}\left(e^x_1+\rho\,e^x_4\right), 
\quad J_xe^x_4=\frac{\lambda_x}{\sqrt{2}} \left(e^x_2+\rho_x\,e^x_3\right),
\end{align*}
if $\a^x_{14}\neq0$. 
From these expressions, and combining \eqref{c1} with 
\eqref{ii2}--\eqref{ii3}, we obtain \eqref{Jij}. Hence \eqref{Jx} 
follows directly, and thus $f$ is $(2,0)$-geodesic. 

\iffalse
Finally we claim that 
$$
\accentset{\sim}{K}^f_{\min}(x)=\accentset{\sim}{K}^f_{\max}(x)\quad 
\text{for all } x\in M_1. 
$$
Since $(M^4,\<\cdot,\cdot\>,J)$ is a K\"ahler manifold, 
its curvature tensor satisfies
$$
\<R(JX,JY)Z,W\>=\<R(X,Y)Z,W\>\quad\text{for all } X,Y,Z,W\in TM. 
$$
Then the Gauss equation yields 
\begin{align}\label{RJ}
\<\a_f(JX,W),\a_f(JY,Z)\>&-\<\a_f(JX,Z),\a_f(JY,W)\>\nonumber\\
-\<\a_f(X,W),\a_f(Y,Z)\>&+\<\a_f(X,Z),\a_f(Y,W)\>\nonumber\\
=\<\accentset{\sim}{R}^f(X,&Y)Z,W\>-\<\accentset{\sim}{R}^f(JX,JY)Z,W\> 
\end{align}
for all $X,Y,Z,W\in TM$. Using \eqref{c1}, \eqref{c22}--\eqref{c4} 
and \eqref{JJ}, it follows from \eqref{RJ} for 
$X=e^x_1,Y=e^x_4,Z=e^x_2,W=e^x_3\in T_xM$ 
that 
$$
(1+\varepsilon_x)\big(\accentset{\sim}{K}^f_{\min}(x)
-\accentset{\sim}{K}^f_{\max}(x)\big)=
(1-\varepsilon_x)\|\a^x_{13}\|^2.
$$
This implies 
$
\accentset{\sim}{K}^f_{\min}(x)=\accentset{\sim}{K}^f_{\max}(x)
$
at points $x\in M_1$ with $\varepsilon_x=1$. The same 
holds at points where $\varepsilon_x=-1$ by part $(ii)$ 
of Lemma \ref{nnic1}. Together with \eqref{min=max}, 
this equality therefore holds at every point of $M^4$.
\fi 
\vspace{0.5ex}

\indent\emph{Case II}. Suppose that $M_1$ is empty. 
By \eqref{min=max}, we have 
$
\accentset{\sim}{K}^f_{\min}=\accentset{\sim}{K}^f_{\max}
$
at every point. Then, at each point $x\in M^4$, either \eqref{00} 
holds, in which case $f$ is totally umbilical, or \eqref{+0} holds. 
In the latter case, by part $(ii)$ of Lemma \ref{nnic1} exactly 
one of the vectors $\a^x_{13}$ and $\a^x_{14}$ 
vanishes and \eqref{kerr} holds at $x$. 

We distinguish two subcases:
\vspace{0.5ex}

\indent\emph{Case $II_a$}. 
Suppose that there exists a point $x_0\in M^4$ such 
that $\mu_1^+(x_0)>0$. Then exactly one of the vectors 
$\a^{x_0}_{13}$ and $\a^{x_0}_{14}$ vanishes, and 
\eqref{kerr} holds at $x_0$. 
The kernel of $\B_+^{[2]}$ 
at $x_0$ is spanned by one of the vectors 
$\rho_{x_0}\eta^{x_0}_1+\eta^{x_0}_3$, 
$-\rho_{x_0}\eta^{x_0}_1+\eta^{x_0}_2$, and the 
matrix of $\B_+^{[2]}$ at $x_0$, with respect to the 
basis $\{\eta^{x_0}_i\}_{i=1}^3$ of 
$\Lambda_+^2T_{x_0}M$, is given by \eqref{ssf}. 
Since $\|Z_{x_0}\|=\sqrt{2}$, it follows that 
$$
Z_{x_0}=
\begin{cases}
\lambda_{x_0}(\rho_{x_0}\eta^{x_0}_1+\eta^{x_0}_3)
\,&\text{if}\;\,\a^{x_0}_{13}\neq0,\\
\lambda_{x_0}(-\rho_{x_0}\eta^{x_0}_1+\eta^{x_0}_2)
\,&\text{if}\;\,\a^{x_0}_{14}\neq0,
\end{cases}
$$ 
where 
$$
\lambda_{x_0}^2=\frac{2}{\rho_{x_0}^2+1}. 
$$
Now, consider the almost complex structure 
$J_{x_0}\colon T_{x_0}M\to T_{x_0}M$ defined by
\begin{align*}
J_{x_0}e^{x_0}_1&=\frac{\lambda_{x_0}}{\sqrt{2}} \left(\rho_{x_0}\,e^{x_0}_2+e^{x_0}_4\right), 
\quad J_{x_0}e^{x_0}_2=\frac{\lambda_{x_0}}{\sqrt{2}} \left(-\rho_{x_0}\,e^{x_0}_1+e^{x_0}_3\right),\\
J_{x_0}e^{x_0}_3&=\frac{\lambda_{x_0}}{\sqrt{2}}\left(-e^{x_0}_2+\rho\,e^{x_0}_4\right), 
\quad J_{x_0}e^{x_0}_4=-\frac{\lambda_{x_0}}{\sqrt{2}} \left(e^{x_0}_1+\rho_{x_0}\,e^{x_0}_3\right),
\end{align*}
if $\a^{x_0}_{13}\neq0$, or 
\begin{align*}
J_{x_0}e^{x_0}_1&=\frac{\lambda_{x_0}}{\sqrt{2}}\left(-\rho_{x_0}\,e^{x_0}_2+e^{x_0}_3\right), 
\quad J_{x_0}e^{x_0}_2=\frac{\lambda_{x_0}}{\sqrt{2}} \left(\rho_{x_0}\,e^{x_0}_1-e^{x_0}_4\right),\\
J_{x_0}e^{x_0}_3&=-\frac{\lambda_{x_0}}{\sqrt{2}}\left(e^{x_0}_1+\rho\,e^{x_0}_4\right), 
\quad J_{x_0}e^{x_0}_4=\frac{\lambda_{x_0}}{\sqrt{2}} \left(e^{x_0}_2+\rho_{x_0}\,e^{x_0}_3\right),
\end{align*}
if $\a^{x_0}_{14}\neq0$. Then 
$$
\omega_+(v,w)=\<J_{x_0}v,w\>\quad\text{for all } v,w\in T_{x_0}M. 
$$

Define the skew-symmetric 
endomorphism $J$ of the tangent bundle 
of $M^4$ such that 
$$
\omega_+(X,Y)=\<JX,Y\>\quad\text {for all } X,Y\in\mathcal X(M). 
$$
Clearly, $J$ is parallel because $\omega_+$ is parallel. 
Arguing as in {\em{Case I}}, $J$ is orthogonal, 
and thus $(M^4,\<\cdot,\cdot\>,J)$ is a 
K\"ahler manifold. 

We now show that \eqref{Jx} holds at every point $x\in M^4$.
If $\mu_1^+(x)=0$, then $f$ is totally 
umbilical at $x$, so \eqref{Jx} holds trivially. 
If $\mu_1^+(x)>0$, then by part $(ii)$ of Lemma \ref{nnic1} 
$\varepsilon_x=-1$, exactly one of the vectors $\a^x_{13}$ 
and $\a^x_{14}$ vanishes, and \eqref{kerr} holds at $x$. 
Then 
$$
Z_x=
\begin{cases}
\pm\lambda_x(\rho_x\eta^x_1+\eta^x_3)
\,&\text{if}\;\,\a^x_{13}\neq0,\\
\pm\lambda_x(-\rho_x\eta^x_1+\eta^x_2)
\,&\text{if}\;\,\a^x_{14}\neq0,
\end{cases}
$$ 
where the number $\lambda_x$ satisfies \eqref{lam} 
and consequently, 
\begin{align*}
J_xe^x_1&=\pm\frac{\lambda_x}{\sqrt{2}} \left(\rho_x\,e^x_2+e^x_4\right), 
\quad J_xe^x_2=\pm\frac{\lambda_x}{\sqrt{2}} \left(-\rho_x\,e^x_1+e^x_3\right),\\
J_xe^x_3&=\pm\frac{\lambda_x}{\sqrt{2}}\left(-e^x_2+\rho\,e^x_4\right), 
\quad J_xe^x_4=\mp\frac{\lambda_x}{\sqrt{2}} \left(e^x_1+\rho_x\,e^x_3\right),
\end{align*}
if $\a^x_{13}\neq0$, or 
\begin{align*}
J_xe^x_1&=\pm\frac{\lambda_x}{\sqrt{2}}\left(-\rho_x\,e^x_2+e^x_3\right), 
\quad J_xe^x_2=\pm\frac{\lambda_x}{\sqrt{2}} \left(\rho_x\,e^x_1-e^x_4\right),\\
J_xe^x_3&=\mp\frac{\lambda_x}{\sqrt{2}}\left(e^x_1+\rho\,e^x_4\right), 
\quad J_xe^x_4=\pm\frac{\lambda_x}{\sqrt{2}} \left(e^x_2+\rho_x\,e^x_3\right),
\end{align*}
if $\a^x_{14}\neq0$. 

A direct computation using \eqref{ii2}--\eqref{ii3} 
shows that \eqref{Jij} holds at $x$, and hence \eqref{Jx} 
follows by linearity. Thus the immersion $f$ is 
$(2,0)$-geodesic. 

\vspace{0.5ex}

\indent\emph{Case $II_b$}. 
Suppose that $\mu_1^+=0$ everywhere. Then $f$ is totally 
umbilical, and 
$
\accentset{\sim}{K}^f_{\max}=\accentset{\sim}{K}^f_{\min}=-H^2
$
at every point. It follows from the Gauss equation that 
$M^4$ is flat. By the Gauss-Bonnet-Chern theorem, 
the Euler characteristic of $M^4$ vanishes. Poincar\'e 
duality then implies 
$$
\beta_2(M^4)=2\beta_1(M^4)-2.
$$
Since $\beta_2(M^4)>0$, we have $\beta_1(M^4)\geq2$. 
On the other hand, $\beta_1(M^4)\leq4$, with equality 
only if $M^4$ is the flat torus. Therefore, $M^4$ is one 
of the oriented Bieberbach four-manifolds with first 
Betti number $2\leq\beta_1(M^4)\leq4$. 
\qed

\subsection{Four-dimensional submanifolds with finite fundamental group}

\begin{theorem}\label{n=4p1f}
Let $f\colon M^4\to\accentset{\sim}{M}^{4+m}$ be an 
isometric immersion of a compact, oriented four-dimensional 
Riemannian manifold with finite fundamental group. 
Suppose that the inequality \eqref{1} is satisfied at every point. 
Then one of the following 
assertions holds:
\vspace{1ex}

\noindent $(i)$ $M^4$ is diffeomorphic 
to a spherical space form. 
\vspace{1ex}

\noindent $(ii)$
Equality holds in \eqref{1} at every point, and one 
of the following occurs: 
\vspace{0.5ex}
\begin{enumerate}[topsep=1pt,itemsep=1pt,partopsep=1ex,parsep=0.5ex,leftmargin=*, label=(\roman*), align=left, labelsep=-0.5em]
\item [(a)] $M^4$ is a K\"ahler manifold biholomorphic 
to the complex projective plane $\CP^2$, equipped with a 
complex structure for which $f$ 
is $(2,0)$-geodesic.
\item [(b)] $M^4$ is isometric 
to a Riemannian product $(N_1,g_1)\times(N_2,g_2)$, 
where each factor is diffeomorphic to $\Sf^2$, and 
either both factors have nonnegative Gaussian 
curvature, or one factor has positive Gaussian 
curvature while the other has negative curvature 
at some point. Moreover, $\accentset{\sim}{K}^f_{\min}
=\accentset{\sim}{K}^f_{\max}$ everywhere, and at 
each point where $f$ is not totally umbilical, 
there exist two distinct Dupin principal normals $\delta_1$ 
and $\delta_2$, both of multiplicity $2$, such that 
$$
\<\delta_1,\delta_2\>=-\accentset{\sim}{K}^f_{\min} 
\quad \text{and}\quad 
E_{\delta_i}=TN_i,\; i=1,2.
$$
\end{enumerate}
\end{theorem}

\proof
Proposition \ref{propu} implies that $M^4$ 
has nonnegative isotropic curvature. We begin 
by proving the theorem for simply connected 
submanifolds. The proof is divided into two cases.
\vspace{0.5ex}

\indent\emph{Case I}. Suppose that there exists 
a point at which $M^4$ has positive isotropic 
curvature. By Remark (iv) in \cite{ses}, 
$M^4$ admits a metric of positive isotropic 
curvature. By the result of Moore and Micallef 
\cite{MM}, $M^4$ must be homeomorphic to $\Sf^4$. 

Furthermore, $M^4$ is locally irreducible. Indeed, if this 
were not the case, Theorem \ref{duke} would 
imply that $M^4$ is isometric to a Riemannian 
product of two compact surfaces, which leads to 
a contradiction. Hence, $M^4$ is locally irreducible, 
and by Theorem \ref{acta}, one of the three cases 
described in that theorem must hold.

Since $M^4$ is homeomorphic to $\Sf^4$, case $(ii)$ 
is excluded. In case $(i)$, the manifold is clearly diffeomorphic 
to $\Sf^4$. Moreover, the only compact symmetric spaces 
of dimension four are the round sphere, the product of 
two 2-dimensional spheres, the complex projective plane, and 
the quaternionic projective line $\HP^1\cong \Sf^4$. 
Consequently, in case (iii), the manifold $M^4$ must also 
be diffeomorphic to $\Sf^4$. 
\vspace{0.5ex}

\indent\emph{Case II}. Suppose now that $M^4$ has no 
point of positive isotropic curvature. Then, by part $(ii)$ 
of Proposition \ref{propu}, the inequality \eqref{1} holds 
as an equality everywhere. 

Assume first that $M^4$ is locally irreducible; 
therefore $M^4$ is as described in Theorem \ref{acta}. 
In case $(i)$, $M^4$ is clearly diffeomorphic to $\Sf^4$. 
In case (ii)}, since $M^4$ is a 
K\"ahler manifold biholomorphic to $\CP^2$, it follows that 
$\beta_+(M)=1$ or $\beta_-(M)=1$. 
The result then follows from Theorem \ref{n=4beta}. 
In case (iii), $M^4$ is isometric to a compact symmetric 
space. Hence, $M^4$ is isometric to a sphere, or to 
$\CP^2$; in the latter case, the conclusion again 
follows from Theorem \ref{n=4beta}. 

Suppose now that $M^4$ is locally reducible. Then 
$M^4$ must be isometric to a Riemannian product as 
in one of the two cases in Theorem \ref{duke}. In either 
case, it follows from part $(iii)$ of Proposition \ref{propprod} 
that 
$
\accentset{\sim}{K}^f_{\min}=\accentset{\sim}{K}^f_{\max} 
$
at every point, and that at each point $x\in M^4$ where $f$ 
is not totally umbilical, there exist two distinct Dupin 
principal normals $\delta_1$ and $\delta_2$, 
both of multiplicity $2$, such that 
$$
\<\delta_1,\delta_2\>(x)=-\accentset{\sim}{K}^f_{\min}(x) 
\quad \text{and}\quad 
T_xM=E_{\delta_1}(x)\oplus E_{\delta_2}(x).
$$

In case $(a)$, $M^4$ is isometric to a Riemannian 
product $(N_1,g_1)\times(N_2,g_2)$, where each 
$N_i$, $i=1,2$, is diffeomorphic to $\Sf^2$ and has 
nonnegative Gaussian curvature. 
In case $(b)$, $M^4$ is isometric to a Riemannian product 
$(N_1,g_1)\times(N_2,g_2)$, where the Gaussian 
curvature of the surface $N_1$ is negative somewhere, 
while $N_2$ has positive Gaussian curvature. Since $M^4$ 
is simply connected, each factor is also simply connected, 
so that both are diffeomorphic to $\Sf^2$. Hence, the 
submanifold is as described in part $(b)$ of the statement 
of the theorem. This completes the proof of the theorem 
for simply connected submanifolds.
\vspace{0.5ex}

Now suppose that the fundamental group $\pi_1(M^4)$ 
of $M^4$ is finite. Consider the universal covering 
$\pi\colon\hat M^4\to M^4$. Since 
the fundamental group of $M^4$ is finite, 
$\hat M^4$ must be compact. Moreover, the 
isometric immersion $\hat f=f\circ\pi$ satisfies 
condition \eqref{1}. Therefore, 
by the preceding argument, we conclude that 
either $\hat M^4$ is diffeomorphic 
to $\Sf^4$, or the submanifold $\hat f$ 
is as described in part $(ii)$ of the statement 
of the theorem. 
\vspace{0.5ex}

Assume first that $\hat M^4$ is diffeomorphic 
to $\Sf^4$. Clearly, $M^4$ is 
locally irreducible. Then, one of the three cases 
in Theorem \ref{acta} must occur. Since $M^4$ is 
homeomorphic to $\Sf^4$, case $(ii)$ 
is excluded. Moreover, the only compact symmetric 
spaces of dimension four are the round sphere, the 
product of two 2-dimensional spheres, the complex 
projective plane, and the quaternionic projective 
line $\HP^1\cong\Sf^4$. Consequently, in all cases, 
the manifold $M^4$ must be diffeomorphic to a 
spherical space form. 
\vspace{0.5ex}

Suppose now that the submanifold $\hat f$ 
is as described in part $(ii)$ of of the statement 
of the theorem. In particular, $\hat M^4$ 
is diffeomorphic to the tours $\Sf^2\times\Sf^2$, 
or it is a K\"ahler manifold biholomorphic to $\CP^2$. 
By Theorem 4.10 in \cite{MW}, one of the following holds:
\begin{enumerate}[topsep=1pt,itemsep=1pt,partopsep=1ex,parsep=0.5ex,leftmargin=*, label=(\roman*), align=left, labelsep=-0.5em]
\item[(1)] $M^4$ carries a metric of positive isotropic 
curvature. 
\item[(2)] $M^4$ is diffeomorphic to a product 
$\Sf^2\times\varSigma^2$, where $\varSigma^2$ is a 
compact surface.
\item[(3)] $M^4$ is a K\"ahler manifold biholomorphic 
to $\CP^2$.
\end{enumerate} 

Case (1) cannot occur. Indeed, if $M^4$ carries a 
metric of positive isotropic curvature, then by \cite{MM} 
we have $\pi_2(M^4)=0$, which implies 
$
\pi_2(\hat M^4)\cong\pi_2(M^4)=0, 
$ 
a contradiction, since 
$\pi_2(\Sf^2\times\Sf^2)\cong\Z\oplus\Z$. 
Suppose that case (2) holds. Since $\pi_1(M^4)$ is 
finite, it follows that $M^4$ is diffeomorphic to
$\Sf^2\times\Sf^2$. Thus, the covering map 
$\pi\colon\accentset{\sim}{M}^4\to M^4$ 
is a global isometry, and the result follows. 
Finally, if $M^4$ is a K\"ahler manifold biholomorphic 
to $\CP^2$, then the result follows directly from the 
proof of the theorem in the simply connected case.
\qed

\subsection{Four-dimensional submanifolds with infinite fundamental group}
This section is devoted to the study of four-dimensional 
submanifolds with infinite fundamental group that 
satisfy the pinching condition \eqref{1}. More precisely, 
we prove the following theorem.

\begin{theorem}\label{n=4p1inf}
Let $f\colon M^4\to\accentset{\sim}{M}^{4+m}$ be an 
isometric immersion of a compact, oriented four-dimensional 
Riemannian manifold with infinite fundamental group. 
Suppose that the inequality \eqref{1} is satisfied at every point. 
Then one of the following 
assertions holds:
\vspace{1ex}

\noindent $(i)$ The universal cover of $M^4$ is 
isometric to a Riemannian product 
$\R\times N$, where $N$ is diffeomorphic to 
$\Sf^3$ and has nonnegative Ricci curvature. 
\vspace{1ex}

\noindent $(ii)$
Equality holds in \eqref{1},
$\accentset{\sim}{K}^f_{\min}=\accentset{\sim}{K}^f_{\max} $
everywhere, and one of the following occurs: 
\vspace{0.1ex}
\begin{enumerate}[topsep=1pt,itemsep=1pt,partopsep=1ex,parsep=0.5ex,leftmargin=*, label=(\roman*), align=left, labelsep=-0.5em]
\item [(a)] 
The universal cover of $M^4$ is 
isometric to a Riemannian product 
$(N_1,g_1)\times(N_2,g_2)$, where 
$N_2$ is diffeomorphic to $\Sf^2$, and either 
$(N_1,g_1)=(\R^2,g_0)$, where $g_0$ is the 
Euclidean flat metric, and $N_2$ has nonnegative 
Gaussian curvature, or 
the Gaussian curvatures of $(N_1,g_1)$ and 
$(N_2,g_2)$ satisfy
$\min K_{N_2}\geq-\min K_{N_1}>0$. 
Moreover, at any point where 
$f$ is not totally umbilical, there exist two 
distinct Dupin principal normals $\delta_1$ 
and $\delta_2$, both of multiplicity $2$, such that 
$$
\<\delta_1,\delta_2\>=-\accentset{\sim}{K}^f_{\min} 
\quad \text{and}\quad 
E_{\delta_i}=\pi_*TN_i,\; i=1,2,
$$
where $\pi$ is the universal covering map of $M^4$.
\item [(b)] $M^4$ is flat with first Betti number 
$1\leq\beta_1(M^4)\leq4$, and the immersion $f$ 
is totally umbilical, satisfying 
$
\accentset{\sim}{K}^f_{\min}=
\accentset{\sim}{K}^f_{\max}=-H^2
$ 
at every point. 
\end{enumerate}
\end{theorem}
\proof
Proposition \ref{propu} implies that $M^4$ has 
nonnegative isotropic curvature. We first claim 
that $M^4$ is locally reducible. Suppose, 
to the contrary, that $M^4$ is locally irreducible. 
Then one of the cases $(i)$--$(iii)$ in Theorem 
\ref{acta} applies. In each of these cases, 
the universal cover $\hat M^4$ is compact, 
which contradicts the assumption that the 
fundamental group of $M^4$ is infinite. 

Hence, $M^4$ is locally reducible. It then follows 
from Theorem \ref{duke} that the universal cover 
$\hat M^4$ is isometric to a Riemannian product 
as in cases $(a)$ or $(b)$ of that theorem. 
The proof of the theorem is divided into two cases.
\vspace{0.5ex}

\indent\emph{Case I}. Suppose that there exists 
a point at which $M^4$ has positive isotropic 
curvature. Since $M^4$ has nonnegative isotropic 
curvature, it follows from Remark (iv) in \cite{ses} 
that $M^4$ admits a metric of positive isotropic 
curvature. By the result of Moore and Micallef 
\cite{MM}, we have $\pi_2(M^4)=0$, and thus 
$\pi_2(\hat M^4)=0$. Since $\pi_2(\Sf^2)\cong\Z$, 
the universal cover $\hat M^4$ cannot be 
isometric to a Riemannian product 
as in case $(b)$ of Theorem \ref{duke}. 

Hence, the universal cover $\hat M^4$ is 
isometric to a Riemannian product
$$
(\R^{n_0},g_0)\times(N_1^{n_1},g_1)\times(N_2^{n_2},g_2), 
$$
where $n_0\geq1$, $g_0$ is the flat Euclidean metric, and either 
$n_i=2$ and $N_i=\Sf^2$, $i=1,2$, has nonnegative Gaussian curvature, 
or $n_i=3$ and $N_i$ is compact with nonnegative Ricci 
curvature. 
The fact that $\pi_2(\hat M^4)=0$ implies that, $n_0=1$ and $n_1=3$. 
Consequently, $\hat M^4$ is isometric to a Riemannian 
product $\R\times N_1$, where $N_1$ is a compact, simply 
connected manifold with nonnegative Ricci curvature. 
By Theorem 1.2 in \cite{H}, $N_1$ is diffeomorphic to $\Sf^3$. 
\vspace{0.5ex}

\indent\emph{Case II}. Now suppose that $M^4$ has no 
point of positive isotropic curvature. Then part $(ii)$ of 
Proposition \ref{propu} implies that the inequality 
\eqref{1} holds as an equality everywhere.

Suppose first that case $(b)$ in Theorem \ref{duke} 
holds, namely that $\hat M^4$ is isometric to a Riemannian 
product of two surfaces $(N_1,g_1)\times(N_2,g_2)$, 
where $N_2$ is compact and the Gaussian curvatures 
of $(N_1,g_1)$ and $(N_2,g_2)$ satisfy
$\min K_{N_2}\geq-\min K_{N_1}>0$. 
Clearly, the isometric immersion $\hat f=f\circ\pi$ 
satisfies condition \eqref{1}, where $\pi\colon\hat M^4\to M^4$ 
is the covering map. It then follows from 
part $(iii)$ of Proposition \ref{propprod} that the submanifold 
is as described in part $(a)$ of the theorem. 

Assume now that case $(a)$ in Theorem \ref{duke} 
holds, namely the universal cover $\hat M^4$ is 
isometric to a Riemannian product
$$
(\R^{n_0},g_0)\times(N_1^{n_1},g_1)\times(N_2^{n_2},g_2), 
$$
where $n_0\geq1$, $g_0$ is the flat Euclidean metric, and either 
$n_i=2$ and $N_i=\Sf^2$, $i=1,2$, has nonnegative Gaussian 
curvature, or $n_i=3$ and $N_i$ is compact with 
nonnegative Ricci curvature. Then, $n_0=1,2,4$. 

If $n_0=1$, we argue as in \emph{Case I}, to 
conclude that $\hat M^4$ is isometric to a 
Riemannian product $\R\times N_1$, where 
$N_1$ is diffeomorphic to $\Sf^3$ with 
nonnegative Ricci curvature. 
If $n_0=n_1=2$, then it follows from 
part $(iii)$ of Proposition \ref{propprod} that the 
submanifold is as described in part $(a)$ 
of the theorem.

Suppose $n_0=4$, namely that $\hat M^4$ is isometric 
to $\R^4$. By the Gauss-Bonnet-Chern theorem, 
the Euler characteristic of $M^4$ vanishes. Poincar\'e 
duality then implies 
$$
\beta_2(M^4)=2\beta_1(M^4)-2.
$$
Thus, $\beta_1(M^4)\geq1$. 
On the other hand, $\beta_1(M^4)\leq4$, with equality 
only if $M^4$ is the flat torus. Therefore, $M^4$ is an oriented Bieberbach four-manifold with first Betti number satisfying $1\leq\beta_1(M^4)\leq4$. 

From part $(ii)$ of Proposition \ref{propu}, it follows that at 
each point $x\in M^4$ there exists an orthonormal four-frame 
$\{e_i\}_{i=1}^4$ such that all conditions \eqref{c1}--\eqref{c5} 
hold. 
The Gauss equation, together with \eqref{c1} and 
\eqref{c22}, implies that the sectional curvatures satisfy 
\begin{align*}
K(e_1\wedge e_2)&=\|\a_{11}\|^2+\accentset{\sim}{R}^f_{1221}(x),\quad
K(e_3\wedge e_4)=\|\a_{44}\|^2+\accentset{\sim}{R}^f_{3443}(x),\\
K(e_1\wedge e_3)&=\|\a_{14}\|^2+
\frac{1}{3}\big(\accentset{\sim}{K}^f_{\max}(x)-\accentset{\sim}{K}^f_{\min}(x)\big),\\
K(e_1\wedge e_4)&=\|\a_{13}\|^2+
\frac{1}{3}\big(\accentset{\sim}{K}^f_{\max}(x)-\accentset{\sim}{K}^f_{\min}(x)\big). 
\end{align*}
Since $M^4$ is flat, these equalities yield 
$\a_{13}=\a_{14}=0$ and $\accentset{\sim}{K}^f_{\max}(x)
=\accentset{\sim}{K}^f_{\min}(x)$. 
By the last two equations in \eqref{c1}, we also have 
$\a_{23}=\a_{24}=0$. Hence, 
$$
\|\a_{11}\|^2=\|\a_{44}\|^2=-\accentset{\sim}{K}^f_{\min}(x).
$$
Moreover, \eqref{c22} implies
$$
\<\a_{11},\a_{44}\>=-\accentset{\sim}{K}^f_{\min}(x).
$$
From these equalities and \eqref{c2}, we conclude 
that $\a_{11}=\a_{44}=\mathcal H_f(x)$. Therefore, 
$H^2(x)=-\accentset{\sim}{K}^f_{\min}(x)$, 
and $f$ is totally umbilical. 
\qed
\vspace{1.5ex}

\noindent\emph{Proof of Theorem \ref{n=4}:} 
The assertion follows directly from Theorems 
\ref{n=4p1f} and \ref{n=4p1inf}.
\qed 

\section{Submanifolds of space forms}
\noindent\emph{Proof of Theorem \ref{sfhigh}:} 
Arguing by contradiction, suppose that $M^n$ is neither diffeomorphic to a 
spherical space form nor has its universal cover 
$\hat M^n$ isometric to a Riemannian product 
$\R\times N$, where $N$ is diffeomorphic to 
$\Sf^{n-1}$ and has nonnegative isotropic curvature. 
Then it follows from Theorem \ref{n5} that equality holds in 
\eqref{sfpinch}, and one of the cases in part $(iii)$ of that theorem 
occurs. 

Case $(a)$ is ruled out by the fact that the complex projective 
space, the quaternionic projective space and the Cayley plane 
(all endowed with their canonical Riemannian metrics) do not 
admit totally geodesic immersions into a space form.
In cases $(b)$--$(c)$, the ambient space is Euclidean, 
and the submanifold has positive index of relative nullity 
at each point. This contradicts the compactness of $M^n$ 
in both cases (see Corollary 1.6 in \cite{DT}). 
\iffalse
Therefore, $M^n$ is either diffeomorphic to a 
spherical space form or its universal cover 
$\hat M^n$ isometric to a Riemannian product 
$\R\times N$, where $N$ is diffeomorphic to 
$\Sf^{n-1}$ and has nonnegative isotropic curvature. 

Now assume that $M^n$ is diffeomorphic to a 
spherical space form. Observe that 
$$
S\le4c+\frac{n^2H^2}{n-2}\leq 
nc+\frac{n^3H^2}{4(n-2)}-\frac{n(n-4)}{4(n-2)}H\sqrt{
n^2H^2+8c(n-2)}.
$$
Hence, the immersion $f$ satisfies the pinching 
condition studied in \cite{v3} for $k=2$. Since 
the fundamental group of $M^n$ is finite, by Corollary 
4 in that paper that either $M^n$ is homeomorphic to 
$\Sf^n$, or equality holds in the above inequality at 
every point and the submanifold is the standard 
embedding of a torus in a sphere as in part $(i)$--$(ii)$ 
of that corollary. 
\fi
\qed 
\vspace{1ex}

\noindent\emph{Proof of Corollary \ref{cor}:} 
We observe that 
$$
S\leq2c+\frac{n^2H^2}{n-1}\leq4c+\frac{n^2H^2}{n-2}
$$
at every point. By Theorems \ref{sfhigh} and 
\ref{sf4}, $M^n$ is either diffeomorphic to $\Sf^n$, 
or its universal cover is isometric to a Riemannian 
product $\R\times N$, where $N$ is diffeomorphic 
to $\Sf^{n-1}$, unless the submanifold is as in part 
$(iii)$ of Theorem \ref{sf4}. In the latter case, 
equality holds in the above inequality, implying 
$c=0$ and that $f$ is minimal, which is a contradiction. 

On the other hand, we also have 
\be\label{v3}
S\leq2c+\frac{n^2H^2}{n-1}\leq nc+\frac{n^3H^2}{2(n-1)}
-\frac{n(n-2)}{2(n-1)}H\sqrt{n^2H^2+4c(n-1)}. 
\ee
Hence, the submanifold satisfies the pinching 
condition studied in \cite{v3} for $k=1$. 
It then follows from Theorem 2 in that paper that either 
the universal cover of $M^n$ is homeomorphic to 
$\Sf^n$, or the above inequality holds with equality 
everywhere. In the latter case, we obtain $c=0$ and 
the submanifold is as described in parts $(i)$, $(ii)$ 
or $(iii)$ of that theorem, which contradicts \eqref{v3}, 
now holding as an equality. 
Therefore, $M^n$ is diffeomorphic to $\Sf^n$. 
\qed
\vspace{1ex}

\noindent\emph{Proof of Corollary \ref{cor1l}:} 
By assumption, we have 
\be\label{2??}
S\leq\frac{16}{3}\big(\inf\accentset{\sim}{K}
-\frac{1}{4}\sup\accentset{\sim}{K}\big)
+\frac{n^2H^2}{n-2}
\leq\frac{16}{3}\big(\accentset{\sim}{K}^f_{\min}
-\frac{1}{4}\accentset{\sim}{K}^f_{\max}\big)
+\frac{n^2H^2}{n-2}.
\ee

Suppose that $M^n$ is neither diffeomorphic 
to a spherical space form, nor has its universal cover 
isometric to a Riemannian product $\R\times N$, 
where $N$ is diffeomorphic to $\Sf^{n-1}$. Then it 
follows from Theorem \ref{n5} that 
equality holds in \eqref{2??} everywhere. In 
particular, 
\be\label{infsup}
\accentset{\sim}{K}^f_{\min}(x)=\inf\accentset{\sim}{K},
\quad
\accentset{\sim}{K}^f_{\max}(x)=\sup\accentset{\sim}{K}
\quad\text{for all }x\in M^n. 
\ee
We distinguish two cases.

Assume first that the submanifold is as described in part 
$(iii$-$a)$ of Theorem \ref{n5}. Then $f$ is totally 
geodesic. It follows from \eqref{2??}, now 
holding as an equality, together with \eqref{infsup}, that 
$$
\inf\accentset{\sim}{K}=
\frac{1}{4}\sup\accentset{\sim}{K}>0.
$$
By the Bonnet-Myers theorem, the manifold 
$\accentset{\sim}{M}^{n+m}$ is compact. 
Clearly, $\accentset{\sim}{M}^{n+m}$ has weakly 
$1/4$-pinched sectional curvatures. By Theorem 1 
in \cite{BSacta} and the 
classical theorem due to Berger \cite{BSacta} 
and Klingenberg \cite{Kling}, $\accentset{\sim}{M}^{n+m}$
is either a rank-one symmetric space other than the sphere, 
or it is diffeomorphic but not isometric to the sphere $\Sf^{n+m}$.
\qed

\section{Examples of submanifolds satisfying the pinching condition}

In this section, we present several classes of submanifolds satisfying condition \eqref{1}, thereby showing that our results cannot, in general, be strengthened to extrinsic rigidity statements. In particular, we construct numerous geometrically distinct submanifolds that are diffeomorphic either to the sphere $\Sf^n$ or to the torus $\Sf^1\times\Sf^{n-1}$, all satisfying \eqref{1}.

Observe that if an isometric immersion
$f\colon M^n\to\accentset{\sim}{M}^{n+m}$
satisfies \eqref{1}, then the composition
$g\circ f\colon M^n\to\accentset{\sim}{N}$
also satisfies \eqref{1}, provided that
$g\colon\accentset{\sim}{M}^{n+m}\to\accentset{\sim}{N}$
is totally geodesic. Together with the examples described below, this observation yields a broad class of submanifolds satisfying \eqref{1}.

There exist many submanifolds in Euclidean space 
that satisfy \eqref{1} and are diffeomorphic to a sphere. 
Indeed, let $f\colon M^n\to\R^{n+1}$, $n\geq4$, be an ovaloid 
in Euclidean space with principal curvatures 
$0<\lambda_1\leq\dots\leq\lambda_n$. 
It follows from the inequalities 
$S\leq n\lambda^2_n$ and $H\geq\lambda_1$
that, if the global minimum of $\lambda_1$ and 
the global maximum of $\lambda_n$ satisfy 
\be\label{lambda2}
\max_M\lambda _n\leq\min_M\lambda _1\big({\frac{n}{n-2}}\big)^{1/2}, 
\ee
then $f$ satisfies condition \eqref{1} at every point. In particular, 
if \eqref{lambda2} holds strictly, then \eqref{1} is satisfied as a strict 
inequality everywhere. By Hadamard’s classical theorem, 
every ovaloid in $\R^{n+1}$ is diffeomorphic to $\Sf^n$. 

A large class of ellipsoids provides concrete examples satisfying 
condition \eqref{lambda2}. Consider, for instance, the 
ellipsoid in $\R^{n+1}$ defined by 
$$ 
\frac{x_1^2}{a_1^2}+\cdots+
\frac{x_{n+1}^2}{a_{n+1}^2} = 1,
$$
where $0<a_1\leq\dots\leq a_{n+1}$. A 
straightforward computation shows that the 
minimum and the maximum of the principal 
curvatures of the ellipsoid are $a_1/a_{n+1}^2$ 
and $a_{n+1}/a_1^2$, respectively. It follows 
that condition \eqref{lambda2} holds provided
$$
a_{n+1}\leq a_1\big({\frac{n}{n-2}}\big)^{1/6}. 
$$

The next proposition presents geometrically 
distinct isometric immersions of manifolds diffeomorphic 
to the torus $\Sf^1 \times \Sf^{n-1}$, providing further 
examples that satisfy condition \eqref{1}. 

\begin{proposition}\label{ex}
Let $g\colon N^{n-1}\to\R^{m_1}$, $n\geq4$, be 
an isometric immersion of a manifold $N^{n-1}$ satisfying, 
at every point, 
$$
S_g<\frac{(n-1)^2}{n-2}H^2_g. 
$$
Let $\gamma\colon\Sf^1\to\R^{m_2}$ be a closed 
unit-speed curve whose first curvature $\kappa_1$ 
satisfies 
\be\label{k1}
\kappa_1^2\leq
\frac{n-2}{n-3}\,\min_N\left\{\frac{(n-1)^2}{n-2}H^2_g-S_g\right\}.
\ee
Then the product immersion 
$
f=\gamma\times g\colon\Sf^1\times N^{n-1}\to\R^{m_1+m_2}
$ 
satisfies condition \eqref{1}. Moreover, $N^{n-1}$ is 
diffeomorphic to $\Sf^{n-1}$. 
\end{proposition}
\proof
The squared length $S_f$ of the second fundamental form, 
and the mean curvature $H_f$ of the product immersion 
$f$ are given by 
$$
S_f=\kappa_1^2+S_g\quad {\text{and}}\quad 
n^2H^2_f=\kappa_1^2+(n-1)^2H^2_g. 
$$
From these expressions, it follows that condition 
\eqref{1} for $f$ is equivalent 
to inequality \eqref{k1}. Moreover, Corollary \ref{cor} implies that 
$N^{n-1}$ is diffeomorphic to $\Sf^{n-1}$.\qed
\vspace{1ex}

Next, we show that there exist many isometric 
immersions $g \colon N^{n-1} \to \R^m$ satisfying 
the hypotheses of Proposition \ref{ex}.
In particular, we construct geometrically distinct 
immersions of $\Sf^n$ into Euclidean space 
$\R^{n+1}$ that satisfy condition \eqref{1} for all 
$n\geq 3$, either with strict inequality 
or with equality.

Let $g\colon N^{n-1}\to\R^n$, $n\geq4$, be an ovaloid 
with principal curvatures 
$0<\lambda_1\leq\dots\leq\lambda_{n-1}$. 
It follows directly from the inequalities 
$S_g\leq (n-1)\lambda^2_{n-1}$ and $H_g\geq\lambda_1$
that, if 
\be\label{lambda}
\max_N\lambda _{n-1}<\min_N\lambda _1\big({\frac{n-1}{n-2}}\big)^{1/2}, 
\ee
then $g$ satisfies the hypotheses of Proposition \ref{ex}. 

A large class of ellipsoids provides concrete examples 
satisfying condition \eqref{lambda}. Specifically, the 
ellipsoid in $\R^n$ defined by 
$$ 
\frac{x_1^2}{a_1^2}+\cdots+
\frac{x_n^2}{a_n^2} = 1,
$$
where $0<a_1\leq\dots\leq a_n$, satisfies 
condition \eqref{lambda} provided that
$$
a_n<a_1\big({\frac{n-1}{n-2}}\big)^{1/6}. 
$$

Consequently, Proposition \ref{ex} 
provides geometrically distinct isometric 
immersions of manifolds diffeomorphic 
to the torus $\Sf^1 \times \Sf^{n-1}$ for 
$n\geq4$, which also satisfy condition 
\eqref{1}. Thus, we obtain numerous compact, 
geometrically distinct submanifolds that 
strictly satisfy \eqref{1} at every point. 
Moreover, this strict form is preserved under 
sufficiently small smooth deformations 
of any such example. 

\begin{example}[\bf{Totally geodesic submanifolds of 
a CROSS}]
\emph{By the Gauss equation, a totally geodesic immersion satisfying \eqref{1} induces weakly $1/4$-pinched sectional curvature on $M^n$. That is, for each $x \in M^n$, its sectional curvatures satisfy 
$0\leq {K}(\sigma_1)\leq4 K(\sigma_2)$ for all two-planes $\sigma_1, \sigma_2 \subset T_x M^n$. If in addition $M^n$ is compact, 
then by Theorem 1 in \cite{BSacta}, $M^n$ is either 
diffeomorphic to a spherical space form 
or locally symmetric. In case $M^n$ is locally symmetric and 
non-flat, then its universal cover is a CROSS (rank-one symmetric 
space), namely the complex projective space $\CP^{n/2}$, 
the quaternionic projective space $\HP^{n/4}$ and the 
Cayley plane $\OP^2$. 
}

\emph{In particular, any totally geodesic submanifold of 
a CROSS satisfies condition 
\eqref{1} and is itself a CROSS. 
We recall that the totally geodesic submanifolds of 
$\CP^n$ are 
$\CP^k\subset\CP^n,\RP^k\subset\CP^n, k<n$. 
The totally geodesic submanifolds of 
$\HP^n,\;\HP^k\subset\HP^n,\CP^k\subset\HP^n,
\RP^k\subset\HP^n, k<n$, and 
the totally geodesic submanifolds of the Cayley 
plane are 
$\HP^2\subset\OP^2,\CP^2\subset\OP^2, 
\Sf^k\subset\OP^2, k<8$.
}
\end{example}

\begin{example}[\bf{Weakly $1/4$-pinched manifolds}]\emph{
Let $M^n$ be a compact Riemannian manifold with 
weakly $1/4$-pinched sectional curvatures, and let 
$\accentset{\sim}{M}^{n+m}$ be the Riemannian 
product $\accentset{\sim}{M}^{n+m}=M^n\times N^m$, 
where $N^m$ is an arbitrary Riemannian manifold. For 
any point $x_0\in N^m$, consider the totally geodesic 
inclusion 
$$
f\colon M^n\to\accentset{\sim}{M}^{n+m},\quad f(x)=(x, x_0).
$$ 
It is immediate that $f$ satisfies condition \eqref{1}. 
}
\end{example}

\begin{example}\emph{
Let $f$ be the product immersion 
$$
f={\rm{Id}_{\mathbb T^2}}\times j\colon\mathbb T^2
\times\Sf^{n-2}(r)\to\mathbb T^2\times\R^{n-1},
$$ 
where ${\rm{Id}_{\mathbb T^2}}$ denotes the identity map 
of the flat $2$-torus $\mathbb T^2=\Sf^1\times\Sf^1$ 
and $j\colon\Sf^{n-2}(r)\to\mathbb \R^{n-1}$ 
is the standard umbilical inclusion of the sphere $\Sf^{n-2}(r)$. 
It is clear that $f$ satisfies the pinching 
condition \eqref{1}. In fact, this immersion is of the type 
described in part $(c)$ of Theorem \ref{n5} and in
part $(d)$ of Theorem \ref{n=4}.
}
\end{example}

\begin{example}[\bf{Totally geodesic submanifolds of flat manifolds}]\emph{
Any totally geodesic submanifold of a flat manifold trivially 
satisfies condition \eqref{1} and has infinite fundamental 
group. In fact, it is well known (see \cite{W}) that every 
compact flat Riemannian manifold 
$\accentset{\sim}{M}^{n+m}$ is a quotient 
$\R^{n+m}/\Gamma$, where $\Gamma$ is a Bieberbach 
group. Consequently, the classification of such manifolds 
reduces to the classification of Bieberbach groups, which 
are discrete groups of Euclidean isometries acting freely 
and cocompactly on $\R^{n+m}$. 
}

\emph{Every totally geodesic submanifold of 
$\R^{n+m}/\Gamma$ is the image of an affine subspace 
$A\subset\R^{n+m}$ such that 
$$
\Gamma_A=\left\{\gamma\in\Gamma: \gamma(A)=A \right\}
$$
acts cocompactly on $A$. Hence, any totally geodesic submanifold 
is itself a flat manifold of the form $A/\Gamma_A$. 
}
\end{example}

\begin{example}\emph{
A straightforward computation shows that the 
standard isometric embedding of the torus 
$$
g\colon\Sf^2(r)\times\Sf^2(\sqrt{R^2-r^2})\to 
\Sf^5(R)
$$ 
satisfies condition \eqref{1}. This 
example corresponds to the class described 
in part $(c)$ of Theorem \ref{n=4}.
}

\emph{Moreover, let 
$\accentset{\sim}{M}=\Sf^5(R)\times N$ be the Riemannian 
product with an arbitrary Riemannian manifold $N$. For 
any point $y_0\in N$, consider the embedding 
$$
f\colon\Sf^2(r)\times\Sf^2(\sqrt{R^2-r^2})
\to\accentset{\sim}{M},\quad f(x)=(g(x), y_0).
$$ 
It is immediate that $f$ is also of the type 
described in part $(c)$ of Theorem \ref{n=4}. 
}

\emph{
Finally, the 
standard isometric embedding 
$\Sf^1(r)\times\Sf^3(\sqrt{1-r^2})\hookrightarrow\Sf^5$ 
satisfies condition \eqref{1} whenever $r\geq1/2$, and 
its fundamental group is infinite. This example shows 
that the class of submanifolds appearing in part $(ii)$ 
of Theorem \ref{n=4} is nonempty. 
}
\end{example}

\begin{example}\emph{Let $\R^{n,1}$ denote the 
$(n+1)$-dimensional Lorentzian space equipped 
with the Lorentzian inner product $\<\cdot,\cdot\>$ 
of signature $(n,1)$. The hyperbolic space $\mathbb H^n(r)$ 
of constant sectional curvature $-1/r^2$ is realized as the 
hypersurface in $\R^{n,1}$ given by 
$$
\mathbb H^n(r)=\left\{x=(x_0,\dots,x_n)\in\R^{n,1}: 
\<x,x\>=-r^2,\; x_0>0\right\}.
$$
\indent
Let 
$$
i_1\colon\mathbb H^2(r_1)\to
\mathbb H^3(r_1)\subset\R^{3,1},\quad 
i_2\colon\Sf^2(R)\to\Sf^3(r_2)\subset\R^4
$$ 
be, respectively, a totally geodesic inclusion of a geodesic plane in 
$\mathbb H^3(r_1)$ and a totally umbilical inclusion of 
a sphere in $\Sf^3(r_2)$. Consider the 
exterior product of these immersions (see \cite{DT}). 
Given an orthogonal decomposition
$$
\R^{7,1}=\R^{3,1}\oplus\R^4, 
$$
where $\R^{3,1}$ carries the Lorentzian metric of 
signature $(3,1)$ and $\R^4$ carries the Euclidean 
metric, we choose the radii so that 
$-r_1^2+r_2^2=-r^2$. 
The exterior product of $i_1$ and $i_2$ is the 
immersion 
$$
g\colon \mathbb H^2(r_1)\times\Sf^2(R)\to
\mathbb H^7(r)\subset\R^{7,1}
$$
defined by
$$
g(x)=\left(i_1(x_1),i_2(x_2)\right), \quad x=(x_1,x_2).
$$
It is straightforward to verify that the immersion $g$ satisfies 
condition \eqref{1} with equality.
}

\emph{Now consider a compact hyperbolic $3$-manifold 
$\mathbb H^3(r_1)/\Gamma$, where $\Gamma$ is 
a discrete torsion-free subgroup of the isometry group 
$O(3,1)$ of $\mathbb H^3(r_1)$. Any 
totally geodesic plane 
$\mathbb H^2(r_1)\subset\mathbb H^3(r_1)$ that is 
invariant under $\Gamma$ gives rise 
to a totally geodesic immersion 
$j\colon\varSigma\to\mathbb H^3(r_1)/\Gamma$, 
where $\varSigma=\mathbb H^2(r_1)/\Gamma$ is 
a closed hyperbolic surface of genus $g\geq2$. 
}
\emph{Let $\mathsf{j}\colon O(3,1)\to O(7,1)$ be the 
block-diagonal inclusion 
$$
\mathsf{j}(A)=\begin{pmatrix}
A & 0\\
0 & \rm{Id}_{\R^4}
\end{pmatrix}. 
$$
Then $\mathbb H^7(r)/\mathsf{j}(\Gamma)$ is a 
smooth non-compact hyperbolic manifold, diffeomorphic to 
$\left(\mathbb H^3(r_1)/\Gamma\right)\times\R^4$. 
Moreover, the immersion $g$ constructed above 
descends to an immersion
$$
f\colon\varSigma\times\Sf^2(R)\to
\mathbb H^7(r)/\mathsf{j}(\Gamma).
$$
It is clear that $f$ satisfies condition \eqref{1} as an 
equality. In fact, the immersion $f$ is of the type 
described in part $(d)$ of Theorem \ref{n=4}.
}

\emph{Finally, let 
$\accentset{\sim}{M}$ be the Riemannian product 
$\accentset{\sim}{M}=
\left(\mathbb H^7(r)/\mathsf{j}(\Gamma)\right)\times N$, 
where $N$ is an arbitrary Riemannian manifold. For 
any point $y_0\in N$, consider the immersion 
$$
F\colon\varSigma\times\Sf^2(R)\to\accentset{\sim}{M},
\quad F(x)=(f(x), y_0).
$$ 
It is immediate that $F$ is also of the type described 
in part $(d)$ of Theorem \ref{n=4}. 
}
\end{example}

\begin{example}[\bf{Submanifolds of hyperbolic spaces satisfying} \eqref{1}]
\emph{
Let $g\colon M^n\to\R^{n+m-1}$, $n\geq5$, be an 
isometric immersion satisfying condition \eqref{1} strictly 
at every point. Let $c<0$ be a constant such that 
$$
c\geq\frac{n-2}{2(n-4)}\max_M\Big\{S_g-\frac{n^2H^2_g}{n-2}\Big\}.
$$
Consider the isometric 
immersion 
$
f=j\circ g\colon M^n\to\mathbb H_c^{n+m},
$
where $j\colon\R^{n+m-1}\to\mathbb H_c^{n+m}$ is the 
umbilical immersion of $\R^{n+m-1}$ as a horosphere in the 
hyperbolic space $\mathbb H_c^{n+m}$ of curvature $c$. 
It is straightforward to verify that the immersion $f$ satisfies 
condition \eqref{1} at every point. 
}
\end{example}

\begin{example}[\bf{Submanifolds of hyperbolic manifolds satisfying} \eqref{1}]
\emph{
Let $\mathbb H^{n+m}(r)/\Gamma$ be a hyperbolic manifold with 
finite volume, where $\Gamma \subset O(n+m,1)$ is a discrete 
torsion-free subgroup, and let 
$\varSigma^n\subset \mathbb H^{n+1}(r)\subset \mathbb H^{n+m}(r)$ 
be a horosphere. Recall that a horosphere is isometric to Euclidean 
space $\mathbb R^n$. An isometry of $\mathbb H^{n+m}(r)$ preserves 
$\varSigma^n$ if and only if it is a parabolic element fixing the point 
on $\partial_\infty \mathbb H^{n+m}(r)$ that determines the 
horosphere. Consequently, any subgroup 
$\Gamma_0 \subset \Gamma$ that preserves $\varSigma^n$ 
consists entirely of parabolic elements with a common fixed point 
on $\partial_\infty \mathbb H^{n+m}(r)$, and its 
action on $\varSigma^n$ is by Euclidean isometries. 
}

\emph{
The quotient $M^n = \Sigma^n / \Gamma_0$ 
is compact if and only if $\Gamma_0$ is a cocompact discrete 
subgroup of the Euclidean isometry group of $\mathbb R^n$. By 
the Bieberbach theorem, this is equivalent to $\Gamma_0$ being 
virtually $\mathbb Z^n$. Therefore, such a compact quotient exists 
if and only if $\Gamma$ contains a rank-$n$ abelian subgroup of 
parabolic isometries fixing a common point on 
$\partial_\infty \mathbb H^{n+m}(r)$, assuming that $\Gamma$ 
is a lattice. Geometrically, this situation arises precisely when 
$\mathbb H^{n+m}(r)/\Gamma$ possesses a cusp. In that 
case, a maximal parabolic subgroup of $\Gamma$ acts 
cocompactly on a horosphere determined by the 
corresponding ideal point, and the resulting quotient 
is a compact flat manifold.} 

\emph{
The natural projection induces an isometric immersion
$
j \colon M^n \longrightarrow \mathbb H^{n+m}(r)/\Gamma.
$
Since horospheres in $\mathbb H^{n+1}(r)\subset\mathbb H^{n+m}(r)$ 
are totally umbilical when viewed as submanifolds of $\mathbb H^{n+m}(r)$, 
the immersion $j$ is totally umbilical and its mean curvature 
satisfies $H^2=-c$, where $c=-1/r^2$ is the curvature 
of the hyperbolic manifold. When $n=4$, these submanifolds 
are precisely the compact flat four-manifolds that arise as cusp 
cross-sections of finite-volume hyperbolic $(4+m)$-manifolds, 
and they are of the type described in part $(b)$ 
of Theorem~\ref{n=4}.
}
\end{example} 

\begin{example}[\bf{$(2,0)$-geodesic K\"ahler submanifolds}]\emph{
The standard minimal embedding $g\colon\CP^2_{4/3}\to\Sf^7$ 
of the complex projective plane $\CP^2_{4/3}$ with constant 
holomorphic curvature $4/3$ is a $(2,0)$-geodesic 
immersion and satisfies condition \eqref{1}. 
}

\emph{Moreover, let $\accentset{\sim}{M}=\Sf^7\times N^m$ 
be the Riemannian product of $\Sf^7$ with an arbitrary 
Riemannian manifold $N^m$. For any point 
$y_0\in N^m$, consider the embedding 
$$
f\colon \CP^2_{4/3}\to\accentset{\sim}{M},\quad f(x)=(g(x), y_0).
$$ 
It is immediate that $f$ is $(2,0)$-geodesic and satisfies condition 
\eqref{1}, so that it is as in part $(iii$-$a)$ of Theorem \ref{n=4}. 
}
\end{example}

\section{Appendix: Symmetric spaces}

It is well known that symmetric spaces of compact type have nonnegative isotropic curvature. Moreover, if the sectional curvature is non-constant, then at each point there exist orthonormal four-frames with vanishing complex sectional curvature.

Let $M$ be a symmetric space with $\dim M \ge 4$. 
For each point $p\in M$, denote by $\mathcal F_p$ the set of 
all orthonormal four-frames $F=\{e_i\}_{i=1}^4\subset T_pM$ 
with vanishing complex sectional curvature. Each four-frame 
$F\in\mathcal F_p$ spans a four-dimensional subspace 
$V_F=\spa\left\{e_i\right\}_{i=1}^4\subset T_pM$. 
The aim of this section is to provide a proof of the following result, 
which is used in the proof of the main theorems.

\begin{proposition}\label{fillup}
Let $M$ be a compact, simply connected, irreducible 
symmetric space with non-constant sectional curvature and 
$\dim M\geq5$. Then, at each point $p\in M$, the orthogonal complements 
$V_F^\perp\subset T_pM$ of the subspaces $V_F$, spanned 
by all frames $F\in\mathcal F_p$, satisfy
$$
\sum_{F\in \mathcal F_p}V_F^\perp= T_pM.
$$
\end{proposition}

To prove this, we recall some basic facts about the isotropy representation 
of symmetric spaces.

\subsection{The isotropy representation}
Let $G$ be a Lie group with Lie algebra $\mathfrak g$.
For each $g \in G$ consider the conjugation map $
C_g\colon G\to G$ defined by 
$$
C_g(x) = gxg^{-1},\quad x\in G.
$$
This is a Lie group isomorphism whose 
differential 
$$
\mathrm{Ad}(g):=d(C_g)_e:
\mathfrak g \to \mathfrak g.
$$
at the identity $e\in G$ is a Lie algebra automorphism, 
that is 
$$
\mathrm{Ad}(g)[X,Y]
=\left[\mathrm{Ad}(g)X,\mathrm{Ad}(g)Y\right]
\quad\text{for all }X,Y \in \mathfrak g\:\text{ and } g\in G.
$$

Since $C_{gh}=C_g\circ C_h$, it follows that 
$\mathrm{Ad}(gh)=\mathrm{Ad}(g)\circ\mathrm{Ad}(h)$ and thus
$
\mathrm{Ad}\colon G\to GL(\mathfrak g)
$ 
is a Lie group homomorphism (see Chapter 3 in \cite{CE} or \cite{Hel}). 
Hence we obtain a representation
$$
\mathrm{Ad}\colon G \to \mathrm{Aut}(\mathfrak g)
$$
called the \emph{adjoint representation} of $G$. 
\vspace{0.5ex}

Let $M$ be a symmetric space. The isometry group 
$G=I(M)$ acts transitively on $M$. Fixing a 
point $p\in M$, the isotropy group $K$ is the 
stabilizer of $p$, that is $K=G_p=\{g\in G: g(p)=p\}$.
Then we may write $M=G/K$. 

Let $\mathfrak g = \mathfrak k \oplus \mathfrak p$
be the Cartan decomposition, where $\mathfrak g$ and 
$\mathfrak k$ denote the Lie algebras of $G$ and $K$, 
respectively. At a point $p=eK\in M$, we identify
$\mathfrak p \cong T_pM$. 
A fundamental property of symmetric spaces is 
$\mathrm{Ad}(K)\mathfrak p \subset \mathfrak p$. 
In particular, $[\mathfrak k,\mathfrak p]\subset \mathfrak p$.

The differential of each $k \in K$ at the base point 
$p=eK$ is an orthogonal transformation of $T_pM$ 
(see \cite{E}) and satisfies 
$dk_p=\mathrm{Ad}(k)\big|_{\mathfrak p}$. Therefore, 
the restriction
$$
\mathrm{Ad}(k)\big|_{\mathfrak p}
:
\mathfrak p \to \mathfrak p,\quad k\in K, 
$$
defines the \emph{isotropy representation} 
(see \cite{Besse} or \cite{E})
$$
\mathrm{Ad}|_K\colon K \to O(\mathfrak p).
$$
Thus the isotropy action on $T_pM$
is precisely the restriction of the adjoint representation.

For symmetric spaces, under the identification $\mathfrak p \cong T_pM$, 
the curvature tensor is given by 
$$
R(X,Y)Z = -[[X,Y],Z],
\qquad X,Y,Z \in \mathfrak p.
$$
Since $\mathrm{Ad}(k)$ preserves both the Lie brackets and the Cartan decomposition, 
we obtain 
$$
R\big(\mathrm{Ad}(k)X,\mathrm{Ad}(k)Y\big)\mathrm{Ad}(k)Z
=\mathrm{Ad}(k)R(X,Y)Z.
$$
Hence the curvature tensor is $K$-invariant.

It is well known (see Corollary 3 in \cite{E}, or
Corollary~6.10 in \cite{Z}) that, if $M=G/K$ is 
irreducible, then the isotropy representation
of $K$ on $\mathfrak p$ is irreducible.
In particular, if the isotropy representation of $K$ on 
$\mathfrak p$ is irreducible and $U\subset\mathfrak p$ 
is $K$-invariant a subspace, then either $U=\left\{0\right\}$ or 
$U=\mathfrak p$. 
\vspace{1ex}

\noindent\emph{Proof of Proposition \ref{fillup}:} 
Define
$$
\mathcal W_p:=\sum_{F\in\mathcal F_p} 
V_F^\perp \subset T_pM,\quad\mathcal V_p:
=\bigcap_{F\in\mathcal F_p} V_F.
$$
Taking orthogonal complements, we obtain
$\mathcal W_p^\perp=\mathcal V_p$. Hence, to prove 
that $\mathcal W_p=T_pM$, it suffices to show that 
$\mathcal V_p=\{0\}$. 

Let $\mathfrak g = \mathfrak k \oplus \mathfrak p$
be the Cartan decomposition, where
$\mathfrak g$ and $\mathfrak k$ 
denote the Lie algebras of $G$ and $K$, respectively.
Fix a point $p=eK\in M$ and identify $T_pM \simeq \mathfrak p$.
The isotropy group $K$ acts on $\mathfrak p \simeq T_p M$ via 
the isotropy representation
$$
k \cdot X = \mathrm{Ad}(k) X, \quad X \in \mathfrak p.
$$
Since $\mathrm{Ad}(k)$ preserves both the Lie brackets 
and the Cartan decomposition, the $(0,4)$-curvature tensor 
is $K$-invariant, that is
$$
R(kX,kY,kZ,kW) = R(X,Y,Z,W)
\quad\text{for all }k\in K\text{ and } X,Y,Z,W\in\mathfrak p.
$$
It follows that if $F \in \mathcal F_p$, then $kF \in \mathcal F_p$ 
for all $k\in K$. Therefore, 
$$
k(\mathcal V_p) 
= k\big(\bigcap_{F\in \mathcal F_p} V_F\big) 
= \bigcap_{F\in \mathcal F_p} k(V_F) 
= \bigcap_{F\in \mathcal F_p} V_{kF} 
= \mathcal V_p.
$$
Consequently, the subspace $\mathcal V_p$ is $K$-invariant. 
Since $M$ is irreducible, the isotropy representation of $K$ 
on $\mathfrak p$ is irreducible. Hence, any $K$-invariant 
subspace of $\mathfrak p$ is either $\{0\}$ or $\mathfrak p$. 

Compact symmetric space of non-constant sectional curvature 
admit four-frames at every point with vanishing isotropic 
curvature. In particular, $\mathcal F_p\neq\emptyset$. By 
definition, each $V_F$ is four-dimensional, and therefore 
$\dim \mathcal V_p \le \dim V_F = 4$. 
Since $\dim M = \dim \mathfrak p \ge 5$, the case 
$\mathcal V_p = \mathfrak p$ is impossible. 
Hence, $\mathcal V_p=\{0\}$. 
\qed

\noindent Theodoros Vlachos\\
University of Ioannina \\
Department of Mathematics\\
45110 Ioannina -- Greece\\
e-mail: tvlachos@uoi.gr

\end{document}